\documentclass[11pt,letterpaper]{amsart}
\usepackage{amssymb,amsmath,amsthm,amsfonts,mathrsfs,hyperref}
\usepackage{times}
\usepackage{enumerate}
\usepackage{cite,titletoc}
\usepackage{color}
\usepackage[toc,page,title,titletoc,header]{appendix}
\usepackage{multirow}

\def\cA{\mathcal A}
\def\cB{\mathcal B}

\def\cK{\mathcal K}

\def\cO{\mathcal O}

\def\cW{\mathcal W}

\def\sF{\mathscr F}
\def\N{\mathop{\mathbb N\kern 0pt}\nolimits}
\def\Z{\mathop{\mathbb Z\kern 0pt}\nolimits}
\def\Q{\mathop{\mathbb Q\kern 0pt}\nolimits}
\def\R{\mathop{\mathbb R\kern 0pt}\nolimits}
\def\T{\mathop{\mathbb T\kern 0pt}\nolimits}
\def\SS{\mathop{\mathbb S\kern 0pt}\nolimits}

\def\ds{\displaystyle}
\def\f{\frac}
\def\al{\alpha}
\def\supp{\mathop{\rm supp}\nolimits}

\def\p{\partial}
\def\t{\tilde}

\def\ve{\varepsilon}

\def\dive{\operatorname{div}}

\def\supp{\operatorname{supp}}
\def\ls{\lesssim}
\def\gt{\gtrsim}

\newcommand{\w}[1]{\langle {#1} \rangle}

\hypersetup{colorlinks=true,linkcolor=blue,citecolor=red,urlcolor=cyan}

\theoremstyle{plain}
\newtheorem{theorem}{Theorem}[section]

\newtheorem{lemma}[theorem]{Lemma}

\newtheorem{remark}{Remark}[section]

\theoremstyle{definition}

\numberwithin{equation}{section}

\title[Wave equations in exterior domains]{Global smooth solutions of 2-D quadratic quasilinear wave equations with null conditions in exterior domains}

\author[Fei Hou]{Fei Hou}\address[Fei Hou]{School of Mathematics, Nanjing University, Nanjing, 210093, China}\email{fhou$@$nju.edu.cn}

\author[Huicheng Yin]{Huicheng Yin}\address[Huicheng Yin]{School of Mathematical Sciences and Mathematical Institute, Nanjing Normal University, Nanjing, 210023, China}\email{huicheng$@$nju.edu.cn,05407$@$njnu.edu.cn}

\author[Meng Yuan]{Meng Yuan}\address[Meng Yuan]{School of Computer Science and Artificial Intelligence, Aliyun School of Big Data, School of Software, Changzhou University, Changzhou, 213164, China}\email{ym$@$cczu.edu.cn}
\thanks{The first and second authors are supported by the National key research and development program of China (No.2024YFA1013301).
        In addition, the authors are supported by the NSFC (No.~12571237, No.~12331007).}

\begin{document}

\maketitle

\begin{abstract}
For 3-D quadratic quasilinear wave equations with or without null conditions in  exterior domains,
when the  compatible initial data and Dirichlet boundary values are given, the global existence or the maximal existence
time of small data smooth solutions have been  established in early references.
For the Cauchy problem of
2-D quadratic quasilinear wave equations with null conditions,
it has been shown that the small data smooth solutions exist globally.
However, for the corresponding 2-D initial boundary value problem in exterior domains, it is  still open whether the global solutions exist. In the present paper, we solve this open problem
through proving the global existence of small solutions in exterior domains.
Our main ingredients include:  deriving new precise pointwise estimates for
the initial boundary value problem of 2-D linear wave equations in exterior domains;
finding appropriate divergence structures of quasilinear wave equations under null conditions;
introducing a good unknown to eliminate the resulting  $Q_0$ type  nonlinearity, and establishing some
crucial pointwise spacetime decay estimates of solutions and their derivatives.
\vskip 0.2 true cm

\noindent


\vskip 0.2 true cm\noindent

\end{abstract}

\vskip 0.2 true cm

\addtocontents{toc}{\protect\thispagestyle{empty}}
\tableofcontents

\section{Introduction}
In this paper, we focus on the {\it initial boundary value problem} (abbreviated as IBVP) of 2-D quadratic 
quasilinear wave equations
in exterior domains with homogeneous Dirichlet boundary conditions
\begin{equation}\label{QWE}
\left\{
\begin{aligned}
&\Box u=Q(\p u,\p^2u),\hspace{2.6cm} (t,x)\in[0,+\infty)\times\cK,\\
&u(t,x)=0,\hspace{3.7cm} (t,x)\in[0,+\infty)\times\p\cK,\\
&(u,\p_tu)(0,x)=(\ve u_0,\ve u_1)(x),\qquad x\in\cK,
\end{aligned}
\right.
\end{equation}
where $x=(x_1,x_2)$, $\p=(\p_0,\p_1,\p_2)=(\p_t,\p_{x_1},\p_{x_2})$,
$\ds\Box=\p_t^2-\Delta$ with $\Delta=\p_1^2+\p_2^2$,
$\cK=\R^2\setminus\cO$, the obstacle $\cO\subset \R^2$ is compact and contains the origin,
the boundary $\p\cK=\p\cO$ is smooth, $\ve>0$ is small,
$(u_0,u_1)\in C^\infty(\cK)$ and $\supp(u_0,u_1)\subset\{x\in\cK:|x|\le M_0\}$ with
some fixed constant $M_0>1$.
In addition, the nonlinearity $Q(\p u,\p^2u)$ admits such a quadratic form
\begin{equation}\label{nonlinear}
Q(\p u,\p^2u)=\sum_{\alpha,\beta,\gamma=0}^2Q^{\alpha\beta\gamma}\p^2_{\alpha\beta}u\p_{\gamma}u,
\end{equation}
where $Q^{\alpha\beta\gamma}=Q^{\beta\alpha\gamma}$.
Meanwhile, the null condition holds:
\begin{equation}\label{null:condition}
\sum_{\alpha,\beta,\gamma=0}^2Q^{\alpha\beta\gamma}\xi_{\alpha}\xi_{\beta}\xi_{\gamma}\equiv0\quad
\text{for any $(\xi_0,\xi_1,\xi_2)\in\{\pm1\}\times\SS$}.
\end{equation}
It is pointed out that the compatibility conditions of $(u_0,u_1)$ and $u|_{(0,+\infty)\times\p\cK}=0$ are necessary for finding smooth solutions of \eqref{QWE}.
To illustrate the compatibility conditions explicitly, setting $J_ku=\{\p_x^a u:0\le|a|\le k\}$ and $\p_t^ku(0,x)=F_k(J_ku_0,J_{k-1}u_1)$ ($0\le k\le 2N$), where $F_k$ depends on $Q(\p u, \p^2u)$,
$J_ku_0$ and $J_{k-1}u_1$.
For the initial boundary values of \eqref{QWE}, the compatibility conditions up to $(2N+1)-$order
mean that all $F_k$ vanish on $\p\cK$ for $0\le k\le 2N$.

Our main result in the paper is
\begin{theorem}\label{thm1}
Suppose that the obstacle $\cO$ is star-shaped and \eqref{null:condition} holds.
It is assumed that
\begin{equation}\label{initial:data}
\|u_0\|_{H^{2N+1}(\cK)}+\|u_1\|_{H^{2N}(\cK)}\le1 \quad\text{with $N\ge59$},
\end{equation}
and the compatibility conditions up to $(2N+1)-$order hold for
the initial boundary values of \eqref{QWE}.
Then there is a small constant $\ve_0>0$ such that when $\ve\le\ve_0$,
\eqref{QWE} has a global solution $u\in\bigcap\limits_{j=0}^{2N+1}C^{j}([0,\infty), H^{2N+1-j}(\cK))$.
Meanwhile, there is a constant $C>0$ such that
\addtocounter{equation}{1}
\begin{align}
\sum_{|a|\le N+1}|\p Z^au|&\le C\ve\w{x}^{-1/2}\w{t-|x|}^{-1},\tag{\theequation a}\label{thm1:decay:a}\\
\sum_{|a|\le N+1}\sum_{i=1}^2|\bar\p_iZ^au|&\le C\ve\w{x}^{-1/2}\w{t+|x|}^{0.001-1},\tag{\theequation b}\label{thm1:decay:b}\\
\sum_{|a|\le N}|Z^au|&\le C\ve\w{t+|x|}^{0.001-1/2}\w{t-|x|}^{0.001-1/2},\tag{\theequation c}\label{thm1:decay:c}
\end{align}
where $Z=\{\p,\Omega\}$ with $\Omega=x_1\p_2-x_2\p_1$, $\w{x}=\sqrt{1+|x|^2}$,  $\bar\p_i=\frac{x_i}{|x|}\p_t+\p_i$ ($i=1,2$)
are the good derivatives (tangent to the outgoing light cone $|x|=t$).
Furthermore, the following time decay estimate of local energy holds
\begin{equation}\label{thm1:decay:LE}
\sum_{|a|\le2N-41}\|\p^au\|_{L^2(\cK_R)}\le C_R\ve(1+t)^{-1},
\end{equation}
where $R>\max\{diam(\cO),1\}$ is a fixed constant, $\cK_R=\cK\cap\{x: |x|\le R\}$ and $C_R>0$ is a constant
depending on $R$.
\end{theorem}

\begin{remark}\label{rmk1-1}
It is known from \cite{Klainerman} or \cite[Page 79]{MetcalfeSogge05} that
under the null condition \eqref{null:condition}, the nonlinearity $Q(\p u,\p^2u)$ in \eqref{nonlinear}
admits such a form
\begin{equation}\label{nonlinear:nullform-1}
\begin{split}
Q(\p u,\p^2u)=\sum_{\mu=0}^2\bigl(&{\t C}^{0,\mu}Q_0(\p_{\mu}u,u)+\sum_{\alpha,\beta=0}^2{\t C}^{\alpha\beta,\mu}Q_{\alpha\beta}(\p_{\mu}u,u)\\
&+{\t C}^\mu\underbrace{\Box u\p_{\mu}u}_{\bf cubic~null~form~term~due~to~\eqref{QWE}-\eqref{null:condition}}\bigr),
\end{split}
\end{equation}
where ${\t C}^{0,\mu}$, ${\t C}^{\alpha\beta,\mu}$ and ${\t C}^\mu$ are constants,
the relativistic null form $Q_0(f,g)$
and the gauge-type null forms $Q_{\alpha\beta}(f,g)$ are represented by
\begin{equation}\label{Intro:null:form}
\begin{split}
Q_0(f,g)&:=\p_tf\p_tg-\sum_{j=1}^2\p_jf\p_jg,\\
Q_{\alpha\beta}(f,g)&:=\p_{\alpha}f\p_{\beta}g-\p_{\beta}f\p_{\alpha}g,\quad\alpha,\beta=0,1,2.
\end{split}
\end{equation}
On the other hand, due to $\square u\p_{\mu}u=Q_{0\mu}(\p_t u, u)+Q_0(\p_{\mu}u,u)-\ds\sum_{j=1}^3Q_{j\mu}(\p_ju,u)$,
then \eqref{nonlinear:nullform-1} can be rewritten as
\begin{equation}\label{nonlinear:nullform}
\begin{split}
Q(\p u,\p^2u)=\sum_{\mu=0}^2\bigl(C^{0,\mu}Q_0(\p_{\mu}u,u)+\sum_{\alpha,\beta=0}^2C^{\alpha\beta,\mu}Q_{\alpha\beta}(\p_{\mu}u,u)
\bigr),
\end{split}
\end{equation}
where $C^{0,\mu}$ and $C^{\alpha\beta,\mu}$ are constants.
\end{remark}

\begin{remark}\label{rmk1-2}
For the general 2-D quadratic quasilinear wave equations
in exterior domains
\begin{equation*}\label{QWE1}
\left\{
\begin{aligned}
&\Box v=Q(\p v, \p^2v),\hspace{2.5cm} (t,x)\in[0,+\infty)\times\cK,\\
&v(t,x)=0,\hspace{3.6cm} (t,x)\in[0,+\infty)\times\p\cK,\\
&(v,\p_tv)(0,x)=(\ve v_0,\ve v_1)(x),\qquad x\in\cK,
\end{aligned}
\right.
\end{equation*}
where $Q(\p v, \p^2v)=\ds\sum_{\alpha,\beta,\gamma=0}^2Q^{\alpha\beta\gamma}\p^2_{\alpha\beta}v\p_{\gamma}v
+\sum_{\alpha,\beta,\gamma,\delta=0}^2Q^{\alpha\beta\gamma\delta}\p^2_{\alpha\beta}v\p_{\gamma}v\p_{\delta}v$
with $\ds\sum_{\alpha,\beta,\gamma=0}^2Q^{\alpha\beta\gamma}\xi_{\alpha}\xi_{\beta}\xi_{\gamma}\equiv0$
and $\ds\sum_{\alpha,\beta,\gamma,\delta=0}^2Q^{\alpha\beta\gamma\delta}\xi_{\alpha}\xi_{\beta}\xi_{\gamma}\xi_{\delta}\equiv0$
for any $(\xi_0,\xi_1,\xi_2)\in\{\pm1\}\times\SS$, by the analogous proof on Theorem \ref{thm1} together with
the techniques in \cite{HouYinYuan25}, one can establish the global existence of small solution $v$.
This means that for the 2-D quadratic quasilinear wave equations with null conditions in exterior domains,
the  global existence of smooth small data solutions has been shown successfully.
\end{remark}

\begin{remark}\label{rmk1-3}
The local energy time decay estimate \eqref{thm1:decay:LE} seems to be optimal since
the related time decay rate is just $(1+t)^{-1}\ln^{-2}(2+t)$ for the corresponding linear problem
(see (17) on Page 324 of \cite{Kubo13} or \cite{Burg}, \cite{Vainberg75}).
Meanwhile, in view of \eqref{thm1:decay:LE} and the Sobolev embedding theorem near $\cK_R$,
the pointwise estimate \eqref{thm1:decay:a} is also optimal.
In addition, the pointwise estimate \eqref{thm1:decay:c} is extraordinarily close to
the decay rate $\w{t+|x|}^{-1/2}\w{t-|x|}^{-1/2}$ of the solution $v_{hom}$ to
the 2-D free wave equation $\square v_{hom}=0$ with $(v_{hom}, \p_tv_{hom})$ $=(v_0,v_1)$
and $(v_0,v_1)$ being compactly supported.
Indeed, the decay rate of $v_{hom}$ can be directly derived from the Poisson formula
\begin{equation*}
v_{hom}(t,x)=\frac{1}{2\pi}\p_t\biggl(\int_{|x-y|\le t}\frac{v_0(y)dy}{\sqrt{t^2-|x-y|^2}}\biggr)
+\frac{1}{2\pi}\int_{|x-y|\le t}\frac{v_1(y)dy}{\sqrt{t^2-|x-y|^2}}.
\end{equation*}
\end{remark}

\begin{remark}\label{rmk1-4}
For the 2-D quadratic quasilinear scalar wave equations
in exterior domains
\begin{equation}\label{QWE2}
\left\{
\begin{aligned}
&\Box v+\ds\sum_{\alpha,\beta,\gamma=0}^2
Q^{\alpha\beta\gamma}\p^2_{\alpha\beta}v\p_{\gamma}v=0,\hspace{1.5cm} (t,x)\in[0,+\infty)\times\cK,\\
&v(t,x)=0,\hspace{4.8cm} (t,x)\in[0,+\infty)\times\p\cK,\\
&(v,\p_tv)(0,x)=(\ve v_0,\ve v_1)(x),\qquad \qquad \quad x\in\cK,
\end{aligned}
\right.
\end{equation}
it is shown in \cite{LRX25} that the lifespan $T_{\ve}$ of smooth solution $v$ fulfills
\begin{equation}\label{Sonn-1}
T_\ve\ge\ds\frac{\kappa}{\ve^2|\ln \ve|^3},
\end{equation}
where  $\kappa>0$ is some suitable fixed constant.
This obviously brings logarithmic loss $|\ln \ve|^{-3}$ in \eqref{Sonn-1} comparing to the well-known optimal result $T_{\ve}\ge\f{C}{\ve^2}$
for the corresponding Cauchy problem (see \cite{Hormander97book} and \cite{Alinhac99}). By the divergence structure equality
$\p^2_{\alpha\beta}v\p_{\gamma}v=\f12\{\p_{\alpha}(\p_{\beta}v\p_{\gamma}v)-\p_{\gamma}(\p_{\alpha}v\p_{\beta}v)
+\p_{\beta}(\p_{\alpha}v\p_{\gamma}v)\}$, together with the method in this paper, we may
prove that the optimal lifespan $T_{\ve}\ge\f{C}{\ve^2}$ holds for the smooth solution $v$ of \eqref{QWE2}.
The detailed proof will be given in our forthcoming paper.
\end{remark}

\begin{remark}
Consider the initial boundary value problem of 2-D quadratic quasilinear wave equation
system in exterior domain
\begin{equation}\label{QWE3}
\left\{
\begin{aligned}
&\Box v^I=Q^I(\p v,\p^2v),\quad I=1,\cdots,M,\quad(t,x)\in[0,+\infty)\times\cK,\\
&v(t,x)=0,\hspace{4.6cm} (t,x)\in[0,+\infty)\times\p\cK,\\
&(v,\p_tv)(0,x)=(\ve v_0,\ve v_1)(x),\hspace{1.7cm} x\in\cK,
\end{aligned}
\right.
\end{equation}
where $v=(v^1, ..., v^M)$, $M\in\Bbb N$,
$Q^I(\p v,\p^2v)=\ds\sum_{J,K=1}^M\sum_{\alpha,\beta,\gamma=0}^2Q_{IJK}^{\alpha\beta\gamma}\p^2_{\alpha\beta}v^J\p_{\gamma}v^K$ ($I=1,\cdots, M$) with the symmetric coefficients $Q_{IJK}^{\alpha\beta\gamma}=Q_{IJK}^{\beta\alpha\gamma}=Q_{JIK}^{\alpha\beta\gamma}$,
and the null condition holds:
\begin{equation}\label{null:condition-0}
\sum_{\alpha,\beta,\gamma=0}^2Q_{IJK}^{\alpha\beta\gamma}\xi_{\alpha}\xi_{\beta}\xi_{\gamma}\equiv0,
\qquad (\xi_0,\xi_1,\xi_2)\in\{\pm1\}\times\SS, \quad I,J,K=1,\cdots, M.
\end{equation}
By \cite{Klainerman} or \cite[Page 79]{MetcalfeSogge05},
under the null condition \eqref{null:condition-0},  $Q^I(\p v,\p^2v)$ ($I=1,\cdots, M$)
admit such forms
\begin{equation}\label{nonlinear:nullform2}
Q^I(\p v,\p^2v)=\sum_{J,K=1}^M\sum_{\mu=0}^2
\Big(C_{IJK}^{0,\mu}Q_0(\p_{\mu}v^J,v^K)+\sum_{\alpha,\beta=0}^2C_{IJK}^{12,\alpha\beta}Q_{\alpha\beta}(\p_{\mu}v^J,v^K)\Big),
\end{equation}
where $C_{IJK}^{0,\mu}$ and $C_{IJK}^{12,\alpha\beta}$ are constants.
By the analogous method in the paper,
one can show the global existence of small solution $v$ to problem \eqref{QWE3}.

\end{remark}

\vskip 0.1 true cm

We now give systematic reviews on the small data solution problems of  quasilinear wave equations
with the Cauchy initial data or the initial boundary values in exterior domains. Let $v_1$ and $v_2$
solve the following $n-$dimensional quasilinear wave equations ($n\ge 2$),
respectively,
\begin{equation}\label{QWE-C}
\left\{
\begin{aligned}
&\Box v+\sum_{\alpha,\beta=0}^nQ^{\alpha\beta}(\p v)\p^2_{\alpha\beta}v=0,\hspace{1cm}  (t,x)\in [0,\infty)\times\Bbb R^n,\\
&(v,\p_tv)(0,x)=(\ve v_{1,0},\ve v_{1,1})(x), \qquad x\in\Bbb R^n
\end{aligned}
\right.
\end{equation}
and
\begin{equation}\label{QWE-B}
\left\{
\begin{aligned}
&\Box v+\sum_{\alpha,\beta=0}^nQ^{\alpha\beta}(\p v)\p^2_{\alpha\beta}v=0,\hspace{1cm} (t,x)\in[0,+\infty)\times\cK,\\
&v(t,x)=0,\hspace{4.1cm} (t,x)\in[0,+\infty)\times\p\cK,\\
&(v,\p_tv)(0,x)=(\ve v_{2,0},\ve v_{2,1})(x),\qquad x\in\cK,
\end{aligned}
\right.
\end{equation}
where $\ds Q^{\alpha\beta}(\p v)=\sum_{\gamma=0}^nQ^{\alpha\beta\gamma}\p_{\gamma}v
+\sum_{\gamma,\nu=0}^nQ^{\alpha\beta\gamma\nu}\p_{\gamma}v\p_{\nu}v$, $\ve>0$ is small, $(v_{1,0}, v_{1,1})\in C_0^{\infty}(\Bbb R^n)$ and $(v_{2,0}, v_{2,1})\in C_0^{\infty}(\cK)$.
Denote the lifespan of smooth solutions to problems \eqref{QWE-C} and \eqref{QWE-B} by $T_{1,\ve}$ and $T_{2,\ve}$, respectively.
So far there have been a lot of basic results as follows:

\vskip 0.1 true cm

${\bf\bullet}$ When $n\ge 4$, $T_{1,\ve}=\infty$ and $T_{2,\ve}=\infty$ hold.
One can see  \cite{Klainerman80}, \cite{KP83}, \cite{Li-T} or the Chapter VI of \cite{Hormander97book} for \eqref{QWE-C} and
\cite{Shib-Tsuti86}, \cite{Chen89}, \cite{Hayashi95} for \eqref{QWE-B}, respectively.

\vskip 0.1 true cm

${\bf\bullet}$ When $n=3$, the sharp bounds $T_{1,\ve}\ge e^{\f{C}{\ve}}$
(\cite{JohnKlainerman84}, \cite{Alinhac01b} and \cite{Ding-1}) and $T_{2,\ve}\ge e^{\f{C}{\ve}}$ (\cite{KSS04}) hold.
Especially, when the null condition is imposed, then
$T_{1,\ve}=\infty$ (\cite{Christodoulou86}, \cite{Klainerman}) and $T_{2,\ve}=\infty$
(\cite{Godin95}, \cite{KSS00}, \cite{MetcalfeSogge05}).

\vskip 0.1 true cm

${\bf\bullet}$ When $n=2$ and $Q^{\alpha\beta\gamma}=0$ for all $0\le \alpha,\beta,\gamma\le2$,
the sharp bounds $T_{1,\ve}\ge e^{\f{C}{\ve^2}}$
and $T_{2,\ve}\ge e^{\f{C}{\ve^2}}$ hold. Especially, when $Q^{\alpha\beta\gamma}=0$ for all $0\le \alpha,\beta,\gamma\le2$
and the null condition $\ds\sum_{\alpha,\beta,\gamma,\nu=0}^2
Q^{\alpha\beta\gamma\nu}\xi_{\alpha}\xi_{\beta}\xi_{\gamma}\xi_{\nu}\equiv 0$ for
any $(\xi_0,\xi_1,\xi_2)\in\{\pm1\}\times\SS$ are imposed, there hold  $T_{1,\ve}=\infty$ and $T_{2,\ve}=\infty$.
One can be referred to \cite{Kovalyov87}, \cite{Godin93}, \cite{Hoshiga95}, \cite{Katayama95} for
$T_{1,\ve}$ and \cite{KKL13}, \cite{Kubo13}, \cite{Kubo14}, \cite{HouYinYuan25} for $T_{2,\ve}$, respectively.

\vskip 0.1 true cm

${\bf\bullet}$ When $n=2$, $Q^{\alpha\beta\gamma}\not=0$ for some $0\le \alpha,\beta,\gamma\le2$,
it is known that the sharp bound $T_{1,\ve}\ge\f{C}{\ve^2}$ holds if the null condition is violated
(i.e., $\ds\sum_{\alpha,\beta,\gamma=0}^2Q^{\alpha\beta\gamma}\xi_{\alpha}\xi_{\beta}\xi_{\gamma}\not\equiv 0$ for
$(\xi_0,\xi_1,\xi_2)\in\{\pm1\}\times\SS$), otherwise, $T_{1,\ve}=\infty$ if the null conditions hold
(i.e., $\ds\sum_{\alpha,\beta,\gamma=0}^2Q^{\alpha\beta\gamma}\xi_{\alpha}\xi_{\beta}\xi_{\gamma}\equiv 0$
and $\ds\sum_{\alpha,\beta,\gamma,\nu=0}^2
Q^{\alpha\beta\gamma\nu}\xi_{\alpha}\xi_{\beta}\xi_{\gamma}\xi_{\nu}\equiv 0$ for any
$(\xi_0,\xi_1,\xi_2)\in\{\pm1\}\times\SS$), see \cite{Alinhac99}-\cite{Alinhac01b} and \cite{Hoshiga95}.
However, it is still open for $T_{2,\ve}=\infty$ when the null conditions hold and
for $T_{2,\ve}\ge\f{C}{\ve^2}$ when the null conditions are not fulfilled.
\vskip 0.2 true cm
From the above, the main results on the lifespan $T_{1,\ve}$ and $T_{2,\ve}$ can be described roughly as follows

\vskip 0.1 true cm

\begin{table}[!h]
\renewcommand\arraystretch{1.3}
\begin{tabular}{|c|c|c|c|}
\hline
 & Quadratic nonlinearity & Cauchy problem & Exterior domain problem\\
\hline
$n\ge4$ &  & $\ds T_{1,\ve}=\infty$ & $\ds T_{2,\ve}=\infty$ \\
\hline
\multirow{2}{*}{$n=3$} & Null condition & $\ds T_{1,\ve}=\infty$ & $\ds T_{2,\ve}=\infty$ \\
\cline{2-4}
 & Without null condition & $\ds T_{1,\ve}\ge e^{C/\ve}$ & $\ds T_{2,\ve}\ge e^{C/\ve}$ \\
\hline
\multirow{2}{*}{$n=2$} & Null condition & $\ds T_{1,\ve}=\infty$ & {\color{red}$\ds T_{2,\ve}=\infty$~Open} \\
\cline{2-4}
 & Without null condition & $\ds T_{1,\ve}\ge C/\ve^2$ & {\color{red}$\ds T_{2,\ve}\ge C/\ve^2$~Open} \\
\hline
\end{tabular}
\end{table}

\vskip 0.2 true cm
In order to solve the 2-D problem \eqref{QWE} globally, the following essential difficulties will be met and overcome.

\vskip 0.4 true cm

{\bf $\bullet$ Lack of the KSS type estimates and lower decay rates for 2-D linear wave equations}

\vskip 0.3 true cm

Note that for the 3-D linear wave equation, such KSS type estimates (see \cite[Prop 2.1]{KSS02jam} or (1.7)-(1.8)
on Page 190 of \cite{MetcalfeSogge06})
\begin{equation}\label{KSSestimate-1}
\begin{split}
&\quad(\ln(2+t))^{-1/2}\Big(\|\w{x}^{-1/2}\p v\|_{L_t^2L_x^2([0,t]\times\R^3)}+\|\w{x}^{-3/2}v\|_{L_t^2L_x^2([0,t]\times\R^3)}\Big)\\
&\ls\|\p v(0,x)\|_{L^2(\R^3)}+\int_0^t\|\Box v(s,x)\|_{L^2(\R^3)}ds
\end{split}
\end{equation}
or
\begin{equation}\label{KSSestimate-2}
\begin{split}
&\quad\|\w{x}^{-1/2-}\p v\|_{L_t^2L_x^2([0,t]\times\R^3)}+\|\w{x}^{-3/2-}v\|_{L_t^2L_x^2([0,t]\times\R^3)}\\
&\ls\|\p v(0,x)\|_{L^2(\R^3)}+\int_0^t\|\Box v(s,x)\|_{L^2(\R^3)}ds
\end{split}
\end{equation}
play key roles in establishing the long time existence of small solution $v_2$ to 3-D problem \eqref{QWE-B}.
However, the proof of the KSS estimate in \cite{KSS02jam} heavily relies on the strong Huygens principle
for 3-D space, which is not suitable for 2-D case.
On the other hand, the multiplier method used in \cite{MetcalfeSogge06} and \cite{Sterbenz05} for the initial
boundary value problem of 3-D nonlinear wave equations fails for \eqref{QWE-B} with $n=2$ (see \cite[Page 485]{Hepd-Metc23}).

On the other hand, for the solution $w$ of the 2-D linear wave equation
\begin{equation}\label{Intro-wave-compact}
\left\{
\begin{aligned}
&\Box w=g(t,x),\hspace{2.2cm} (t,x)\in[0,\infty)\times\R^2,\\
&(w,\p_tw)(0,x)=(0,0),\qquad x\in\R^2,
\end{aligned}
\right.
\end{equation}
where $\supp_x g(t,x)\subset\{x: 1\le|x|\le a\}$ and the constant $a>1$,
it follows from (4.2) with $m=0$ in \cite{Kubo14} and (57) with $\eta=0,\rho=1$  in \cite{Kubo13} that
\begin{equation}\label{Kubo13compact1}
|w|\le C\sup_{(s,y)\in[0,t]\times\R^2}\w{s}^{1/2}|g(s,y)|
\end{equation}
and
\begin{equation}\label{Kubo13compact2}
\begin{split}
&|\p w|\le \f{C\ln(2+t+|x|)}{\w{x}^{1/2}\w{t-|x|}+\w{t+|x|}^{1/2}\w{t-|x|}^{1/2}}\sum_{|b|\le1}\sup_{(s,y)\in[0,t]\times\R^2}\w{s}|\p^bg(s,y)|.
\end{split}
\end{equation}
Due to the lack of time decay of $w$ in \eqref{Kubo13compact1} and the appearance of the large factor $\ln(2+t+|x|)$
in \eqref{Kubo13compact2}, the method in \cite{Kubo13} can not be applied to global solution problem \eqref{QWE-B} with $n=2$.
This is specially emphasized (see page 320 of \cite{Kubo13})
for 2-D problem that ``Unfortunately, we are not able to handle the quadratic nonlinearity,
because we can only show that the solution itself grows at most logarithmic order".
In fact, if one uses the analogous pointwise estimate \eqref{Kubo13compact2} for the solution $v_2$
of \eqref{QWE-B},
then by the standard energy method of wave equations, the following inequality of the
higher energy $E_{high}(t)=\ds\sum_{|\al|=0}^{N_0}\int_{\Bbb R^2}|\p\p^{\al}v_2|^2dx$ ($N_0\ge 2$) for \eqref{QWE-B} can be obtained
\begin{equation*}
\frac{d}{dt}E_{high}(t)\le C\ve(1+t)^{-\zeta}E_{high}(t),\quad \zeta\in(1/2,1).
\end{equation*}
This, together with the Gronwall's Lemma,  leads to
\begin{equation}\label{Intro:energy}
E_{high}(t)\le E_{high}(0)e^{C\ve(1+t)^{1-\zeta}}.
\end{equation}
Therefore, it is hard to control the uniform largeness of $E_{high}(t)$ and the smallness of solution $v_2$
due to the exponential growth order in time in \eqref{Intro:energy}.

Fortunately, by making use of the Littlewood-Paley decompositions and some properties of the curve measure
on the unit circle $\mathbb{S}$, we can improve the spacetime decay rates in \eqref{Kubo13compact1}
-\eqref{Kubo13compact2} essentially and obtain  the following more precise pointwise estimates
(see Lemmas \ref{lem:impr:pw}-\ref{lem:impr:dpw} below) for any $\mu\in(0,1/2]$
\begin{equation}\label{HYY24compact1}
|w|
\le \f{C\ln^2(2+t+|x|)}{\w{t+|x|}^{1/2}\w{t-|x|}^{\mu}}\sup_{(s,y)\in[0,t]\times\R^2}\w{s}^{1/2+\mu}|(1-\Delta)g(s,y)|\quad
\end{equation}
and
\begin{equation}\label{HYY24compact2}
|\p w|
\le \f{C}{\w{x}^{1/2}\w{t-|x|}}\sup_{(s,y)\in[0,t]\times\R^2}\w{s}|(1-\Delta)^3g(s,y)|.
\end{equation}
This implies from \eqref{HYY24compact2} that the large factor $\ln(2+t+|x|)$
in \eqref{Kubo13compact2} has been removed through adding the orders of derivatives
for $g$.

\vskip 0.4 true cm

{\bf $\bullet$ Finding suitable divergence structures of $Q(\p u,\p^2u)$
under null condition}

\vskip 0.4 true cm

For the solution $w$ of problem \eqref{Intro-wave-compact}, it follows from \eqref{HYY24compact1} and \eqref{HYY24compact2}
that $\p w$ has a better space-time decay rate than that of $w$. In addition, when the nonlinearity
$Q(\p u, \p^2u)$ in \eqref{QWE} satisfies the null condition,
$Q(\p u,\p^2u)=
\ds\sum_{\mu=0}^2\bigl(C^{0,\mu}Q_0(\p_{\mu}u,u)+\sum_{\alpha,\beta=0}^2C^{\alpha\beta,\mu}Q_{\alpha\beta}(\p_{\mu}u,u)\bigr)$  holds in terms of \eqref{nonlinear:nullform}.
Note that by \cite[Lemma 6.6.4]{Hormander97book} or \cite[Lemma 3.3 of Chatper II]{Sogge95book}, the null forms
admit the following estimate
\begin{equation}\label{null:decay}
|Q_0(f,g)|+|Q_{\alpha,\beta}(f,g)|\le C(1+t+|x|)^{-1}\sum_{\substack{\Gamma\in\{\p,\Omega,t\p_t+r\p_r,\\x_i\p_t+t\p_i,i=1,2\}}}(|\Gamma f||\p g|+|\p f||\Gamma g|),
\end{equation}
which yields an additional decay factor $(1+t+|x|)^{-1}$ in \eqref{null:decay} than the ordinary quadratic forms.
However, the vector fields $t\p_t+r\p_r$ and $x_i\p_t+t\p_i (i=1,2)$ are hard to be used
for treating the 2-D initial boundary value problem, see \cite{KSS04}.
Therefore, \eqref{null:decay} can not be utilized directly.
Fortunately, the null form $Q_{12}$ has such a good estimate which only involves the derivatives in spacetime
and the space rotation vector field
\begin{equation}\label{Intro:Q12}
|Q_{12}(f,g)|\le C(1+|x|)^{-1}\sum_{Z\in\{\p,\Omega\}}(|Zf||\p g|+|\p f||Zg|).
\end{equation}
Furthermore, by the introduction of a new good unknown by us (see $V_a$
in \eqref{def:goodunknown} below), the $Q_0$ null form can be eliminated
and spacial localized,

Based on these important facts mentioned above, in order to solve the global solution problem \eqref{QWE}, a natural question arises:
Whether is it possible to write the null form
$Q_{0i}(f,g)$ as an efficient linear combinations of $Q_0(f,g)$, $Q_{12}(f,g)$, some terms with divergence form
and higher order error terms?

Indeed, we can find the following interesting algebraic equalities for $x\neq0$ and $i=1,2$, $\alpha=0,1,2$:
\begin{equation}\label{Q12}
\begin{split}
Q_{12}(f,g)&=\underbrace{\sum_{i=1}^2\p_i(\frac{x_if\Omega g}{|x|^2})-\p_\theta(\frac{f\p_rg}{|x|})}_{\bf divergence~term},\\
Q_{0i}(\p_{\alpha}f,g)&=\underbrace{C^{12}_{i\alpha}Q_{12}(\p f,g)}_{\bf divergence~term}+C(\omega)Q_0(\p f,g)
+\underbrace{C(\omega)(\Box f)\p g}_{\bf cubic~null~form~term}\\
&\quad+\underbrace{C(\omega)\frac{(\Omega\p f)\p g+(\Omega g)\p^2 f}{|x|}}_{\text{\bf quadratic~error~term~with~decay~factor~$|x|^{-1}$}},
\end{split}
\end{equation}
where $\p_\theta=\Omega=x_1\p_2-x_2\p_1$ and $\p_r=\frac{x_1\p_1+x_2\p_2}{r}$
with $r=|x|$, $C^{12}_{i\alpha}$ are constants, $C(\omega)$ is the generic smooth polynomial of $\omega=(\frac{x_1}{r},\frac{x_2}{r})$.
Moreover, substituting the form of $Q_{12}(f,g)$ into $Q_{0i}(\p_{\alpha}f,g)$
and taking a delicate computation yield such a better structure for $x\neq0$,
\begin{equation}\label{Q12-1}
\begin{split}
Q_{0i}(\p_{\alpha}f,g)&=C(\omega)Q_0(\p f,g)+\underbrace{C(\omega)(\Box f)\p g}_{\bf cubic~null~form~term}\\
&\quad +\underbrace{\frac{C(\omega)g\Omega\p f}{|x|^2}}_{\text{\bf quadratic~error~term~with~better~decay~factor~$|x|^{-2}$}}\\
&\quad +\sum_{\mu=0}^2\underbrace{\p_{\mu}(\frac{C(\omega)g\Omega\p f}{|x|})
+\p_\theta(\frac{C(\omega)g\p^2f}{|x|})}_{\bf divergence~term}.
\end{split}
\end{equation}

Based on \eqref{Intro:Q12}, \eqref{Q12} and \eqref{Q12-1}, by involved analysis,
some required estimates for proving Theorem \ref{thm1} can be derived
in Section \ref{sect4}.

\vskip 0.4 true cm

{\bf $\bullet$ Weak decay of the local energy in 2-D case}

\vskip 0.3 true cm

Note that some local energy decay estimates for the initial boundary value problem of
linear wave equations have early been given in \cite{Morawetz75,Vainberg75} and the references therein.
It is shown that for the exterior domain problem of 3-D linear wave equation, the local energy of the solution
admits the exponential decay in time; for the 2-D case, the time decay rate of the local energy is only $(1+t)^{-1}$
while it is $(1+t)^{-1}\ln^{-2}(2+t)$ for the Dirichlet boundary value problem.
From this, near the boundary in 3-D case, the factor $t$ in the scaling vector field $t\p_t+r\p_r$
can be absorbed by the exponential decay of the local energy (see more details in the proof of Lemma 2.8 in \cite{MetcalfeSogge05});
while in 2-D case, the factor $t$ in $t\p_t+r\p_r$ is difficult to be treated.
This leads to that the crucial scaling vector field $t\p_t+r\p_r$ will be removed in the proof of Theorem \ref{thm1} by us.
For compensating this kind of loss, analogously to \eqref{thm1:decay:a}-\eqref{thm1:decay:c},
some new pointwise estimates  will be established by utilizing the local energy decay estimates
and deriving the related estimates near the boundary (see Lemma \ref{lem:ibvp:loc} below).
To this end, the local energy decay estimate \eqref{thm1:decay:LE} for the solution $u$ itself and its derivatives in time
is firstly obtained, then by the induction on the orders of the derivatives in space and by the elliptic estimates for the operator $\Delta=\p_t^2-\Box$, the estimates \eqref{thm1:decay:LE} will be obtained for all the required derivatives in space and time.

\vskip 0.1 true cm

Theorem \ref{thm1} will be proved by the continuous method.
At first, we derive the energy estimates for problem \eqref{QWE} (see \eqref{energy:B} and \eqref{loc:impr:A} below)
\begin{equation}\label{intro:energy}
\begin{split}
&\sum_{|a|\le2N-10}\sum_{Z\in\{\p,\Omega\}}\|\p Z^au\|_{L^2(\cK)}\le C\ve(1+t)^{0+},\\
&\sum_{|a|\le2N-17}\|\p^au\|_{L^2(\cK_R)}\le C\ve(1+t)^{-1+}.
\end{split}
\end{equation}
Together with  the Klainerman-Sobolev inequality, it is obtained from the first inequality of \eqref{intro:energy} that
\begin{equation}\label{intro:L2Linf}
\sum_{|a|\le2N-12}|\p Z^au|\le C\ve\w{x}^{-1/2}(1+t)^{0+}.
\end{equation}
In terms of \eqref{intro:L2Linf},
we start to prove the estimates \eqref{thm1:decay:a}-\eqref{thm1:decay:c}.
As in \cite{Kubo13} and \cite{Kubo14},
the weighted $L^\infty-L^\infty$ estimates for the 2-D linear wave equation
will be applied.
It is emphasized that the space-time pointwise decay factor $\w{t-|x|}^{-1}$ of \eqref{thm1:decay:a}
will play a key role in the establishment of the related energy estimates.
Meanwhile, the precise spacetime pointwise estimate \eqref{thm1:decay:c} with the additional decay
factor $\w{t-|x|}^{0.001-1/2}$ is also  crucial in the proof of Theorem \ref{thm1} (see the
detailed explanations in Remarks \ref{rmk4-1} and \ref{rmk5-1}).

The process of the pointwise estimates on the solution $u$ of problem \eqref{QWE}
can be described in turn (see Section \ref{sect6} below)
\begin{equation*}
\begin{split}
\sum_{|a|\le2N-18}|Z^au| &\le C\ve\w{t+|x|}^{0+},\\
\sum_{|a|\le2N-20}|\bar\p Z^au| &\le C\ve\w{x}^{-1+}\quad\text{for $|x|\ge1+t/2$}, \\
\end{split}
\end{equation*}

\begin{equation*}
\begin{split}
\sum_{|a|\le2N-24}|Z^au| &\le C\ve\w{t+|x|}^{-1/2+},\\
\sum_{|a|\le2N-26}|\bar\p Z^au| &\le C\ve\w{x}^{-3/2+}\quad\text{for $|x|\ge1+t/2$},\\
\sum_{|a|\le2N-24}|\p Z^au| &\le C\ve\w{x}^{-1/2}\w{t-|x|}^{-1/2}\w{t}^{0+},\\
\sum_{|a|\le2N-37}|\p Z^au| &\le C\ve\w{x}^{-1/2}\w{t-|x|}^{-1}\w{t}^{0+},\\
\sum_{|a|\le2N-41}\|\p^au\|_{L^2(\cK_R)} &\le C\ve(1+t)^{-1},\\
\sum_{|a|\le2N-43}|Z^au| &\le C\ve\w{t+|x|}^{-1/2+}\w{t-|x|}^{-0.4},\\
\sum_{|a|\le2N-53}|\p Z^au| &\le C\ve\w{x}^{-1/2}\w{t-|x|}^{-1},\\
\sum_{|a|\le2N-59}|Z^au| &\le C\ve\w{t+|x|}^{-1/2+}\w{t-|x|}^{-1/2+}.
\end{split}
\end{equation*}
Note that in the previous process, the decay rates of the pointwise estimates
of $Z^au,\p Z^au$ are gradually improved when the number $|a|$ of the vector fields decreases.
Based on this, the estimates \eqref{thm1:decay:a}-\eqref{thm1:decay:c} can be obtained and
further the proof of Theorem \ref{thm1} is completed for $N\ge 59$.

\vskip 0.2 true cm

The paper is organized as follows.
In Section \ref{sect2}, for the quadratic quasilinear wave equations, the related null condition structures
and the divergence structures are studied, meanwhile, a new good unknown $V_a$
(see \eqref{def:goodunknown}) is introduced.
Several basic lemmas and
some important results on the pointwise estimates
of solutions to the 2-D linear wave equations are given in Section \ref{sect3}.
In Section \ref{sect-app-B}, the improved pointwise estimates in Section \ref{sect3} are proved.
In Section \ref{sect4}, some pointwise estimates of the initial boundary value problem
with the divergence forms are derived.
In Section \ref{sect5}, the required energy estimates are achieved by the ghost weight technique
in \cite{Alinhac01a} and the elliptic estimates.
In Section \ref{sect6}, the crucial pointwise estimates are established step by step and then Theorem \ref{thm1} is shown.

\vskip 1 true cm

\noindent \textbf{Notations:}
\begin{itemize}
  \item $\cK:=\R^2\setminus\cO$, $\cK_R:=\cK\cap\{x: |x|\le R\}$, $R>1$ is a fixed constant which may be changed in different places.
  \item For technical reasons and without loss of generality, one can assume $\p\cK\subset\{x:c_0<|x|<1/2\}$, where $c_0$ is a positive constant.
  \item The cutoff function $\chi_{[a,b]}(s)\in C^\infty(\R)$ with $a,b\in\R$, $0<a<b$, $0\le\chi_{[a,b]}(s)\le1$ and
  \begin{equation*}
    \chi_{[a,b]}(s)=\left\{
    \begin{aligned}
    0,\qquad s\le a,\\
    1,\qquad s\ge b.
    \end{aligned}
    \right.
  \end{equation*}
  \item $\w{x}:=\sqrt{1+|x|^2}$.
  \item $\N_0:=\{0,1,2,\cdots\}$ and $\N:=\{1,2,\cdots\}$.
  \item $\p_0:=\p_t$, $\p_1:=\p_{x_1}$, $\p_2:=\p_{x_2}$, $\p_x:=\nabla_x=(\p_1,\p_2)$, $\p:=(\p_t,\p_1,\p_2)$.
  \item For $|x|>0$, define $\bar\p_i:=\p_i+\frac{x_i}{|x|}\p_t$ ($i=1,2$) and $\bar\p=(\bar\p_1,\bar\p_2)$.
  \item $\Omega_{ij}:=x_i\p_j-x_j\p_i$ with $i,j=1,2$, $\Omega:=\Omega_{12}$, $Z=\{Z_1,Z_2,Z_3,Z_4\}:=\{\p_t,\p_1,\p_2,\Omega\}$,
  $\tilde Z=\{\tilde Z_1,\tilde Z_2,\tilde Z_3,\tilde Z_4\}$ with $\tilde Z_1:=\p_t$ and $\tilde Z_i:=\chi_{[1/2,1]}(x)Z_i$ for $i=2,3,4$.
  \item The polar coordinates: $r=\sqrt{x_1^2+x_2^2}$, $\theta=\arctan\frac{x_2}{x_1}$, $\p_\theta=\Omega$.
  \item $\omega=(\omega_1,\omega_2)=(\frac{x_1}{r},\frac{x_2}{r})$.
  \item $\p_x^a:=\p_1^{a_1}\p_2^{a_2}$ for $a\in\N_0^2$, $\p^a:=\p_t^{a_1}\p_1^{a_2}\p_2^{a_3}$ for $a\in\N_0^3$
      and $Z^a:=Z_1^{a_1}Z_2^{a_2}Z_2^{a_3}Z_4^{a_4}$ for $a\in\N_0^4$.
  \item The commutator $[X,Y]:=XY-YX$.
  \item For $f,g\ge0$, $f\ls g$ or $g\gt f$ denotes $f\le Cg$ for a generic constant $C>0$ independent of $\ve$;
    $f\approx g$ means $f\ls g$ and $g\ls f$.
  \item Denote $\|u(t)\|=\|u(t,\cdot)\|$ with norm $\|\cdot\|=\|\cdot\|_{L^2(\cK)},\|\cdot\|_{L^2(\cK_R)},\|\cdot\|_{H^k(\cK)}$.
  \item $|\bar\p f|:=(|\bar\p_1f|^2+|\bar\p_2f|^2)^{1/2}$.
  \item $\cW_{\mu,\nu}(t,x)=\w{t+|x|}^\mu(\min\{\w{x},\w{t-|x|}\})^\nu$ for $\mu, \nu\in\Bbb R$.
  \item $\ds|Z^{\le j}f|:=\Big(\sum_{0\le|a|\le j}|Z^af|^2\Big)^\frac12$ and $Z^{\le1}f=(f,Zf)$.
\end{itemize}

\vskip 0.1 true cm

\section{Null condition, divergence structure and a new good unknown}\label{sect2}

\begin{lemma}\label{lem:eqn:high}
Let $u$ be the smooth solution of \eqref{QWE} with \eqref{nonlinear} and \eqref{null:condition}.
Then for any multi-index $a$, $Z^au$ satisfies
\begin{equation}\label{eqn:high}
\Box Z^au=\sum_{b+c\le a}\sum_{\alpha,\beta,\gamma=0}^2Q_{abc}^{\alpha\beta\gamma}\p^2_{\alpha\beta}Z^bu\p_{\gamma}Z^cu,
\end{equation}
where $Q_{abc}^{\alpha\beta\gamma}$ are constants, in particular $Q_{aa0}^{\alpha\beta\gamma}=Q^{\alpha\beta\gamma}$ holds.
Furthermore, for any $(\xi_0,\xi_1,\xi_2)\in\{\pm1\}\times\SS$, one has
\begin{equation}\label{null:high}
\sum_{\alpha,\beta,\gamma=0}^2Q_{abc}^{\alpha\beta\gamma}\xi_\alpha\xi_\beta\xi_\gamma\equiv0.
\end{equation}
In addition, there are constants $C_{abc}^{0,\alpha}$ and $C_{abc}^{\mu\nu,\alpha}$ such that
\begin{equation}\label{eqn:high'}
\Box Z^au=\sum_{b+c\le a}\sum_{\alpha=0}^2\Big\{C_{abc}^{0,\alpha}Q_0(\p_{\alpha}Z^bu,Z^cu)
+\sum_{0\le\mu<\nu\le2}C_{abc}^{\mu\nu,\alpha}Q_{\mu\nu}(\p_{\alpha}Z^bu,Z^cu)\Big\}.
\end{equation}

\end{lemma}
\begin{proof}
See Lemma 6.6.5 of \cite{Hormander97book} and Remark \ref{rmk1-1}.
\end{proof}

\begin{lemma}\label{lem:null:structure}
Suppose that the constants $N^{\alpha\beta\mu}$ satisfy
\begin{equation*}
\sum_{\alpha,\beta,\mu=0}^2N^{\alpha\beta\mu}\xi_\alpha\xi_\beta\xi_\mu\equiv0
\quad\text{for any $\xi=(\xi_0,\xi_1,\xi_2)\in\{\pm1\}\times\SS$}.
\end{equation*}
Then for any smooth functions $f,g,h$, it holds that
\begin{equation}\label{null:structure}
\begin{split}
\Big|\sum_{\alpha,\beta,\mu=0}^2N^{\alpha\beta\mu}
\p^2_{\alpha\beta}f\p_{\mu}g\Big|&\ls|\bar\p\p f||\p g|+|\p^2f||\bar\p g|,\\
\Big|\sum_{\alpha,\beta,\mu=0}^2N^{\alpha\beta\mu}
\p_{\alpha}f\p_{\beta}g\p_{\mu}h\Big|&\ls|\bar\p f||\p g||\p h|+|\p f||\bar\p g||\p h|+|\p f||\p g||\bar\p h|.
\end{split}
\end{equation}
In addition, one has
\begin{equation}\label{Q12:null:structure}
\begin{split}
Q_{12}(f,g)&=\frac{1}{r}(\p_rf\p_\theta g-\p_\theta f\p_rg)=\sum_{i=1,2}\p_i(\frac{x_if\Omega g}{|x|^2})
-\p_\theta(\frac{f\p_rg}{|x|}),\quad x\neq0,\\
|Q_{12}(f,g)|&\ls(1+|x|)^{-1}(|Zf||\p g|+|\p f||Zg|).
\end{split}
\end{equation}
\end{lemma}
\begin{proof}
Since the proof of \eqref{null:structure} is rather analogous to that in Section 9.1 of \cite{Alinhac:book} or \cite[Lemma 2.2]{HouYin20jde}, 
we omit it here.
In the polar coordinates $(x_1,x_2)=(r\cos\theta,r\sin\theta)$, \eqref{Q12:null:structure} can be directly verified
by the fact of $\p_{\theta}=x_1\p_2-x_2\p_1$.
\end{proof}

Next we give out the suitable structure of $Q_{0i}$ null forms, which will play an important role in the proof of Theorem \ref{thm1}.
\begin{lemma}[Structures of $Q_{0i}$ null forms]\label{lem:Q0i:structure}
For any functions $f,g$, $i=1,2$ and $\alpha=0,1,2$, one has that for $x\neq0$,
\begin{equation}\label{Q0i:structure1}
\begin{split}
Q_{0i}(\p_{\alpha}f,g)&=\sum_{\mu=0}^2\Big\{C_{i\alpha}^{12,\mu}Q_{12}(\p_{\mu}f,g)
+C_{i\alpha}^{0,\mu}(\omega)Q_0(\p_{\mu}f,g)+C_{i\alpha}^{\mu}(\omega)\Box f\p_{\mu}g\Big\}\\
&\quad+\sum_{\mu,\nu=0}^2\frac{C_{i\alpha}^{\mu\nu,1}(\omega)\Omega\p_{\mu}f\p_{\nu}g
+C_{i\alpha}^{\mu\nu,2}(\omega)\Omega g\p^2_{\mu\nu}f}{|x|},
\end{split}
\end{equation}
where $C_{i\alpha}^{12,\mu}$ are constants, $C_{i\alpha}^{0,\mu}(\omega)$,
$C_{i\alpha}^{\mu}(\omega)$, $C_{i\alpha}^{\mu\nu,1}(\omega)$ and $C_{i\alpha}^{\mu\nu,2}(\omega)$
are polynomials of $\omega$.
Furthermore, it holds that
\begin{equation}\label{Q0i:structure2}
\begin{split}
Q_{0i}(\p_{\alpha}f,g)&=\sum_{\mu=0}^2\Big\{C_{i\alpha}^{12,\mu}Q_{12}(\p_{\mu}f,g)
+C_{i\alpha}^{0,\mu}(\omega)Q_0(\p_{\mu}f,g)+C_{i\alpha}^{\mu}(\omega)\Box f\p_{\mu}g\Big\}\\
&\quad+\sum_{\mu=0}^2\p_{\mu}(\frac{C(\omega)g\Omega\p f}{|x|})+\p_\theta(\frac{C(\omega)g\p^2f}{|x|})
+\frac{C(\omega)g\Omega\p f}{|x|^2},
\end{split}
\end{equation}
where $C(\omega)$ stands for the generic polynomial of $\omega$.
\end{lemma}

\begin{proof}
At first, due to
\begin{equation}\label{Q0i:pf1}
\begin{split}
Q_{0i}(\p_tf,g)&=\p^2_tf\p_ig-\p^2_{0i}f\p_tg=\p^2_tf\p_ig-Q_0(\p_if,g)-\sum_{j=1,2}\p^2_{ij}f\p_jg\\
&=\p^2_tf\p_ig-Q_0(\p_if,g)-\sum_{j=1,2}Q_{ij}(\p_jf,g)-\Delta f\p_ig\\
&=\Box f\p_ig-Q_0(\p_if,g)-\sum_{j=1,2}Q_{ij}(\p_jf,g),
\end{split}
\end{equation}
then \eqref{Q0i:structure1} holds for $\alpha=0$.

Next we calculate $Q_{0i}(\p_jf,g)$ with $j=1,2$. Note that
\begin{equation}\label{Q0i:pf2}
\begin{split}
\p^2_{ij}f&=\sum_{k=1,2}\omega^2_k\p^2_{ij}f=\sum_{k=1,2}\frac{\omega_k\Omega_{ki}\p_jf+\omega_ix_k\p^2_{jk}f}{|x|}\\
&=\sum_{k=1,2}\frac{\omega_k\Omega_{ki}\p_jf+\omega_i\Omega_{kj}\p_kf}{|x|}+\omega_i\omega_j\Delta f.
\end{split}
\end{equation}
Then it follows from \eqref{Q0i:pf2} and direct computations that
\begin{equation}\label{Q0i:pf3}
\begin{split}
&Q_{0i}(\p_jf,g)=\p^2_{0j}f\p_ig-\p^2_{ij}f\p_tg\\
&=\p^2_{0j}f\p_ig-\omega_i\omega_j\Delta f\p_tg-\sum_{k=1,2}\p_tg\frac{\omega_k\Omega_{ki}\p_jf+\omega_i\Omega_{kj}\p_kf}{|x|}\\
&=\p^2_{0j}f\p_ig+\omega_i\omega_j\Box f\p_tg-\omega_i\omega_j\p^2_tf\p_tg-\sum_{k=1,2}\p_tg\frac{\omega_k\Omega_{ki}\p_jf+\omega_i\Omega_{kj}\p_kf}{|x|}\\
&=\omega_i\omega_j\Box f\p_tg-\omega_i\omega_jQ_0(\p_tf,g)-\sum_{k=1,2}\p_tg\frac{\omega_k\Omega_{ki}\p_jf+\omega_i\Omega_{kj}\p_kf}{|x|}\\
&\qquad+\p^2_{0j}f\p_ig-\sum_{k=1,2}\omega_i\omega_j\p^2_{0k}f\p_kg.
\end{split}
\end{equation}
The last line of \eqref{Q0i:pf3} can be calculated as follows
\begin{equation}\label{Q0i:pf4}
\begin{split}
\p^2_{0j}f\p_ig-\sum_{k=1,2}\omega_i\omega_j\p^2_{0k}f\p_kg
&=\sum_{k=1,2}(\omega_k^2\p^2_{0j}f\p_ig-\omega_i\omega_j\p^2_{0k}f\p_kg)\\
&=\sum_{k=1,2}\frac{\omega_k\p^2_{0j}f\Omega_{ki}g+\omega_i\Omega_{kj}\p_tf\p_kg}{|x|}.
\end{split}
\end{equation}
Substituting \eqref{Q0i:pf4} into \eqref{Q0i:pf3} derives
\begin{equation}\label{Q0i:pf5}
\begin{split}
Q_{0i}(\p_jf,g)&=\omega_i\omega_j\Box f\p_tg-\omega_i\omega_jQ_0(\p_tf,g) -\sum_{k=1,2}\p_tg\frac{\omega_k\Omega_{ki}\p_jf+\omega_i\Omega_{kj}\p_kf}{|x|}\\
&\quad +\sum_{k=1,2}\frac{\omega_k\p^2_{0j}f\Omega_{ki}g+\omega_i\Omega_{kj}\p_tf\p_kg}{|x|}.
\end{split}
\end{equation}
This, together with \eqref{Q0i:pf1}, yields the proof of \eqref{Q0i:structure1}.

At last, we rewrite the last two summations of \eqref{Q0i:pf5} into the divergence forms with null structures.
Due to
\begin{equation*}
\begin{split}
\sum_{k=1,2}[\omega_k,\Omega_{ki}]F&=\sum_{k=1,2}(-x_k\delta_{ik}+x_i)\frac{F}{|x|}=\omega_iF,\\
\sum_{k=1,2}[\Omega_{kj},\p_k]F&=\sum_{k=1,2}(-\p_j+\delta_{jk}\p_k)F=-\p_jF,
\end{split}
\end{equation*}
then one has
\begin{equation}\label{Q0i:pf6}
\begin{split}
&-\sum_{k=1,2}\p_tg\frac{\omega_k\Omega_{ki}\p_jf+\omega_i\Omega_{kj}\p_kf}{|x|}
+\sum_{k=1,2}\frac{\omega_k\p^2_{0j}f\Omega_{ki}g+\omega_i\Omega_{kj}\p_tf\p_kg}{|x|}\\
&=-\sum_{k=1,2}\p_t(g\frac{\omega_k\Omega_{ki}\p_jf+\omega_i\Omega_{kj}\p_kf}{|x|})
+\sum_{k=1,2}\frac{\omega_k}{|x|}(\p^2_{0j}f\Omega_{ki}g+\Omega_{ki}\p^2_{0j}fg)\\
&\quad+\sum_{k=1,2}\frac{\omega_i}{|x|}(\Omega_{kj}\p_tf\p_kg+\Omega_{kj}\p^2_{0k}fg)\\
&=-\sum_{k=1,2}\p_t(g\frac{\omega_k\Omega_{ki}\p_jf+\omega_i\Omega_{kj}\p_kf}{|x|})
+\sum_{k=1,2}\Omega_{ki}(\frac{\omega_k\p^2_{0j}fg}{|x|})+\frac{\omega_i\p^2_{0j}fg}{|x|}\\
&\quad+\sum_{k=1,2}\p_k(\frac{\omega_i\Omega_{kj}\p_tfg}{|x|})-\frac{\omega_i\p^2_{0j}fg}{|x|}
-\sum_{k=1,2}\p_k(\frac{\omega_i}{|x|})\Omega_{kj}\p_tfg.
\end{split}
\end{equation}
Therefore, by  \eqref{Q0i:pf5} and \eqref{Q0i:pf6},  we complete the proof of Lemma \ref{lem:Q0i:structure}.
\end{proof}

\begin{lemma}
Let $u$ be the smooth solution of \eqref{QWE} with \eqref{nonlinear} and \eqref{null:condition}.
For any multi-index $a$, one has
\begin{equation}\label{eqn:reduce}
\begin{split}
&\Box Z^au=\sum_{b+c\le a}\Big\{\sum_{\mu,\nu=0}^2|x|^{-1}(C(\omega)\p^2_{\mu\nu}Z^bu\Omega Z^cu
+C(\omega)\Omega\p_{\mu}Z^bu\p_{\nu}Z^cu)\\
&+\sum_{\mu=0}^2\Big[C_{abc}^{12,\mu}Q_{12}(\p_{\mu}Z^bu,Z^cu)
+C_{abc}^{0,\mu}(\omega)Q_0(\p_{\mu}Z^bu,Z^cu)+C(\omega)\Box Z^bu\p_{\mu}Z^cu\Big]\Big\},
\end{split}
\end{equation}
where $C_{abc}^{12,\mu}$ are constants, and $C_{abc}^{0,\mu}(\omega)$ are polynomials of $\omega$.
Introduce the good unknown
\begin{equation}\label{def:goodunknown}
V_a:=\tilde Z^au-\frac{1}{2}\sum_{b+c\le a}\sum_{\mu=0}^2C_{abc}^{0,\mu}(\omega)\chi_{[1/2,1]}(x)\p_{\mu}Z^bu Z^cu.
\end{equation}
Then we have $\ds V_a|_{\p\cK}=0$ and \eqref{eqn:reduce} is reduced to
\begin{equation}\label{eqn:good}
\begin{split}
\Box &V_a=\Box(\tilde Z^a-Z^a)u+\frac{1}{2}\sum_{b+c\le a}\sum_{\mu=0}^2
[\Delta,C_{abc}^{0,\mu}(\omega)\chi_{[1/2,1]}(x)](\p_{\mu}Z^buZ^cu)\\
&+\sum_{b+c\le a}\sum_{\mu=0}^2\Big\{C_{abc}^{12,\mu}Q_{12}(\p_{\mu}Z^bu,Z^cu)
+C(\omega)(1-\chi_{[1/2,1]}(x))Q_0(\p_{\mu}Z^bu,Z^cu)\\
&+C(\omega)\Box Z^bu\p_{\mu}Z^cu
+C(\omega)\chi_{[1/2,1]}(x)(\Box\p_{\mu}Z^buZ^cu+\p_{\mu}Z^bu\Box Z^cu)\Big\}\\
&+\sum_{b+c\le a}\sum_{\mu,\nu=0}^2|x|^{-1}(C(\omega)\p^2_{\mu\nu}Z^bu\Omega Z^cu
+C(\omega)\Omega\p_{\mu}Z^bu\p_{\nu}Z^cu),
\end{split}
\end{equation}
where the commutator on the first line of \eqref{eqn:good} admits the following estimate
\begin{equation}\label{comm:estimate}
|[\Delta,C_{abc}^{0,\mu}(\omega)\chi_{[1/2,1]}(x)](\p_{\mu}Z^buZ^cu)|
\ls|x|^{-2}|Z^{\le1}(\p_{\mu}Z^buZ^cu)|.
\end{equation}
On the other hand, we have a better form of \eqref{eqn:good} with some divergence terms
\begin{equation*}\label{eqn:good:div}
\begin{split}
&\Box V_a=\underbrace{\Box(\tilde Z^a-Z^a)u}_{\bf spacial~localization~term}+\frac{1}{2}\sum_{b+c\le a}\sum_{\mu=0}^2
\underbrace{[\Delta,C_{abc}^{0,\mu}(\omega)\chi_{[1/2,1]}(x)](\p_{\mu}Z^buZ^cu)}_{\text{\bf quadratic~error~term~with~better~decay~factor~$|x|^{-2}$}}\\
\end{split}
\end{equation*}

\begin{equation}\label{eqn:good:div}
\begin{split}
&\quad +\sum_{b+c\le a}\sum_{\mu=0}^2\Big\{\underbrace{C(\omega)(1-\chi_{[1/2,1]}(x))Q_0(\p_{\mu}Z^bu,Z^cu)}_{\bf spacial~localization~term}
+\underbrace{C(\omega)\Box Z^bu\p_{\mu}Z^cu}_{\bf cubic~null~form~term}\\
&\qquad+\underbrace{C(\omega)\chi_{[1/2,1]}(x)(\Box\p_{\mu}Z^buZ^cu+\p_{\mu}Z^bu\Box Z^cu)}_{\bf cubic~null~form~term}\Big\}\\
&\quad +\sum_{b+c\le a}\Big\{\underbrace{\p_\theta(\frac{C(\omega)\p^2Z^bu Z^cu}{|x|})
+\sum_{\mu=0}^2\p_{\mu}(\frac{C(\omega)Z\p Z^bu Z^cu}{|x|})}_{\bf divergence~term}\\
&\quad +\underbrace{\frac{C(\omega)Z\p Z^bu Z^cu}{|x|^2}}_{\text{\bf quadratic~error~term~with~better~decay~factor~$|x|^{-2}$}}\Big\}.
\end{split}
\end{equation}
\end{lemma}
\begin{remark}\label{YHCS-1}
The introduction on the good unknown $V_a$ in \eqref{def:goodunknown} is
such that the resulting $Q_0$ null form nonlinearity can be eliminated and spacial localized.
\end{remark}

\begin{proof}
Note that \eqref{eqn:reduce} comes from \eqref{eqn:high'} and \eqref{Q0i:structure1} directly.

In addition, according to the definition \eqref{def:goodunknown}, one has
\begin{equation*}
\begin{split}
&\Box V_a=\Box(\tilde Z^a-Z^a)u+\Box Z^au+\frac{1}{2}\sum_{b+c\le a}\sum_{\mu=0}^2\Big\{[\Delta,C_{abc}^{0,\mu}(\omega)\chi_{[1/2,1]}(x)](\p_{\mu}Z^buZ^cu)\\
&\qquad\qquad-C_{abc}^{0,\mu}(\omega)\chi_{[1/2,1]}(x)\Box(\p_{\mu}Z^buZ^cu)\Big\}.
\end{split}
\end{equation*}
This, together with the equality $\Box(fg)=g\Box f+f\Box g+2Q_0(f,g)$ and \eqref{eqn:reduce}, yields \eqref{eqn:good}.

The estimate of the commutator \eqref{comm:estimate} comes from the expressions  $\Delta=\p_r^2+r^{-1}\p_r+r^{-2}\Omega^2$, $\omega=(\cos\theta,\sin\theta)$ and $\p_r\omega=0$ in the polar coordinates together with direct computation.

The proof of \eqref{eqn:good:div} is analogous to that of \eqref{eqn:good} with \eqref{Q0i:structure2} instead of \eqref{Q0i:structure1}.
\end{proof}

\section{Preliminaries}\label{sect3}

\subsection{Some useful lemmas}

In this subsection, we list some useful conclusions  including Gronwall's Lemma, the Sobolev embedding, the elliptic estimate
and  the local energy decay estimate.

\begin{lemma}\label{lem:Gronwall}
For the positive constants $A,B,C,D$ with $B<D$ and function $f(t)\ge0$, if
\begin{equation*}
f(t)\le A+B\int_0^t\frac{f(s)ds}{1+s}+C(1+t)^D,
\end{equation*}
then it holds
\begin{equation*}
f(t)\le A(1+t)^B+C(1+t)^D+\frac{BC}{D-B}(1+t)^D.
\end{equation*}
\end{lemma}

\begin{lemma}[Lemma 3.6 of \cite{Kubo14}]
For any $f(x)\in C_0^2(\overline{\cK})$, one has that for $x\in\cK$,
\begin{equation}\label{Sobo:ineq}
\w{x}^{1/2}|f(x)|\ls\sum_{|a|\le2}\|Z^af\|_{L^2(\cK)}.
\end{equation}
\end{lemma}

\begin{lemma}[Lemma 3.2 of \cite{Kubo13}]
Assume $w\in H^j(\cK)$ ($j\ge2$, $j\in\Bbb N$) and $w|_{\p\cK}=0$.
Then for any fixed constant $R>1$ and multi-index $a\in\N_0^2$ with $2\le|a|\le j$, it holds that
\begin{equation}\label{ellip}
\|\p_x^aw\|_{L^2(\cK)}\ls\|\Delta w\|_{H^{|a|-2}(\cK)}+\|w\|_{H^{|a|-1}(\cK_{R+1})}.
\end{equation}
\end{lemma}

\begin{lemma}[Lemma 3.2 of \cite{Kubo14}]
Suppose that the obstacle $\cO$ is star-shaped, $\cK=\R^2\setminus\cO$ and $w$ solves the IBVP
\begin{equation*}
\left\{
\begin{aligned}
&\Box w=G(t,x),\qquad(t,x)\in[0,\infty)\times\cK,\\
&w|_{\p\cK}=0,\\
&(w,\p_tw)(0,x)=(w_0,w_1)(x),\quad x\in\cK,
\end{aligned}
\right.
\end{equation*}
where $\supp_x(w_0(x),w_1(x),G(t,x))\subset\{x: |x|\le R_1\}$ for $R_1>1$.
Then for fixed $R>1$, $\rho\in(0,1]$ and $m\in\N$, there is a positive constant $C=C(m,R,R_1)$ such that
\begin{equation}\label{loc:decay}
\sum_{|a|\le m}\w{t}^\rho\|\p^aw\|_{L^2(\cK_R)}
\le C(\|(w_0,w_1)\|_{H^{m}(\cK)}+\sum_{|a|\le m-1}\sup_{0\le s\le t}\w{s}^\rho\|\p^aG(s)\|_{L^2(\cK)}).
\end{equation}
\end{lemma}

\subsection{Pointwise estimates of Cauchy problem}

\begin{lemma}
Let $v$ be the solution of the Cauchy problem
\begin{equation*}
\left\{
\begin{aligned}
&\Box v=H(t,x),\qquad(t,x)\in[0,\infty)\times\R^2,\\
&(v,\p_tv)(0,x)=(v_0,v_1)(x),\quad x\in\R^2.
\end{aligned}
\right.
\end{equation*}
Then for any fixed $\mu,\nu\in(0,1/2)$, we have
\begin{align}
&\w{t+|x|}^{1/2}\w{t-|x|}^{\mu}|v|\ls\cA_{3,0}[v_0,v_1]
+\sup_{(s,y)\in[0,t]\times\R^2}\w{y}^{1/2}\cW_{1+\mu,1+\nu}(s,y)|H(s,y)|,\label{pw:ivp1}\\
&\w{t+|x|}^{1/2}\w{t-|x|}^{1/2}|v|\ls\cA_{3,0}[v_0,v_1]
+\sup_{(s,y)\in[0,t]\times\R^2}\w{y}^{1/2}\cW_{3/2+\mu,1+\nu}(s,y)|H(s,y)|,\label{pw:ivp2}
\end{align}
where $\ds\cA_{\kappa,s}[f,g]:=\sum_{\tilde\Gamma\in\{\p_1,\p_2,\Omega\}}
(\sum_{|a|\le s+1}\|\w{z}^\kappa\tilde\Gamma^af(z)\|_{L^\infty}+\sum_{|a|\le s}\|\w{z}^\kappa\tilde\Gamma^ag(z)\|_{L^\infty})$ and $\cW_{\mu,\nu}(t,x)=\w{t+|x|}^\mu(\min\{\w{x},\w{t-|x|}\})^\nu$.
\end{lemma}
\begin{proof}
See (4.2) and (4.3) in \cite{Kubo19} for \eqref{pw:ivp1}.
In addition, the proof of \eqref{pw:ivp2} is similar to that of (4.3) in \cite{Kubo19}.
In fact, applying (4.2) and (A.13) in \cite{Kubo19} with $c_j=0,1$ yields \eqref{pw:ivp2}.
\end{proof}
\begin{remark}
When $\mu=0$ in \eqref{pw:ivp2}, it holds that
\begin{equation}\label{pw:crit}
\begin{split}
&\w{t+|x|}^{1/2}\w{t-|x|}^{1/2}|v|\ls\cA_{3,0}[v_0,v_1]\\
&\qquad+\ln(2+t)\sup_{(s,y)\in[0,t]\times\R^2}\w{y}^{1/2}\cW_{3/2,1+\nu}(s,y)|H(s,y)|.
\end{split}
\end{equation}
\end{remark}

\begin{lemma}
Let $v$ be the solution of the Cauchy problem
\begin{equation*}
\left\{
\begin{aligned}
&\Box v=H(t,x),\qquad(t,x)\in[0,\infty)\times\R^2,\\
&(v,\p_tv)(0,x)=(v_0,v_1)(x),\quad x\in\R^2.
\end{aligned}
\right.
\end{equation*}
Then for $\mu,\nu\in(0,1/2)$, one has
\begin{equation}\label{dpw:ivp2}
\begin{split}
&\w{x}^{1/2}\w{t-|x|}^{1+\mu}|\p v|\ls\cA_{4,1}[v_0,v_1]\\
&\qquad+\sum_{|a|\le1}\sup_{(s,y)\in[0,t]\times\R^2}\w{y}^{1/2}\cW_{1+\mu+\nu,1}(s,y)|Z^aH(s,y)|.
\end{split}
\end{equation}
\end{lemma}
\begin{proof}
The proof of \eqref{dpw:ivp2} follows from (4.2) and (4.4) in \cite{Kubo19}.
\end{proof}

\begin{lemma}
Let $v$ be the solution of the Cauchy problem
\begin{equation}\label{wave:compact}
\left\{
\begin{aligned}
&\Box v=H(t,x),\qquad\qquad(t,x)\in[0,\infty)\times\R^2,\\
&(v,\p_tv)(0,x)=(0,0),\quad x\in\R^2,
\end{aligned}
\right.
\end{equation}
where $\supp_x H(t,x)\subset\{x: |x|\le R\}$.
Then we have that for $0<\delta<\rho\le1/2$,
\begin{equation}\label{dpw:loc}
\w{x}^{1/2}\w{t-|x|}^{\rho-\delta}|\p v|
\ls\sum_{|a|\le1}\sup_{(s,y)\in[0,t]\times\R^2}\w{s}^\rho|\p^aH(s,y)|.
\end{equation}
\end{lemma}
\begin{proof}
Choosing $\sigma=\rho-\delta$, $\mu=\delta>0$ and $m=0$ in (51) of \cite{KKL13} yields \eqref{dpw:loc}.
\end{proof}

\subsection{Improved pointwise estimates of the Cauchy problem}\label{sect2-4}

We now list two crucial improved pointwise estimates of the Cauchy problem, whose proofs will be
postponed to Section \ref{sect-app-B}.
\begin{lemma}\label{lem:impr:pw}
Let $v$ be the solution of the Cauchy problem
\begin{equation}\label{YHC-2}
\left\{
\begin{aligned}
&\Box v=H(t,x),\qquad\qquad(t,x)\in[0,\infty)\times\R^2,\\
&(v,\p_tv)(0,x)=(0,0),\quad x\in\R^2,
\end{aligned}
\right.
\end{equation}
provided that $\supp H(t,\cdot)\subset\{|x|\le R\}$.
Then one has that for $\mu\in(0,1/2]$,
\begin{equation}\label{impr:pw:Cauchy}
\w{t+|x|}^{1/2}\w{t-|x|}^{\mu}|v|
\ls\ln^2(2+t+|x|)\sup_{(s,y)\in[0,t]\times\R^2}\w{s}^{1/2+\mu}|(1-\Delta)H(s,y)|.
\end{equation}
\end{lemma}

\begin{remark}\label{rmk:impr:pw}
{\it Recalling (4.2)  with $m=0$ in \cite{Kubo14}, there holds
\begin{equation}\label{YHC-5}
|w|\ls\sup_{(s,y)\in[0,t]\times\R^2}\w{s}^{1/2}|g(s,y)|,
\end{equation}
where $w$ solves $\Box w=g(t,x)$, $(t,x)\in[0,\infty)\times\R^2$, $(w,\p_tw)|_{t=0}=(0,0)$ and $\supp g(t,\cdot)\subset\{1\le|x|\le a\}$
with constant $a>1$.
On the other hand, choosing $\mu=0.01$ in \eqref{impr:pw:Cauchy}, we then have
\begin{equation*}
\begin{split}
|v|&\ls\w{t+|x|}^{-1/2}\w{t-|x|}^{-0.01}\ln^2(2+t+|x|)\sup_{(s,y)\in[0,t]\times\R^2}\w{s}^{0.51}|(1-\Delta)H(s,y)|\\
&\ls\w{t+|x|}^{-1/2+0.01}\w{t}^{0.01}\sup_{(s,y)\in[0,t]\times\R^2}\w{s}^{1/2}|(1-\Delta)H(s,y)|\\
&\ls\w{t+|x|}^{-1/2+0.02}\sup_{(s,y)\in[0,t]\times\R^2}\w{s}^{1/2}|(1-\Delta)H(s,y)|.
\end{split}
\end{equation*}
Thus, \eqref{impr:pw:Cauchy} is really an obvious improvement of (4.2) in \cite{Kubo14}
(also see \eqref{YHC-5}) when $H$ admits higher regularities.
With the new improved \eqref{impr:pw:Cauchy}, the space-time decay estimate of the solution itself in \eqref{thm1:decay:c}
will be shown.}
\end{remark}

\begin{lemma}\label{lem:impr:dpw}
Let $v$ be the solution of the Cauchy problem
\begin{equation}\label{YHC-3}
\left\{
\begin{aligned}
&\Box v=H(t,x),\qquad\qquad(t,x)\in[0,\infty)\times\R^2,\\
&(v,\p_tv)(0,x)=(0,0),\quad x\in\R^2,
\end{aligned}
\right.
\end{equation}
provided that $\supp H(t,\cdot)\subset\{|x|\le R\}$.
Then we have
\begin{equation}\label{impr:dpw:Cauchy}
\w{x}^{1/2}\w{t-|x|}|\p v|
\ls\sup_{(s,y)\in[0,t]\times\R^2}\w{s}|(1-\Delta)^3H(s,y)|.
\end{equation}
\end{lemma}
\begin{remark}
{\it In terms of  (57) with $\eta=0,\rho=1$ in \cite{Kubo13}
\begin{equation}\label{Kubo13:dpw}
|\p w|\ls\frac{\ln(2+t+|x|)}{\w{x}^{1/2}\w{t-|x|}+\w{t-|x|}^{1/2}\w{t-|x|}^{1/2}}
\sum_{|b|\le1}\sup_{(s,y)\in[0,t]\times\R^2}\w{s}|\p^bg(s,y)|,
\end{equation}
where $w$ solves $\Box w=g(t,x)$, $(t,x)\in[0,\infty)\times\R^2$, $(w,\p_tw)|_{t=0}=(0,0)$ and $\supp g(t,\cdot)\subset\{1\le|x|\le a\}$
with constant $a>1$, the factor $\ln(2+t+|x|)$ appearing in \eqref{Kubo13:dpw} has been removed in
our estimate \eqref{impr:dpw:Cauchy}. This essential  improvement will be crucial in the proof of Theorem \ref{thm1}.}
\end{remark}

\section{Estimates of the Cauchy problem: Proofs of Lemmas \ref{lem:impr:pw} and \ref{lem:impr:dpw}}\label{sect-app-B}

Choosing a smooth cutoff function $\psi: \R\rightarrow[0,1]$, which equals 1 on $[-5/4,5/4]$ and vanishes outside $[-8/5,8/5]$,
one sets
\begin{equation*}
\begin{split}
&\psi_k(x):=\psi(|x|/2^k)-\psi(|x|/2^{k-1}),\quad k\in\Z.
\end{split}
\end{equation*}
Let $P_k$ be the Littlewood-Paley projection onto frequency $2^k$
\begin{equation*}
\sF(P_kf)(\xi):=\psi_k(\xi)\sF f(\xi),\quad k\in\Z.
\end{equation*}
We now start to prove Lemma \ref{lem:impr:pw} in Section \ref{sect3}.

\begin{proof}[Proof of Lemma \ref{lem:impr:pw}]
The solution $v$ of problem \eqref{YHC-2} can be represented as
\begin{equation}\label{app:pw1}
\begin{split}
v(t,x)&=\int_0^t\frac{\sin(t-s)|\nabla|}{|\nabla|}H(s)ds
=\int_0^t\frac{e^{i(t-s)|\nabla|}-e^{-i(t-s)|\nabla|}}{i2|\nabla|}H(s)ds\\
&=\frac{1}{2i}\int_0^t\int_{\R^2}\sum_{k\in\Z}(K^+_k-K^-_k)(t-s,x-y)(1-\Delta)H(s,y)dyds,
\end{split}
\end{equation}
where
\begin{equation}\label{app:pw2}
K^{\pm}_k(t,x):=\frac{1}{(2\pi)^2}\int_{\R^2}e^{i(\pm t|\xi|+x\cdot\xi)}\frac{\psi_k(\xi)}{|\xi|(1+|\xi|^2)}d\xi.
\end{equation}
We next show
\begin{equation}\label{app:pw3}
|K^{\pm}_k(t,x)|\ls2^{k-2k_+}(1+2^k(t+|x|))^{-1/2}(1+2^k|t-|x||)^{-1/2},\quad t\ge0,
\end{equation}
where $k_+:=\max\{k,0\}$.

To this end, set $\xi=2^k\eta$ in \eqref{app:pw2}, one then has
\begin{equation}\label{app:pw4}
K^{\pm}_k(t,x)=\frac{2^k}{(2\pi)^2}\int_{\R^2}e^{i2^k(\pm t|\eta|+x\cdot\eta)}\frac{\psi_0(\eta)}{|\eta|(1+2^{2k}|\eta|^2)}d\eta,
\end{equation}
which yields
\begin{equation}\label{app:pw5}
|K^{\pm}_k(t,x)|\ls2^{k-2k_+}.
\end{equation}
Denote $\phi:=i2^k(\pm t|\eta|+x\cdot\eta)$. We have that for $t\ge2|x|$,
\begin{equation}\label{app:pw6}
|\nabla_\eta\phi|=2^k\Big|\pm t\frac{\eta}{|\eta|}+x\Big|\ge2^k|t-|x||\ge2^{k-1}t>0.
\end{equation}
Due to $|\nabla^2_\eta\phi|\ls2^kt$ with $\eta\in\supp\psi_0$,
then it follows from the method of stationary phase that
\begin{equation}\label{app:pw7}
\begin{split}
|K^{\pm}_k(t,x)|&=\frac{2^k}{(2\pi)^2}\Big|\int_{\R^2}\frac{\nabla_\eta\phi\cdot\nabla_\eta e^{i2^k(\pm t|\eta|+x\cdot\eta)}}{|\nabla_\eta\phi|^2}
\frac{\psi_0(\eta)}{|\eta|(1+2^{2k}|\eta|^2)}d\eta\Big|\\
&\ls2^k\int_{\R^2}\Big|\nabla_\eta\cdot\Big(\frac{\psi_0(\eta)\nabla_\eta\phi}{|\eta|(1+2^{2k}|\eta|^2)|\nabla_\eta\phi|^2}\Big)\Big|d\eta\\
&\ls2^{k-2k_+}(2^kt)^{-1}\ls2^{k-2k_+}(2^k(t+|x|))^{-1}.
\end{split}
\end{equation}
This, together with \eqref{app:pw5}, yields \eqref{app:pw3} in the region $t\ge2|x|$.

We now focus on the proof of \eqref{app:pw3} in the region $t\le2|x|$.
If $2^k|x|\le10$, \eqref{app:pw3} can be achieved by \eqref{app:pw5}.
From now on, $2^k|x|\ge10$ is assumed.
Under the polar coordinate $\eta=\rho w$ with $w\in\SS$, \eqref{app:pw4} can be rewritten as
\begin{equation}\label{app:pw8}
K^{\pm}_k(t,x)=\frac{2^k}{(2\pi)^2}\int\int_{\SS}e^{i2^k\rho(\pm t+x\cdot w)}\frac{\psi_0(\rho)}{(1+2^{2k}\rho^2)}dwd\rho,
\end{equation}
where we have used the fact of $\psi_0(x)=\psi_0(|x|)$.
Recalling Corollary 2.37 of \cite{NS11book}, one has
\begin{equation}\label{app:pw9}
\begin{split}
&\widehat{\sigma_{\SS}}(x)=\int_{\SS}e^{ix\cdot w}dw=e^{i|x|}\omega_+(|x|)+e^{-i|x|}\omega_-(|x|),\\
&|\p_r^l\omega_\pm(r)|\ls r^{-1/2-l}\qquad\text{for $r\ge1$ and $l\ge0$}.
\end{split}
\end{equation}
Therefore, \eqref{app:pw8} is reduced to
\begin{equation}\label{app:pw10}
\begin{split}
K^+_k(t,x)&=\frac{2^k}{(2\pi)^2}\int(e^{i2^k\rho(t+|x|)}\omega_+(2^k\rho|x|)+e^{i2^k\rho(t-|x|)}\omega_-(2^k\rho|x|))
\frac{\psi_0(\rho)}{(1+2^{2k}\rho^2)}d\rho\\
&:=K^{++}_k(t,x)+K^{+-}_k(t,x).
\end{split}
\end{equation}
It follows from \eqref{app:pw9} and \eqref{app:pw10} with $2^k|x|\ge10$ that
\begin{equation}\label{app:pw11}
|K^{+\pm}_k(t,x)|\ls2^{k-2k_+}(2^k|x|)^{-1/2}\ls2^{k-2k_+}(1+2^k(t+|x|))^{-1/2},
\end{equation}
which derives \eqref{app:pw3} in the region $2^k|t-|x||\le1$.

For $2^k|t-|x||\ge1$, through integrating by parts with respect to $\rho$ in $K^{+-}_k(t,x)$ we arrive at
\begin{equation}\label{app:pw12}
\begin{split}
|K^{+-}_k(t,x)|&=\Big|\frac{2^k}{(2\pi)^2}\int\frac{\p_\rho e^{i2^k\rho(t-|x|)}}{i2^k(t-|x|)}\omega_-(2^k\rho|x|)
\frac{\psi_0(\rho)}{(1+2^{2k}\rho^2)}d\rho\Big|\\
&\ls2^k(2^k(|t-|x||)^{-1}\int\Big|\p_\rho\Big(\frac{\omega_-(2^k\rho|x|)\psi_0(\rho)}{(1+2^{2k}\rho^2)}\Big)\Big|d\rho\\
&\ls2^{k-2k_+}(2^k|t-|x||)^{-1}(2^k|x|)^{-1/2}\\
&\ls2^{k-2k_+}(1+2^k|t-|x||)^{-1}(1+2^k(t+|x|))^{-1/2}.\\
\end{split}
\end{equation}
The estimates of $K^{++}_k(t,x)$ and $K^-_k(t,x)$ are analogous.
Thus, collecting \eqref{app:pw5}, \eqref{app:pw7} and \eqref{app:pw10}-\eqref{app:pw12} yields \eqref{app:pw3}.

Based on \eqref{app:pw1} and \eqref{app:pw3}, we now prove Lemma \ref{lem:impr:pw}.
For fixed $t$ and $x$, set $\bar{K}:=-[2\log_2(2+t+|x|)]<0$, then
\begin{equation}\label{app:pw13}
2^{\bar{K}}\approx\w{t+|x|}^{-2},\qquad |\bar{K}|\ls\ln(2+t+|x|).
\end{equation}
Thus, \eqref{app:pw1} can be separated into such three parts $v=v_{low}+v_-+v_+$, where
\begin{equation}\label{app:pw14}
\begin{split}
v_{low}(t,x)&:=\int_0^t\int_{\R^2}\sum_{k<\bar{K}}(K^+_k-K^-_k)(t-s,x-y)(1-\Delta)H(s,y)dyds,\\
v_-(t,x)&:=\int_0^t\int_{\R^2}\sum_{\bar{K}\le k<0}(K^+_k-K^-_k)(t-s,x-y)(1-\Delta)H(s,y)dyds,\\
v_+(t,x)&:=\int_0^t\int_{\R^2}\sum_{k\ge0}(K^+_k-K^-_k)(t-s,x-y)(1-\Delta)H(s,y)dyds.
\end{split}
\end{equation}
Due to $\supp(1-\Delta)H(s,\cdot)\subset\{|y|\le R\}$, we can easily obtain from \eqref{app:pw5} and \eqref{app:pw13} that
\begin{equation}\label{app:pw15}
|v_{low}|\ls2^{\bar{K}}t\sup_{(s,y)\in[0,t]\times\R^2}|(1-\Delta)H(s,y)|\ls\w{t+|x}^{-1}\sup_{(s,y)\in[0,t]\times\R^2}|(1-\Delta)H(s,y)|.
\end{equation}
On the other hand, one has
\begin{equation*}
\frac{1}{1+2^k(t+|x|)}\ls\left\{
\begin{aligned}
&\frac{1}{2^k(1+t+|x|)},&&\qquad k\le0,\\
&\frac{1}{1+t+|x|},&&\qquad k\ge0,
\end{aligned}
\right.
\end{equation*}
which leads to
\begin{equation*}
(1+2^k|t+|x||)^{-1/2}(1+2^k|t-|x||)^{-1/2}\ls2^{k_+-k}(1+t+|x|)^{-1/2}(1+|t-|x||)^{-1/2}.
\end{equation*}
Thus, we have from \eqref{app:pw3} that
\begin{equation*}
|K^{\pm}_k(t,x)|\ls2^{-k_+}(1+t+|x|)^{-1/2}(1+|t-|x||)^{-1/2}.
\end{equation*}
Together with \eqref{app:pw13} and \eqref{app:pw14}, this yields
\begin{equation}\label{app:pw16}
\begin{split}
|v_-|&\ls I^\pm|\bar{K}|\sup_{(s,y)\in[0,t]\times\R^2}(1+s)^{1/2+\mu}|(1-\Delta)H(s,y)|\\
&\ls I^\pm\ln(2+t+|x|)\sup_{(s,y)\in[0,t]\times\R^2}(1+s)^{1/2+\mu}|(1-\Delta)H(s,y)|,\\
|v_+|&\ls I^\pm\sup_{(s,y)\in[0,t]\times\R^2}(1+s)^{1/2+\mu}|(1-\Delta)H(s,y)|,
\end{split}
\end{equation}
where
\begin{equation*}
I^\pm:=\int_{|y|\le R}\int_0^t\frac{dsdy}{(1+t-s+|x-y|)^{1/2}(1+|t-s-|x-y||)^{1/2}(1+s)^{1/2+\mu}}.
\end{equation*}
We can assert
\begin{equation}\label{app:pw17}
I^\pm\ls\frac{\ln(2+t+|x|)}{\w{t+|x|}^{1/2}\w{t-|x|}^{\mu}}.
\end{equation}
Note that if \eqref{app:pw17} is proved, then Lemma \ref{lem:impr:pw} will be obtained by \eqref{app:pw14}, \eqref{app:pw15}
and \eqref{app:pw16}.
In the rest part, we are devoted to the proof of assertion \eqref{app:pw17}.

Set $I^\pm=I^\pm_1+I^\pm_2$ with
\begin{equation}\label{app:pw18}
\begin{split}
I^\pm_1&:=\int_{|y|\le R,|x-y|\ge t}\int_0^t\frac{(1+s)^{-1/2-\mu}dsdy}{(1+t-s+|x-y|)^{1/2}(1+|t-s-|x-y||)^{1/2}},\\
I^\pm_2&:=\int_{|y|\le R,|x-y|\le t}\int_0^t\frac{(1+s)^{-1/2-\mu}dsdy}{(1+t-s+|x-y|)^{1/2}(1+|t-s-|x-y||)^{1/2}}.
\end{split}
\end{equation}
From the expression of $I^\pm_1$, one knows $s\le t\le\frac{t+|x-y|}{2}$.
In this case, $1+t-s+|x-y|\approx1+t+|x-y|\approx\w{t+|x|}$ and $1+|t-s-|x-y||=1+s+|x-y|-t\ge1+|x-y|-t\approx\w{t-|x|}$ hold.
Hence, we can achieve
\begin{equation}\label{app:pw19}
\begin{split}
I^\pm_1&\ls\w{t+|x|}^{-1/2}\int_{|y|\le R,|x-y|\ge t}\int_0^t\frac{dsdy}{(1+s+|x-y|-t)^{1/2}(1+s)^{1/2+\mu}}\\
&\ls\w{t+|x|}^{-1/2}\w{t-|x|}^{-\mu}\int_{|y|\le R}\int_0^t\frac{dsdy}{1+s}\\
&\ls\ln(2+t)\w{t+|x|}^{-1/2}\w{t-|x|}^{-\mu}.
\end{split}
\end{equation}
Next, we estimate $I^\pm_2$. Set $I^\pm_2=I^\pm_{21}+I^\pm_{22}+I^\pm_{23}$, where
\begin{equation}\label{app:pw20}
\begin{split}
I^\pm_{21}&:=\int_{\substack{|y|\le R,\\|x-y|\le t}}\int_0^{\frac{t-|x-y|}{2}}
\frac{(1+s)^{-1/2-\mu}dsdy}{(1+t-s+|x-y|)^{1/2}(1+|t-s-|x-y||)^{1/2}},\\
I^\pm_{22}&:=\int_{\substack{|y|\le R,\\|x-y|\le t/2}}\int_{\frac{t-|x-y|}{2}}^t\frac{(1+s)^{-1/2-\mu}dsdy}{(1+t-s+|x-y|)^{1/2}(1+|t-s-|x-y||)^{1/2}},\\
I^\pm_{23}&:=\int_{\substack{|y|\le R,\\t/2\le|x-y|\le t}}\int_{\frac{t-|x-y|}{2}}^t\frac{(1+s)^{-1/2-\mu}dsdy}{(1+t-s+|x-y|)^{1/2}(1+|t-s-|x-y||)^{1/2}}.
\end{split}
\end{equation}
By the expression of $I^\pm_{21}$, we have $1+|t-s-|x-y||\approx\w{t-|x|}$ and $1+t-s+|x-y|\approx\w{t+|x|}$.
On the other hand, $1+|t-s-|x-y||\gt1+s$ holds. Hence,
\begin{equation}\label{app:pw21}
\begin{split}
I^\pm_{21}&\ls\w{t+|x|}^{-1/2}\w{t-|x|}^{-\mu}\int_{|y|\le R}\int_0^{\frac{t-|x-y|}{2}}\frac{dsdy}{1+s}\\
&\ls\ln(2+t)\w{t+|x|}^{-1/2}\w{t-|x|}^{-\mu}.
\end{split}
\end{equation}
For the integral $I^\pm_{22}$, due to $1+t-|x-y|\gt1+t\gt1+t+|x|$, one then has
\begin{equation}\label{app:pw22}
\begin{split}
&\w{t+|x|}^{1/2+\mu}I^\pm_{22}\\
&\ls\int_{|y|\le R}\int_{\frac{t-|x-y|}{2}}^t\frac{dsdy}{(1+t-s+|x-y|)^{1/2}(1+|t-s-|x-y||)^{1/2}}\\
&\ls\int_{|y|\le R}\int_{\frac{t-|x-y|}{2}}^t\frac{dsdy}{1+|t-s-|x-y||}\\
&\ls\ln(2+t+|x|).
\end{split}
\end{equation}
In addition, the integral $I^\pm_{23}$ can be separated into the following two parts
\begin{equation}\label{app:pw23}
\begin{split}
I^\pm_{231}&:=\int_{\substack{|y|\le R,\\t/2\le|x-y|\le t}}\int_{\frac{t-|x-y|}{2}}^{t-|x-y|}
\frac{(1+s)^{-1/2-\mu}dsdy}{(1+t-s+|x-y|)^{1/2}(1+t-s-|x-y|)^{1/2}},\\
I^\pm_{232}&:=\int_{\substack{|y|\le R,\\t/2\le|x-y|\le t}}\int_{t-|x-y|}^t
\frac{(1+s)^{-1/2-\mu}dsdy}{(1+t-s+|x-y|)^{1/2}(1+s+|x-y|-t|)^{1/2}}.
\end{split}
\end{equation}
It follows from direct computation that
\begin{equation}\label{app:pw24}
\begin{split}
I^\pm_{231}&\ls\frac{1}{(1+|x|)^{1/2}\w{t-|x|}^{1/2+\mu}}
\int_{\substack{|y|\le R,\\t/2\le|x-y|\le t}}\int_{\frac{t-|x-y|}{2}}^{t-|x-y|}\frac{dsdy}{(1+t-s-|x-y|)^{1/2}},\\
&\ls\frac{1}{\w{t+|x|}^{1/2}\w{t-|x|}^{\mu}}
\end{split}
\end{equation}
and
\begin{equation}\label{app:pw25}
\begin{split}
I^\pm_{232}&\ls\frac{1}{\w{t+|x|}^{1/2}\w{t-|x|}^{\mu}}
\int_{\substack{|y|\le R,\\t/2\le|x-y|\le t}}\int_{t-|x-y|}^t\frac{dsdy}{(1+s+|x-y|-t)^{1/2}(1+s)^{1/2}}\\
&\ls\frac{1}{\w{t+|x|}^{1/2}\w{t-|x|}^{\mu}}
\int_{|y|\le R}\int_{t-|x-y|}^t\frac{dsdy}{1+s+|x-y|-t}\\
&\ls\ln(2+t+|x|)\w{t+|x|}^{-1/2}\w{t-|x|}^{-\mu}.
\end{split}
\end{equation}
Collecting \eqref{app:pw18}-\eqref{app:pw25} yields \eqref{app:pw17}.
This completes the proof of Lemma \ref{lem:impr:pw}.
\end{proof}

\vskip 0.1 true cm

Next we prove Lemma \ref{lem:impr:dpw}.

\begin{proof}[Proof of Lemma \ref{lem:impr:dpw}]
Instead of \eqref{app:pw1}, for the solution $v$ of problem \eqref{YHC-3}, $\p_tv$ and $\p_jv$ ($j=1,2$) can be expressed as
\begin{equation}\label{app:dpw1}
\begin{split}
\p_tv(t,x)&=\int_0^t\big(\cos(t-s)|\nabla|\big)H(s)ds
=\int_0^t\frac{e^{i(t-s)|\nabla|}+e^{-i(t-s)|\nabla|}}{2}H(s)ds\\
&=\frac{1}{2}\int_0^t\int_{\R^2}\sum_{k\in\Z}(K^{0+}_k+K^{0-}_k)(t-s,x-y)(1-\Delta)^3H(s,y)dyds\\
&\text{with $K^{0\pm}_k(t,x):=\frac{1}{(2\pi)^2}\int_{\R^2}e^{i(\pm t|\xi|+x\cdot\xi)}\frac{\psi_k(\xi)}{(1+|\xi|^2)^3}d\xi$}
\end{split}
\end{equation}
and
\begin{equation}\label{app:dpw2}
\begin{split}
\p_jv(t,x)&=\int_0^t\big(\sin(t-s)|\nabla|\big)R_jH(s)ds\\
&=\frac{1}{2i}\int_0^t\int_{\R^2}\sum_{k\in\Z}(K^+_k-K^-_k)(t-s,x-y)P_kR_j(1-\Delta)^3H(s,y)dyds\\
&\text{with $K^{j\pm}_k(t,x):=\frac{1}{(2\pi)^2}\int_{\R^2}e^{i(\pm t|\xi|+x\cdot\xi)}\frac{\psi_{[k-1,k+1]}(\xi)}{(1+|\xi|^2)^3}d\xi$},
\end{split}
\end{equation}
where $R_j=\p_j|\nabla|^{-1}$ ($j=1,2$) is the Riesz transformation, and the fact $\psi_k(\xi)=\psi_{[k-1,k+1]}(\xi)\psi_k(\xi)$
has been used.
Note that although $P_kR_j(1-\Delta)^3H$ no longer admits the compact support on the space variable,
it can decay sufficiently fast in space variables as follows
\begin{equation}\label{app:dpw3}
|P_kR_j(1-\Delta)^3H(s,x)|\ls2^{2k}(1+2^{k-k_+}|x|)^{-5}\|(1-\Delta)^3H(s,\cdot)\|_{L^\infty(\R^2)}.
\end{equation}
Indeed, note that
\begin{equation}\label{app:dpw4}
\begin{split}
P_kR_j(1-\Delta)^3H(s,x)&=\int_{\R^2}K_{kj}(x-y)(1-\Delta)^3H(s,y)dy,\\
K_{kj}(x)&:=\frac{i}{(2\pi)^2}\int_{\R^2}e^{ix\cdot\xi}\frac{\xi_j\psi_k(\xi)}{|\xi|}d\xi,\\
|K_{kj}(x)|&\ls2^{2k}(1+2^k|x|)^{-5}.
\end{split}
\end{equation}
For $|x|\ge2R$, the assumption of $\supp(1-\Delta)^3H(s,y)\subset\{|y|\le R\}$ ensures $|x-y|\ge|x|/2$.
Thus, it follows from \eqref{app:dpw4} that
\begin{equation*}
\begin{split}
|P_kR_j(1-\Delta)^3H(s,x)|&\ls2^{2k}(1+2^k|x|)^{-5}\|(1-\Delta)^3H(s,\cdot)\|_{L^\infty(\R^2)}\\
&\ls2^{2k}(1+2^{k-k_+}|x|)^{-5}\|(1-\Delta)^3H(s,\cdot)\|_{L^\infty(\R^2)}.
\end{split}
\end{equation*}
For $|x|\le2R$, due to $2^{k-k_+}|x|\ls1$, then \eqref{app:dpw3} can be easily obtained.
Therefore, \eqref{app:dpw3} is proved.

In the following, it suffices to deal with the estimates of $\p_jv$ given by \eqref{app:dpw2} since
the estimate of $\p_tv$ in \eqref{app:dpw1} can be treated similarly.
Analogously to \eqref{app:pw3}, we can arrive at
\begin{equation}\label{app:dpw5}
|K^{j\pm}_k(t,x)|\ls2^{2k-6k_+}(1+2^k|x|)^{-1/2}(1+2^k|t-|x||)^{-2}.
\end{equation}
Substituting \eqref{app:dpw3} and \eqref{app:dpw5} into \eqref{app:dpw2} yields
\begin{equation}\label{app:dpw6}
\begin{split}
&|\p_jv|\ls\sup_{(s,y)\in[0,t]\times\R^2}(1+s)|(1-\Delta)^3H(s,y)|\sum_{k\in\Z}2^{4k-6k_+}\\
&\times\int_0^t\int_{\R^2}\frac{(1+s)^{-1}(1+2^{k-k_+}|y|)^{-5}dyds}{(1+2^k|x-y|)^{1/2}(1+2^k|t-s-|x-y||)^{2}}.
\end{split}
\end{equation}
It is easy to check that
\begin{equation*}
(1+2^k|x-y|)^{-1/2}(1+2^{k-k_+}|y|)^{-1/2}\ls2^{(k_+-k)/2}\w{x}^{-1/2}.
\end{equation*}
This, together with \eqref{app:dpw6}, leads to
\begin{equation}\label{app:dpw7}
\begin{split}
|\p_jv|&\ls\w{x}^{-1/2}\sup_{(s,y)\in[0,t]\times\R^2}(1+s)|(1-\Delta)^3H(s,y)|\sum_{k\in\Z}2^{(7k-11k_+)/2}J^k,\\
J^k&:=\int_{\R^2}\int_0^t\frac{dsdy}{(1+2^k|t-s-|x-y||)^2(1+s)(1+2^{k-k_+}|y|)^4}.
\end{split}
\end{equation}
Analogously to the proof of Lemma \ref{lem:impr:pw}, the integral $J^k$ can be separated into the following two parts
\begin{equation}\label{app:dpw8}
\begin{split}
J^k&=J^k_1+J^k_2,\\
J^k_1&:=\int_{|x-y|\ge t}\int_0^t\frac{dsdy}{(1+2^k(|x-y|+s-t))^2(1+s)(1+2^{k-k_+}|y|)^4},\\
J^k_2&:=\int_{|x-y|\le t}\int_0^t\frac{dsdy}{(1+2^k|t-s-|x-y||)^2(1+s)(1+2^{k-k_+}|y|)^4}.
\end{split}
\end{equation}
The integral $J^k_1$ is further divided into such two parts
\begin{equation}\label{app:dpw9}
\begin{split}
J^k_1&=J^k_{11}+J^k_{12},\\
J^k_{11}&:=\int_{\substack{|x-y|\ge t,\\|y|\ge|t-|x||/2}}\int_0^t\frac{dsdy}{(1+2^k(|x-y|+s-t))^2(1+s)(1+2^{k-k_+}|y|)^4},\\
J^k_{12}&:=\int_{\substack{|x-y|\ge t,\\|y|\le|t-|x||/2}}\int_0^t\frac{dsdy}{(1+2^k(|x-y|+s-t))^2(1+s)(1+2^{k-k_+}|y|)^4}.
\end{split}
\end{equation}
For the integral $J^k_{11}$, there hold $(1+2^{k-k_+}|y|)^{-1}\ls2^{k_+-k}\w{t-|x|}^{-1}$ and
\begin{equation}\label{app:dpw10}
\begin{split}
J^k_{11}&\ls2^{k_+-k}\w{t-|x|}^{-1}\int_0^t\int_{\R^2}\frac{dsdy}{(1+2^ks)^{1/4}(1+s)(1+2^{k-k_+}|y|)^3}\\
&\ls2^{5(k_+-k)/4}\w{t-|x|}^{-1}\int_0^t\int_{\R^2}\frac{dsdy}{(1+s)^{5/4}(1+2^{k-k_+}|y|)^3}\\
&\ls2^{13(k_+-k)/4}\w{t-|x|}^{-1}.
\end{split}
\end{equation}
For the integral $J^k_{12}$, when $|x|<t$, one can find that $0<t-|x|\le|x-y|-|x|\le|y|\le(t-|x|)/2$ holds,
which yields $J^k_{12}=0$.
Therefore, only the case of $|x|\ge t$ is considered. At this time, there hold $|x-y|+s-t\ge|x|-|y|-t\ge(|x|-t)/2$ and
\begin{equation}\label{app:dpw11}
\begin{split}
J^k_{12}&\ls2^{k_+-k}\w{t-|x|}^{-1}\int_{|x-y|\ge t}\int_0^t\frac{(1+s)^{-1}dsdy}{(1+2^k(|x-y|+s-t))^{1/4}(1+2^{k-k_+}|y|)^4}\\
&\ls2^{5(k_+-k)/4}\w{t-|x|}^{-1}\int_{\R^2}\int_0^t\frac{dsdy}{(1+s)^{5/4}(1+2^{k-k_+}|y|)^4}\\
&\ls2^{13(k_+-k)/4}\w{t-|x|}^{-1}.
\end{split}
\end{equation}
Substituting \eqref{app:dpw10} and \eqref{app:dpw11} into \eqref{app:dpw9} derives
\begin{equation}\label{app:dpw12}
J^k_1\ls2^{13(k_+-k)/4}\w{t-|x|}^{-1}.
\end{equation}
Next, we focus on the estimate of $J^k_2$, which can be separated into the following four parts
\begin{equation}\label{app:dpw13}
\begin{split}
J^k_2&=J^k_{21}+J^k_{22}+J^k_{23}+J^k_{24},\\
J^k_{21}&:=\int_{\substack{|x-y|\le t,\\|y|\ge|t-|x||/2}}\int_0^{t-|x-y|}\frac{dsdy}{(1+2^k(t-s-|x-y|))^2(1+s)(1+2^{k-k_+}|y|)^4},\\
J^k_{22}&:=\int_{\substack{|x-y|\le t,\\|y|\ge|t-|x||/2}}\int_{t-|x-y|}^t\frac{dsdy}{(1+2^k(|x-y|+s-t))^2(1+s)(1+2^{k-k_+}|y|)^4},\\
J^k_{23}&:=\int_{\substack{|x-y|\le t,\\|y|\le|t-|x||/2}}\int_0^{(t-|x-y|)/2}\frac{dsdy}{(1+2^k(t-s-|x-y|))^2(1+s)(1+2^{k-k_+}|y|)^4},\\
J^k_{24}&:=\int_{\substack{|x-y|\le t,\\|y|\le|t-|x||/2}}\int_{(t-|x-y|)/2}^t\frac{dsdy}{(1+2^k|t-s-|x-y||)^2(1+s)(1+2^{k-k_+}|y|)^4}.
\end{split}
\end{equation}
Analogously to the estimate \eqref{app:dpw10} of $J^k_{11}$, we have
\begin{equation}\label{app:dpw14}
\begin{split}
J^k_{21}&\ls2^{k_+-k}\w{t-|x|}^{-1}\int_{\substack{|x-y|\le t,\\|y|\ge|t-|x||/2}}\int_0^{t-|x-y|}\frac{dsdy}{(1+s)(1+2^{k-k_+}|y|)^3}\\
&\ls2^{5(k_+-k)/4}\w{t-|x|}^{-1}\int_{|y|\ge|t-|x||/2}\frac{\ln(2+|y|)dy}{(1+|y|)^{1/4}(1+2^{k-k_+}|y|)^{5/2}}\\
&\ls2^{5(k_+-k)/4}\w{t-|x|}^{-1}\int_{\R^2}\frac{dy}{(1+2^{k-k_+}|y|)^{5/2}}\\
&\ls2^{13(k_+-k)/4}\w{t-|x|}^{-1}
\end{split}
\end{equation}
and
\begin{equation}\label{app:dpw15}
\begin{split}
&J^k_{22}\ls2^{k_+-k}\w{t-|x|}^{-1}\int_{|x-y|\le t}\int_{t-|x-y|}^t
\frac{(1+2^k(|x-y|+s-t))^{-1/4}dsdy}{(1+|x-y|+s-t)(1+2^{k-k_+}|y|)^{3}}\\
&\ls2^{5(k_+-k)/4}\w{t-|x|}^{-1}\int_{|x-y|\le t}\int_{t-|x-y|}^t\frac{dsdy}{(1+|x-y|+s-t)^{5/4}(1+2^{k-k_+}|y|)^3}\\
&\ls2^{5(k_+-k)/4}\w{t-|x|}^{-1}\int_{\R^2}\frac{dy}{(1+2^{k-k_+}|y|)^3}\\
&\ls2^{13(k_+-k)/4}\w{t-|x|}^{-1}.
\end{split}
\end{equation}
For $J^k_{23}$ and $J^k_{24}$, when $t<|x|$, due to $0<|x|-t\le|x|-|x-y|\le|y|\le(|x|-t)/2$,
then $J^k_{23}=J^k_{24}=0$.
When $t\ge|x|$, one has $t-|x-y|\ge t-|x|-|y|\ge(t-|x|)/2$.
From the expression of $J^k_{23}$, we know $t-s-|x-y|\ge(t-|x-y|)/2\ge s/2$ and
\begin{equation}\label{app:dpw16}
\begin{split}
J^k_{23}&\ls2^{k_+-k}\w{t-|x|}^{-1}\int_{\R^2}\int_0^{(t-|x-y|)/2}\frac{dsdy}{(1+2^ks)^{1/4}(1+s)(1+2^{k-k_+}|y|)^4}\\
&\ls2^{5(k_+-k)/4}\w{t-|x|}^{-1}\int_{\R^2}\int_0^{(t-|x-y|)/2}\frac{dsdy}{(1+s)^{5/4}(1+2^{k-k_+}|y|)^4}\\
&\ls2^{13(k_+-k)/4}\w{t-|x|}^{-1}.
\end{split}
\end{equation}
For the integral $J^k_{24}$, it is easy to obtain
\begin{equation}\label{app:dpw17}
\begin{split}
&\w{t-|x|}J^k_{24}\ls\int_{\R^2}\int_{(t-|x-y|)/2}^t\frac{dsdy}{(1+2^k|t-s-|x-y|)^{5/4}(1+2^{k-k_+}|y|)^4}\\
&\ls2^{5(k_+-k)/4}\int_{\R^2}\int_{(t-|x-y|)/2}^t\frac{dsdy}{(1+|t-|x-y|-s|)^{5/4}(1+2^{k-k_+}|y|)^4}\\
&\ls2^{13(k_+-k)/4}.
\end{split}
\end{equation}
Collecting \eqref{app:dpw7}, \eqref{app:dpw8} and \eqref{app:dpw12}-\eqref{app:dpw17} yields
\begin{equation*}
\begin{split}
|\p_jv|&\ls\w{x}^{-1/2}\w{t-|x|}^{-1}\sup_{(s,y)\in[0,t]\times\R^2}(1+s)|(1-\Delta)^3H(s,y)|\sum_{k\in\Z}2^{k/4-9k_+/4}\\
&\ls\w{x}^{-1/2}\w{t-|x|}^{-1}\sup_{(s,y)\in[0,t]\times\R^2}(1+s)|(1-\Delta)^3H(s,y)|.
\end{split}
\end{equation*}
This competes the proof of Lemma \ref{lem:impr:dpw}.
\end{proof}

\section{Pointwise estimates of the initial boundary value problem}\label{sect4}

\begin{lemma}({\bf Pointwise estimates in exterior domain})\label{lem:ibvp:pw}
Suppose that the obstacle $\cO$ is star-shaped, $\cK=\R^2\setminus\cO$ and $w(t,x)$ is the solution of the IBVP
\begin{equation*}
\left\{
\begin{aligned}
&\Box w=F(t,x),\qquad(t,x)\in[0,\infty)\times\cK,\\
&w|_{\p\cK}=0,\\
&(w,\p_tw)(0,x)=(w_0,w_1)(x),\quad x\in\cK,
\end{aligned}
\right.
\end{equation*}
where $(w_0,w_1)$ admits a compact support and $\supp_x F(t,x)\subset\{x: |x|\le t+M_0\}$.
Then we have that for any $\mu,\nu\in(0,1/2)$,
\begin{equation}\label{InLW:dpw}
\begin{split}
&\w{x}^{1/2}\w{t-|x|}^{1/2-\mu}|\p w|
\ls\|(w_0,w_1)\|_{H^4(\cK)}+\sum_{|a|\le3}\sup_{s\in[0,t]}\w{s}^{1/2}\|\p^aF(s)\|_{L^2(\cK_4)}\\
&+\sum_{|a|\le3}\sup_{y\in\cK}|\p^aF(0,y)|
+\sum_{|a|\le4}\sup_{(s,y)\in[0,t]\times(\overline{\R^2\setminus\cK_2})}\w{y}^{1/2}\cW_{1+\mu+\nu,1}(s,y)|Z^aF(s,y)|
\end{split}
\end{equation}
and
\begin{equation}\label{InLW:pw}
\begin{split}
&\frac{\w{t+|x|}^{1/2}\w{t-|x|}^{\mu}}{\ln^2(2+t+|x|)}|w|
\ls\|(w_0,w_1)\|_{H^5(\cK)}+\sum_{|a|\le4}\sup_{s\in[0,t]}\w{s}^{1/2+\mu}\|\p^aF(s)\|_{L^2(\cK_4)}\\
&+\sum_{|a|\le4}\sup_{y\in\cK}|\p^aF(0,y)|
+\sum_{|a|\le5}\sup_{(s,y)\in[0,t]\times(\overline{\R^2\setminus\cK_2})}\w{y}^{1/2}\cW_{1+\mu,1+\nu}(s,y)|\p^aF(s,y)|.
\end{split}
\end{equation}
In addition, let $w^{div}(t,x)$ solve
\begin{equation}\label{ibvp:div}
\left\{
\begin{aligned}
&\Box w^{div}=F(t,x)=G+\p_{\theta}G^{\theta}+\sum_{\alpha=0}^2\p_{\alpha}G^{\alpha},\qquad (t,x)\in[0,+\infty)\times\cK,\\
&w^{div}|_{\p\cK}=0,\\
&(w^{div},\p_tw^{div})(0,x)=(w_0,w_1)(x),\qquad\qquad\quad x\in\cK,
\end{aligned}
\right.
\end{equation}
where $(w_0,w_1)$ has a compact support and $\supp_x(G,G^{\theta},G^{\alpha})\subset\{x: |x|\le t+M_0\}$, then it holds that
\begin{equation}\label{InLW:dpw:div}
\begin{split}
&\w{x}^{1/2}\w{t-|x|}|\p w^{div}|
\ls\|(w_0,w_1)\|_{H^9(\cK)}+\sum_{|a|\le8}\sup_{y\in\cK}|\p^aF(0,y)|+\sup_{y\in\cK}|G^0(0,y)|\\
&+\sum_{|a|\le8}\sup_{s\in[0,t]}\w{s}\|\p^aF(s)\|_{L^2(\cK_4)}+\sum_{i=1,2}\sup_{s\in[0,t]}\w{s}^{3/2+\mu}\|G^i(s)\|_{L^\infty(\cK_3)}\\
&+\sup_{(s,y)\in[0,t]\times(\overline{\R^2\setminus\cK_2})}\w{y}^{1/2}\cW_{3/2+\mu,1}(s,y)|G(s,y)|\\
&+\sum_{|a|\le9}\sum_{\alpha=0}^2\sup_{(s,y)\in[0,t]\times(\overline{\R^2\setminus\cK_2})}\w{y}^{1/2}\cW_{1+\mu+\nu,1}(s,y)
|Z^a(G,Z^{\le1}G^\theta,Z^{\le1}G^\alpha)(s,y)|.
\end{split}
\end{equation}
\end{lemma}
\begin{proof}
At first, we prove \eqref{InLW:dpw}.
Let $w_1^b$ and $w_2^b$ be the solutions of
\begin{equation}\label{dpw:pf1}
\left\{
\begin{aligned}
&\Box w_1^b:=(1-\chi_{[2,3]}(x))F(t,x),\\
&\Box w_2^b:=\chi_{[2,3]}(x)F(t,x),\\
&w_1^b|_{\p\cK}=w_2^b|_{\p\cK}=0,\\
&(w_1^b,\p_tw_1^b)(0,x)=(w_0,w_1)(x),\quad (w_2^b,\p_tw_2^b)(0,x)=(0,0).
\end{aligned}
\right.
\end{equation}
Then it follows from the uniqueness of smooth solution to the IBVP that $w^{div}=w_1^b+w_2^b$.
For the estimate of $w_1^b$ in $\cK_R$ with $R>4$, by $\supp_{x}\Box w_1^b\subset\{x: |x|\le3\}$ and \eqref{loc:decay} with $m=3$,
one has that
\begin{equation}\label{dpw:pf2}
\begin{split}
\w{t}^{\rho}\|\p w_1^b\|_{L^\infty(\cK_R)}&\ls\w{t}^{\rho}\|\p w_1^b\|_{H^2(\cK_{R+1})}\\
&\ls\|(w_0,w_1)\|_{H^3(\cK)}+\sum_{|b|\le2}\sup_{s\in[0,t]}\w{s}^{\rho}\|\p^b[(1-\chi_{[2,3]})F(s)]\|_{L^2(\cK)}\\
&\ls\|(w_0,w_1)\|_{H^3(\cK)}+\sum_{|b|\le2}\sup_{s\in[0,t]}\w{s}^{\rho}\|\p^bF(s)\|_{L^2(\cK_3)},
\end{split}
\end{equation}
where $\rho\in(0,1]$.

We now treat $w_1^b$ in the region $|x|\ge R$.
To this end, let $w_1^c=\chi_{[3,4]}w_1^b$ be a function on $\R^2$.
Then $w_1^c$ solves
\begin{equation*}
\left\{
\begin{aligned}
&\Box w_1^c=-[\Delta,\chi_{[3,4]}]w_1^b,\\
&(w_1^c,\p_tw_1^c)(0,x)=(\chi_{[3,4]}w_0,\chi_{[3,4]}w_1).
\end{aligned}
\right.
\end{equation*}
It follows from $\supp_x[\Delta,\chi_{[3,4]}]w_1^b\subset\{x: |x|\le4\}$, \eqref{dpw:ivp2} with $H=0$ and \eqref{dpw:loc} that
\begin{equation}\label{dpw:pf3}
\begin{split}
&\w{x}^{1/2}\w{t-|x|}^{\rho-\delta}|\p w_1^c|\\
&\quad\ls\|(w_0,w_1)\|_{W^{2,\infty}(\cK)}+\sum_{|a|\le1}\sup_{(s,y)\in[0,t]\times\R^2}\w{s}^{\rho}|\p^a[\Delta,\chi_{[3,4]}]w_1^b(s,y)|\\
&\quad\ls\|(w_0,w_1)\|_{H^4(\cK)}+\sum_{|a|\le1}\sup_{s\in[0,t]}\w{s}^{\rho}\|\p^a[\Delta,\chi_{[3,4]}]w_1^b(s,\cdot)\|_{H^2(\R^2)}\\
&\quad\ls\|(w_0,w_1)\|_{H^4(\cK)}+\sum_{|a|\le3}\sup_{s\in[0,t]}\w{s}^{\rho}\|\p^aF(s)\|_{L^2(\cK_4)},
\end{split}
\end{equation}
where $0<\delta<\rho\le1/2$ and the last two inequalities have used the Sobolev embedding and \eqref{loc:decay} with $m=4$.

Since $\p w_1^b=\p w_1^c$ holds in the region $|x|\ge4$, then collecting \eqref{dpw:pf2} with $R>4$ and \eqref{dpw:pf3} yields that for $x\in\cK$,
\begin{equation}\label{dpw:pf4}
\w{x}^{1/2}\w{t-|x|}^{\rho-\delta}|\p w_1^b|
\ls\|(w_0,w_1)\|_{H^4(\cK)}+\sum_{|a|\le3}\sup_{s\in[0,t]}\w{s}^{\rho}\|\p^aF(s)\|_{L^2(\cK_4)}.
\end{equation}
Next, we deal with $w_2^b$.
Let $w_2^c$ be the solution of the Cauchy problem
\begin{equation}\label{dpw:pf5}
\Box w_2^c=\left\{
\begin{aligned}
&\chi_{[2,3]}(x)F(t,x),\qquad&& x\in\cK=\R^2\setminus\cO,\\
&0,&& x\in\overline{\cO}
\end{aligned}
\right.
\end{equation}
with $(w_2^c, \p_tw_2^c)=(0,0)$.
Set $v_2^b=\chi_{[1,2]}(x)w_2^c$ on $\cK$ and then $v_2^b$ is the solution of the IBVP
\begin{equation}\label{dpw:pf6}
\left\{
\begin{aligned}
&\Box v_2^b=\chi_{[1,2]}\Box w_2^c-[\Delta,\chi_{[1,2]}]w_2^c=\chi_{[2,3]}(x)F(t,x)-[\Delta,\chi_{[1,2]}]w_2^c,\\
&v_2^b|_{\p\cK}=0,\quad (v_2^b,\p_tv_2^b)(0,x)=(0,0).
\end{aligned}
\right.
\end{equation}
In addition, let $\tilde v_2^b$ be the solution of the IBVP
\begin{equation}\label{dpw:pf7}
\Box\tilde v_2^b=[\Delta,\chi_{[1,2]}]w_2^c,
\quad\tilde v_2^b|_{\p\cK}=0,\quad (\tilde v_2^b,\p_t\tilde v_2^b)(0,x)=(0,0).
\end{equation}
Therefore, it follows from \eqref{dpw:pf1}, \eqref{dpw:pf6} and \eqref{dpw:pf7} that $w_2^b=\chi_{[1,2]}(x)w_2^c+\tilde v_2^b$.

Due to $\supp_x\Box\tilde v_2^b\subset\{x: |x|\le2\}$, similarly to the treatment of \eqref{dpw:pf4} for $w_1^b$,
we have that
\begin{equation}\label{dpw:pf8}
\begin{split}
\w{x}^{1/2}\w{t-|x|}^{\rho-\delta}|\p \tilde v_2^b|
&\ls\sum_{|a|\le3}\sup_{s\in[0,t]}\w{s}^{\rho}\|\p^a[\Delta,\chi_{[1,2]}]w_2^c(s)\|_{L^2(\cK_4)}\\
&\ls\sum_{|a|\le4}\sup_{s\in[0,t]}\w{s}^{\rho}\|\p^a w_2^c(s)\|_{L^\infty(|y|\le2)}.
\end{split}
\end{equation}
To estimate $w_2^c$,
applying \eqref{pw:ivp1} to \eqref{dpw:pf5} derives
\begin{equation}\label{dpw:pf9}
\begin{split}
&\sum_{|a|\le4}\w{t}^{1/2+\mu}\|\p^a w_2^c(t)\|_{L^\infty(\cK_2)}\\
&\ls\sum_{|a|\le3}\sup_{y\in\cK}|\p^aF(0,y)|+\sum_{|a|\le4}\sup_{(s,y)\in[0,t]\times\R^2}\w{y}^{1/2}\cW_{1+\mu,1+\nu}(s,y)|\p^a(\chi_{[2,3]}F(s,y))|\\
&\ls\sum_{|a|\le3}\sup_{y\in\cK}|\p^aF(0,y)|+\sum_{|a|\le4}\sup_{(s,y)\in[0,t]
\times(\overline{\R^2\setminus\cK_2})}\w{y}^{1/2}\cW_{1+\mu+\nu,1}(s,y)|\p^aF(s,y)|,
\end{split}
\end{equation}
where the term $\ds\sup_{y\in\cK}|\p^aF(0,y)|$ in \eqref{dpw:pf9} comes from the initial data of $\p^aw_2^c$.
On the other hand, it follows from \eqref{dpw:ivp2} that
\begin{equation}\label{dpw:pf10}
\begin{split}
\w{x}^{1/2}\w{t-|x|}^{1+\mu}|\p w_2^c|
&\ls\sum_{|a|\le1}\sup_{(s,y)\in[0,t]\times\R^2}\w{y}^{1/2}\cW_{1+\mu+\nu,1}(s,y)|Z^a(\chi_{[2,3]}F)(s,y)|\\
&\ls\sum_{|a|\le1}\sup_{(s,y)\in[0,t]\times(\overline{\R^2\setminus\cK_2})}\w{y}^{1/2}\cW_{1+\mu+\nu,1}(s,y)|Z^aF(s,y)|.
\end{split}
\end{equation}
Collecting \eqref{dpw:pf4}, \eqref{dpw:pf8}-\eqref{dpw:pf10} with the fact of
\begin{equation}\label{dpw:pf11}
\p w_2^b=\chi_{[1,2]}(x)\p w_2^c+(\p\chi_{[1,2]}(x))w_2^c+\p\tilde v_2^b
\end{equation}
yields \eqref{InLW:dpw}.

In addition, the proof of \eqref{InLW:pw} is similar to that of \eqref{InLW:dpw} with \eqref{pw:ivp1} and \eqref{impr:pw:Cauchy}
instead of \eqref{dpw:ivp2} and \eqref{dpw:loc}, respectively.

Finally, we turn to the proof of \eqref{InLW:dpw:div}.
Analogously to the proofs of \eqref{dpw:pf4} and \eqref{dpw:pf8} with \eqref{impr:pw:Cauchy} instead of \eqref{dpw:loc}, one can see that
\begin{equation}\label{dpw:pf12}
\begin{split}
\w{x}^{1/2}\w{t-|x|}|\p w_1^b|
&\ls\|(w_0,w_1)\|_{H^9(\cK)}+\sum_{|a|\le8}\sup_{s\in[0,t]}\w{s}\|\p^aF(s)\|_{L^2(\cK_4)},\\
\w{x}^{1/2}\w{t-|x|}|\p\tilde v_2^b|
&\ls\sum_{|a|\le9}\sup_{s\in[0,t]}\w{s}\|\p^aw_2^c(s,y)\|_{L^\infty(|y|\le2)}.
\end{split}
\end{equation}
On the other hand, by using \eqref{dpw:ivp2} to \eqref{dpw:pf5}, one obtains that for integer $l$,
\begin{equation}\label{dpw:pf13}
\begin{split}
&\w{x}^{1/2}\w{t-|x|}^{1+\mu}\sum_{|a|\le l}|\p\p^a w_2^c|
\ls\sum_{|a|\le l}\sup_{y\in\cK}|\p^aF(0,y)|\\
&\quad+\sum_{|a|\le1+l}\sup_{(s,y)\in[0,t]\times(\overline{\R^2\setminus\cK_2})}\w{y}^{1/2}\cW_{1+\mu+\nu,1}(s,y)|Z^a(G,ZG^{\alpha},ZG^{\theta})(s,y)|.
\end{split}
\end{equation}
Next, we focus on the estimate of $w_2^c(t,x)$ itself in the region $|x|\le2$.
Note that $w_2^c(t,x)$ is the solution to the inhomogeneous Cauchy problem
\begin{equation}\label{dpw:pf14}
\Box w_2^c=\p_{\theta}(\chi_{[2,3]}(x)G^{\theta})+\sum_{\alpha=0}^2\p_{\alpha}(\chi_{[2,3]}(x)G^{\alpha})+\chi_{[2,3]}(x)G
-\sum_{i=1,2}G^i\p_i(\chi_{[2,3]}(x)).
\end{equation}
Then it follows from $[\Box,\p_{\theta}]=0$ that $w_2^c=\p_{\theta}w^{\theta}+\ds\sum_{\alpha=0}^2\p_{\alpha}w^{\alpha}+w^r$,
where $w^{\theta}$, $w^{\alpha}$, $w^r$ are the solutions of
\begin{equation}\label{dpw:pf15}
\begin{split}
&\Box w^{\Xi}=\chi_{[2,3]}(x)G^{\Xi},\qquad (w^{\Xi},\p_tw^{\Xi})(0,x)=(0,0),\qquad \Xi\in\{\theta,0,1,2\},\\
&\Box w^r=\chi_{[2,3]}(x)G-\sum_{i=1,2}G^i\p_i(\chi_{[2,3]}(x)),(w^r,\p_tw^r)(0,x)=(0,-\chi_{[2,3]}(x)G^0(0,x)),
\end{split}
\end{equation}
respectively.
Applying \eqref{pw:ivp1}, \eqref{pw:ivp2} and \eqref{dpw:ivp2} to $w^r$ and $w^{\Xi}$ with $\Xi\in\{\theta,0,1,2\}$ yields
\begin{equation}\label{dpw:pf16}
\begin{split}
&\w{t+|x|}^{1/2}\w{t-|x|}^{1/2-\mu}|w^r|\ls\sup_{(s,y)\in[0,t]\times(\overline{\R^2\setminus\cK_2})}\w{y}^{1/2}\cW_{3/2-\mu,1+\nu}(s,y)|G(s,y)|\\
&\qquad\qquad+\sup_{y\in\cK}|G^0(0,y)|+\sum_{i=1,2}\sup_{(s,y)\in[0,t]\times(\overline{\cK_3\setminus\cK_2})}\w{s}^{3/2-\mu}|G^i(s,y)|,\\
&\w{t+|x|}^{1/2}\w{t-|x|}^{1/2}|w^r|\ls\sup_{(s,y)\in[0,t]\times(\overline{\R^2\setminus\cK_2})}\w{y}^{1/2}\cW_{3/2+\mu,1}(s,y)|G(s,y)|\\
&\qquad\qquad+\sup_{y\in\cK}|G^0(0,y)|+\sum_{i=1,2}\sup_{(s,y)\in[0,t]\times(\overline{\cK_3\setminus\cK_2})}\w{s}^{3/2+\mu}|G^i(s,y)|,\\
&\w{x}^{1/2}\w{t-|x|}^{1+\mu}|\p w^{\Xi}|\ls\sup_{(s,y)\in[0,t]\times(\overline{\R^2\setminus\cK_2})}\w{y}^{1/2}\cW_{1+\mu+\nu,1}(s,y)|Z^{\le1}G^{\Xi}(s,y)|.
\end{split}
\end{equation}
Note that $|\p_{\theta}w^{\theta}|=|\Omega w^{\theta}|\ls|\p w^{\theta}|$ for $x\in\supp(\p\chi_{[1,2]})$ in \eqref{dpw:pf11}.
Then, collecting \eqref{dpw:pf11}-\eqref{dpw:pf16} derives \eqref{InLW:dpw:div}.
\end{proof}

\begin{remark}\label{rmk4-1}
The pointwise estimate \eqref{thm1:decay:c} (especially the appearance of the decay factor $\w{t-|x|}^{0.001-1/2}$)
is essential in the proof of Theorem \ref{thm1}, which will be derived by the estimate \eqref{InLW:pw}.
Meanwhile, \eqref{InLW:pw} is proved in terms of \eqref{impr:pw:Cauchy}.
It is pointed out that the author in \cite[Theorem 1.1]{Kubo15} has proved the following estimate
\begin{equation}\label{Kubo15cpaa}
\w{t+|x|}^{1/2}(\min\{\w{x},\w{t-|x|}\})^{\mu}|v|
\ls\sum_{|a|\le1}\sup_{(s,y)\in[0,t]\times\R^2}\w{s}^{1/2+\mu}|\p^aH(s,y)|,
\end{equation}
where $\mu\in(0,1/2)$.
If utilizing \eqref{Kubo15cpaa} instead of \eqref{impr:pw:Cauchy}, then the pointwise estimate \eqref{thm1:decay:c}
is correspondingly changed into
\begin{equation}\label{YHC-32}
\sum_{|a|\le N}|Z^au|\le C\ve\w{t+|x|}^{-1/2}(\min\{\w{x},\w{t-|x|}\})^{-\mu}.
\end{equation}
For $|x|\le t/2$, \eqref{YHC-32} gives the decay factor $\w{t}^{-1/2}\w{x}^{-\mu}$.
However, our pointwise estimate \eqref{thm1:decay:c} yields a better and important decay rate $\w{t}^{0.002-1}$.
\end{remark}

\begin{remark}\label{rmk4-2}
The precise estimate \eqref{InLW:dpw:div} is crucial in the establishment of \eqref{thm1:decay:a}.
In view of \eqref{dpw:pf16}, the estimate of the derivatives of solution are better than that of solution itself.
With the structure of the divergence form \eqref{ibvp:div}, comparing with the remaining part $G$, the weights
on the divergence parts $G^{\Xi}$, $\Xi=\theta,0,1,2$ are obviously weaker, see the third and fourth lines of \eqref{InLW:dpw:div}.
Thus, the divergence form \eqref{eqn:good:div} will play a key role in the proof of Theorem \ref{thm1}.
If one still applies \eqref{eqn:good} instead of \eqref{eqn:good:div},
by taking $G(t,x)=|x|^{-1}\Omega\p_{\mu}Z^bu\p_{\nu}Z^cu$ in the third line of \eqref{InLW:dpw:div} and
utilizing the second line of \eqref{Q12:null:structure} and the decay estimate \eqref{thm1:decay:a},
then  $|G(t,x)|\ls\w{x}^{-2}\w{t-|x|}^{-2}$ is obtained.
Due to $\mu>0$ in the third line of \eqref{InLW:dpw:div}, $\w{y}^{1/2}\cW_{3/2+\mu,1}(s,y)|G(s,y)|$
is unbounded in time near the light cone $|y|=s$.
On the other hand, for $\mu=0$ in the third line of \eqref{InLW:dpw:div},
applying \eqref{pw:crit} instead of \eqref{pw:ivp2} in the proof of \eqref{InLW:dpw:div} yields
\begin{equation}\label{YHC-33}
\begin{split}
&\w{x}^{1/2}\w{t-|x|}|\p w^{div}|\ls\cdots\cdots\\
&\quad+\ln(2+t)\sup_{(s,y)\in[0,t]\times(\overline{\R^2\setminus\cK_2})}\w{y}^{1/2}\cW_{3/2,1+\nu}(s,y)|G(s,y)|+\cdots,
\end{split}
\end{equation}
where $\nu>0$.
Due to \eqref{YHC-33}, one can only achieve
\begin{equation}\label{YHC-34}
\sum_{|a|\le N}|\p Z^au|\le C\ve\w{x}^{-1/2}\w{t-|x|}^{-1}\ln(2+t)
\end{equation}
rather than \eqref{thm1:decay:a}.
From \eqref{YHC-34}, one knows that it is far away to solve the global solution problem \eqref{QWE} as pointed on
page 320 of \cite{Kubo13} due to the appearance of the large factor $\ln(2+t)$.
\end{remark}

\begin{lemma}({\bf $L^{\infty}$ estimates near the boundary})\label{lem:ibvp:loc}
Suppose that the obstacle $\cO$ is star-shaped, $\cK=\R^2\setminus\cO$ and $w$ solves
\begin{equation*}
\left\{
\begin{aligned}
&\Box w=F(t,x),\qquad(t,x)\in(0,\infty)\times\cK,\\
&w|_{\p\cK}=0,\\
&(w,\p_tw)(0,x)=(w_0,w_1)(x),\quad x\in\cK,
\end{aligned}
\right.
\end{equation*}
where $(w_0,w_1)$ has a compact support and $\supp_x F(t,x)\subset\{x: |x|\le t+M_0\}$.
Then one has that for any $\mu,\nu\in(0,1/2)$ and $R>1$,
\begin{equation}\label{loc:dt:sharp}
\begin{split}
&\w{t}\|\p_tw\|_{L^\infty(\cK_R)}
\ls\|(w_0,w_1)\|_{H^3(\cK)}+\sum_{|a|\le2}\sup_{s\in[0,t]}\w{s}\|\p^aF(s)\|_{L^2(\cK_3)}\\
&\quad+\sum_{|a|\le2}\sup_{y\in\cK}|\p^aF(0,y)|+\sum_{|a|\le3}\sup_{(s,y)\in[0,t]\times\cK}\w{y}^{1/2}\cW_{1+\mu+\nu,1}(s,y)|Z^aF(s,y)|.
\end{split}
\end{equation}
In addition, let $w^{div}(t,x)$ be the solution of
\begin{equation*}
\left\{
\begin{aligned}
&\Box w^{div}=F(t,x)=G+\p_{\theta}G^{\theta}+\sum_{\alpha=0}^2\p_{\alpha}G^{\alpha},\quad(t,x)\in(0,\infty)\times\cK,\\
&w^{div}|_{\p\cK}=0,\\
&(w^{div},\p_tw^{div})(0,x)=(w_0,w_1)(x),\qquad\qquad x\in\cK,
\end{aligned}
\right.
\end{equation*}
where $(w_0,w_1)$ has a compact support and $\supp_x(G,G^{\alpha},G^{\theta})\subset\{x: |x|\le t+M_0\}$, then
\begin{equation}\label{loc:divform}
\begin{split}
&\w{t}\|w^{div}\|_{L^\infty(\cK_R)}
\ls\|(w_0,w_1)\|_{H^2(\cK)}+\sum_{|a|\le1}\sup_{y\in\cK}|\p^aF(0,y)|+\sup_{y\in\cK}|G^0(0,y)|\\
&\quad+\sum_{|a|\le1}\sup_{s\in[0,t]}\w{s}\|\p^aF(s)\|_{L^2(\cK_3)}+\sum_{i=1,2}\sup_{s\in[0,t]}\w{s}^{3/2+\mu}\|G^i(s)\|_{L^\infty(\cK_3)}\\
&\quad+\sup_{(s,y)\in[0,t]\times(\overline{\R^2\setminus\cK_2})}\w{y}^{1/2}\cW_{3/2+\mu,1}(s,y)|G(s,y)|\\
&\quad+\sum_{|a|\le2}\sup_{(s,y)\in[0,t]\times(\overline{\R^2\setminus\cK_2})}\w{y}^{1/2}\cW_{1+\mu+\nu,1}(s,y)
|Z^a(G,Z^{\le1}G^\theta,Z^{\le1}G^\alpha)(s,y)|.
\end{split}
\end{equation}

\end{lemma}
\begin{proof}
At first, we prove \eqref{loc:divform}.
Let $w_1^b$, $w_2^b$, $w_2^c$ and $\tilde v_2^b$ be defined by \eqref{dpw:pf1}, \eqref{dpw:pf5} and \eqref{dpw:pf7} in Lemma \ref{lem:ibvp:pw}, respectively.
Analogously to \eqref{dpw:pf2}, we can arrive at
\begin{equation}\label{loc:divform:pf1}
\begin{split}
\w{t}\|w_1^b\|_{L^\infty(\cK_R)}&\ls\w{t}\|w_1^b\|_{H^2(\cK_{R+1})}\\
&\ls\|(w_0,w_1)\|_{H^2(\cK)}+\sum_{|b|\le1}\sup_{s\in[0,t]}\w{s}\|\p^b[(1-\chi_{[2,3]})F(s)]\|_{L^2(\cK)}\\
&\ls\|(w_0,w_1)\|_{H^2(\cK)}+\sum_{|b|\le1}\sup_{s\in[0,t]}\w{s}\|\p^bF(s)\|_{L^2(\cK_3)}.
\end{split}
\end{equation}
Next, we treat $w_2^b=\chi_{[1,2]}(x)w_2^c+\tilde v_2^b$ in the region $|x|\le R$.
Similarly to \eqref{dpw:pf2} and \eqref{loc:divform:pf1}, applying \eqref{loc:decay} to \eqref{dpw:pf7} with $m=2$ yields
\begin{equation}\label{loc:divform:pf2}
\w{t}\|\tilde v_2^b\|_{L^\infty(\cK_R)}\ls\sum_{|a|\le2}\sup_{s\in[0,t]}\w{s}\|\p^aw_2^c(s)\|_{L^2(\cK_2)}.
\end{equation}

Note that $\ds w_2^c=\p_{\theta}w^{\theta}+\sum_{\alpha=0}^2\p_{\alpha}w^{\alpha}+w^r$ with $w^{\theta}$, $w^{\alpha}$, $w^r$  defined by \eqref{dpw:pf15}.
Combining \eqref{dpw:pf13} and \eqref{dpw:pf16} with \eqref{loc:divform:pf1} and \eqref{loc:divform:pf2} yields \eqref{loc:divform}.

We now give the proof  on the estimate \eqref{loc:dt:sharp}.
It follows from \eqref{dpw:pf2} with $\p=\p_t$ that
\begin{equation}\label{loc:dt:pf1}
\w{t}\|\p_tw_1^b\|_{L^\infty(\cK_R)}
\ls\|(w_0,w_1)\|_{H^3(\cK)}+\sum_{|b|\le2}\sup_{s\in[0,t]}\w{s}\|\p^bF(s)\|_{L^2(\cK_3)}.
\end{equation}
Note that $\p_tw_2^b=\chi_{[1,2]}(x)\p_tw_2^c+\p_t\tilde v_2^b$, and $\phi=\p_t\tilde v_2^b$ satisfies
\begin{equation}\label{loc:dt:pf2}
\Box\phi=[\Delta,\chi_{[1,2]}]\p_tw_2^c, \quad \phi|_{\p\cK}=0,\quad (\phi,\p_t\phi)_{t=0}=(0,0).
\end{equation}
Similarly to \eqref{loc:divform:pf1}, we can achieve
\begin{equation}\label{loc:dt:pf3}
\w{t}\|\p_t\tilde v_2^b\|_{L^\infty(\cK_R)}\ls\w{t}\|\phi\|_{H^2(\cK_{R+1})}
\ls\sum_{|a|\le2}\sup_{s\in[0,t]}\w{s}\|\p^a\p_tw_2^c(s)\|_{L^\infty(\cK_2)}.
\end{equation}
On the other hand, analogously to \eqref{dpw:pf13}, one has
\begin{equation}\label{loc:dt:pf4}
\begin{split}
&\w{x}^{1/2}\w{t-|x|}^{1+\mu}\sum_{|a|\le2}|\p_t\p^aw_2^c|
\ls\sum_{|a|\le2}\sup_{y\in\cK}|\p^aF(0,y)|\\
&\qquad+\sum_{|a|\le3}\sup_{(s,y)\in[0,t]\times(\overline{\R^2\setminus\cK_2})}\w{y}^{1/2}\cW_{1+\mu+\nu,1}(s,y)|Z^aF(s,y)|.
\end{split}
\end{equation}
Collecting \eqref{loc:dt:pf1}-\eqref{loc:dt:pf4} yields \eqref{loc:dt:sharp}.
\end{proof}

\section{Energy estimates}\label{sect5}

\subsection{Bootstrap assumptions}

We make the following bootstrap assumptions
\begin{align}
&\sum_{|a|\le N+1}|\p Z^au|\le\ve_1\w{x}^{-1/2}\w{t-|x|}^{-1},\label{BA1}\\
&\sum_{|a|\le N+1}|\bar\p Z^au|\le\ve_1\w{x}^{-1/2}\w{t+|x|}^{\ve_2-1},\label{BA2}\\
&\sum_{|a|\le N}|Z^au|\le\ve_1\w{t+|x|}^{\ve_2-1/2}\w{t-|x|}^{\ve_2-1/2},\label{BA3}
\end{align}
where $\ve_1\in(\ve,1)$ will be determined later and $\ve_2=10^{-3}$.

Note that due to $\supp_x(u_0,u_1)\subset\{x: |x|\le M_0\}$,
the solution $u$ of problem \eqref{QWE} is supported on the space variable $x$ in $\{x\in\cK:|x|\le t+M_0\}$.

\subsection{Energy estimates}

\begin{lemma}
Under the assumptions of Theorem \ref{thm1}, let $u$ be the solution of \eqref{QWE} and suppose that \eqref{BA1} and \eqref{BA2} hold.
Then there is a positive constant $C_0$ such that
\begin{equation}\label{energy:time}
\sum_{j\le2N}\|\p\p_t^ju\|_{L^2(\cK)}\ls\ve(1+t)^{C_0\ve_1},
\end{equation}
where $\ve_1>0$ is small enough. Especially,
\begin{equation}\label{energy:time'}
\sum_{j\le2N}\|\p\p_t^ju\|_{L^2(\cK)}\ls\ve(1+t)^{\ve_2}.
\end{equation}
\end{lemma}
\begin{proof}
Note that \eqref{eqn:high} can be written as
\begin{equation}\label{energy:time1}
\Box Z^au=\sum_{\alpha,\beta,\gamma=0}^2\Big\{Q^{\alpha\beta\gamma}\p^2_{\alpha\beta}Z^au\p_{\gamma}u
+\sum_{\substack{b+c\le a,\\b<a}}Q_{abc}^{\alpha\beta\gamma}\p^2_{\alpha\beta}Z^bu\p_{\gamma}Z^cu\Big\}.
\end{equation}
Through multiplying \eqref{energy:time1} by $e^q\p_tZ^au$ with $q=\arctan(|x|-t)$
(such a ghost weight $e^q$ was firstly introduced in \cite{Alinhac01a}) and taking
direct computation, we have
\begin{equation}\label{energy:time2}
\begin{split}
&\quad\;\frac12\p_t[e^q(|\p_tZ^au|^2+|\nabla Z^au|^2)]-\dive(e^q\p_tZ^au\nabla Z^au)
+\frac{e^q}{2\w{t-|x|}^2}|\bar\p Z^au|^2\\
&=\sum_{\alpha,\beta,\gamma=0}^2\Big\{Q^{\alpha\beta\gamma}e^q\p_tZ^au\p^2_{\alpha\beta}Z^au\p_{\gamma}u
+\sum_{\substack{b+c\le a,\\b<a}}Q_{abc}^{\alpha\beta\gamma}e^q\p_tZ^au\p^2_{\alpha\beta}Z^bu\p_{\gamma}Z^cu\Big\}.
\end{split}
\end{equation}
For the first term in the second line of \eqref{energy:time2}, it follows from direct computation with $Q^{\alpha\beta\gamma}=Q^{\beta\alpha\gamma}$ that
\begin{equation}\label{energy:time3}
\begin{split}
&\quad\; Q^{\alpha\beta\gamma}e^q\p_tZ^au\p^2_{\alpha\beta}Z^au\p_{\gamma}u\\
&=\p_{\alpha}(Q^{\alpha\beta\gamma}e^q\p_tZ^au\p_{\beta}Z^au\p_{\gamma}u)
-Q^{\alpha\beta\gamma}e^q\p_tZ^au\p_{\alpha}q\p_{\beta}Z^au\p_{\gamma}u\\
&\quad-Q^{\alpha\beta\gamma}e^q\p_tZ^au\p_{\beta}Z^au\p^2_{\alpha\gamma}u
-Q^{\alpha\beta\gamma}e^q\p_t\p_{\alpha}Z^au\p_{\beta}Z^au\p_{\gamma}u\\
&=\p_{\alpha}(Q^{\alpha\beta\gamma}e^q\p_tZ^au\p_{\beta}Z^au\p_{\gamma}u)
-Q^{\alpha\beta\gamma}e^q\p_tZ^au\p_{\alpha}q\p_{\beta}Z^au\p_{\gamma}u\\
&\quad-Q^{\alpha\beta\gamma}e^q\p_tZ^au\p_{\beta}Z^au\p^2_{\alpha\gamma}u
-\p_t(\frac{1}{2}Q^{\alpha\beta\gamma}e^q\p_{\alpha}Z^au\p_{\beta}Z^au\p_{\gamma}u)\\
&\quad+\frac{1}{2}Q^{\alpha\beta\gamma}e^q\p_tq\p_{\alpha}Z^au\p_{\beta}Z^au\p_{\gamma}u
+\frac{1}{2}Q^{\alpha\beta\gamma}e^q\p_{\alpha}Z^au\p_{\beta}Z^au\p^2_{0\gamma}u,
\end{split}
\end{equation}
where the summation $\ds\sum_{\alpha,\beta,\gamma=0}^2$ is omitted in \eqref{energy:time3}.

Choose $Z^a=\p_t^j$ with $j=|a|$ in \eqref{energy:time3} and the general notation $Z^a$ is still used in the remaining part.
By integrating \eqref{energy:time2} and \eqref{energy:time3} over $[0,t]\times\cK$
with the boundary conditions $\p_t^lu|_{\p\cK}=0$ for any integer $l\ge0$, one has
\begin{equation}\label{energy:time4}
\begin{split}
&\quad\;\|\p Z^au(t)\|^2_{L^2(\cK)}+\int_0^t\int_{\cK}\frac{|\bar\p Z^au(s,x)|^2}{\w{s-|x|}^2}dxds\\
&\ls\|\p Z^au(0)\|^2_{L^2(\cK)}+\|\p u(0)\|_{L_x^\infty}\|\p Z^au(0)\|^2_{L^2(\cK)}\\
&\quad+\|\p u(t)\|_{L_x^\infty}\|\p Z^au(t)\|^2_{L^2(\cK)}
+\int_0^t\int_{\cK}(\sum_{\substack{|b|+|c|\le|a|,\\|b|<|a|}}|I^{abc}_1|+|I^a_2|)dxds,
\end{split}
\end{equation}
where
\begin{equation}\label{energy:time5}
\begin{split}
I^{abc}_1&:=\sum_{\alpha,\beta,\gamma=0}^2Q_{abc}^{\alpha\beta\gamma}e^q\p_tZ^au\p^2_{\alpha\beta}Z^bu\p_{\gamma}Z^cu,\\
I^a_2&:=\sum_{\alpha,\beta,\gamma=0}^2(-Q^{\alpha\beta\gamma}e^q\p_tZ^au\p_{\alpha}q\p_{\beta}Z^au\p_{\gamma}u
-Q^{\alpha\beta\gamma}e^q\p_tZ^au\p_{\beta}Z^au\p^2_{\alpha\gamma}u\\
&\qquad+\frac{1}{2}Q^{\alpha\beta\gamma}e^q\p_tq\p_{\alpha}Z^au\p_{\beta}Z^au\p_{\gamma}u
+\frac{1}{2}Q^{\alpha\beta\gamma}e^q\p_{\alpha}Z^au\p_{\beta}Z^au\p^2_{0\gamma}u).
\end{split}
\end{equation}
Although $Z^a=\p_t^j$ has been taken here, $I^{abc}_1$ and $I^a_2$ can be analogously treated for $Z=\{\p_t,\p_1,\p_2,\Omega\}$.
When $|b|\le N$ in $I^{abc}_1$, it follows from Lemmas \ref{lem:eqn:high}-\ref{lem:null:structure} that
\begin{equation}\label{energy:time6}
|I^{abc}_1|\ls|\p_tZ^au||\bar\p\p Z^bu||\p Z^cu|+|\p_tZ^au||\p^2Z^bu||\bar\p Z^cu|.
\end{equation}
For $x\in\cK\cap\{x: |x|\le1+s/2\}$, it can be deduced from \eqref{BA1} that
\begin{equation}\label{energy:time7}
\sum_{|b|\le N}\int_{\cK\cap\{x:|x|\le1+s/2\}}|I^{abc}_1|dx
\ls\ve_1(1+s)^{-1}\sum_{|b|\le|a|}\|\p Z^bu(s)\|^2_{L^2(\cK)}.
\end{equation}
For $x\in\cK\cap\{x: |x|\ge1+s/2\}$, \eqref{BA2} leads to
\begin{equation}\label{energy:time8}
\sum_{|b|\le N}\int_{\cK\cap\{x:|x|\ge1+s/2\}}|\p_tZ^au||\bar\p\p Z^bu||\p Z^cu|dx\ls\ve_1(1+s)^{-1}\sum_{|c|\le|a|}\|\p Z^cu(s)\|^2_{L^2(\cK)}.
\end{equation}
For the last term in \eqref{energy:time6}, it follows from \eqref{BA1} and the Young inequality that
\begin{equation}\label{energy:time9}
\begin{split}
&\quad\sum_{|b|\le N}\int\int_{\cK\cap\{x:|x|\ge1+s/2\}}|\p_tZ^au||\p^2Z^bu||\bar\p Z^cu|dxds\\
&\ls\ve_1\int_0^t\frac{\|\p Z^au(s)\|^2_{L^2(\cK)}ds}{1+s}
+\sum_{|c|\le|a|}\ve_1\int_0^t\int_{\cK}\frac{|\bar\p Z^cu(s,x)|^2}{\w{s-|x|}^2}dxds.
\end{split}
\end{equation}
When $|b|\ge N+1$ in $I^{abc}_1$, we can see that $|c|\le N-1$ and
\begin{equation*}
\begin{split}
I_1^{abc}&=\p_{\alpha}(Q_{abc}^{\alpha\beta\gamma}e^q\p_tZ^au\p_{\beta}Z^bu\p_{\gamma}Z^cu)
-Q_{abc}^{\alpha\beta\gamma}\p_tZ^au\p_{\beta}Z^bu\p_{\alpha}(e^q\p_{\gamma}Z^cu)\\
&\quad-Q_{abc}^{\alpha\beta\gamma}e^q\p_t\p_{\alpha}Z^au\p_{\beta}Z^bu\p_{\gamma}Z^cu\\
&=\p_{\alpha}(Q_{abc}^{\alpha\beta\gamma}e^q\p_tZ^au\p_{\beta}Z^bu\p_{\gamma}Z^cu)
-Q_{abc}^{\alpha\beta\gamma}\p_tZ^au\p_{\beta}Z^bu\p_{\alpha}(e^q\p_{\gamma}Z^cu)\\
&\quad-\p_t(Q_{abc}^{\alpha\beta\gamma}e^q\p_{\alpha}Z^au\p_{\beta}Z^bu\p_{\gamma}Z^cu)
+Q_{abc}^{\alpha\beta\gamma}\p_{\alpha}Z^au\p_t(e^q\p_{\beta}Z^bu\p_{\gamma}Z^cu),
\end{split}
\end{equation*}
Thus, the analogous estimates for $I^{abc}_1$ with $|b|\ge N+1$ in \eqref{energy:time5} can be established as in \eqref{energy:time6}-\eqref{energy:time9} with the boundary conditions and the fact $\bar\p q=0$.
Collecting \eqref{energy:time6}-\eqref{energy:time9} yields
\begin{equation}\label{energy:time10}
\begin{split}
&\int_0^t\int_{\cK}\sum_{\substack{|b|+|c|\le|a|,\\|b|<|a|}}|I^{abc}_1|dxds
\ls\ve^2+\ve_1\sum_{|a|\le2N}\|\p Z^au(t)\|^2_{L^2(\cK)}\\
&+\ve_1\sum_{|b|\le|a|}\Big\{\int_0^t\frac{\|\p Z^bu(s)\|^2_{L^2(\cK)}ds}{1+s}+\int_0^t\int_{\cK}\frac{|\bar\p Z^bu(s,x)|^2}{\w{s-|x|}^2}dxds\Big\}.
\end{split}\end{equation}
In view of \eqref{null:condition}, Lemma \ref{lem:null:structure}, \eqref{BA1}, \eqref{BA2} and $\bar\p q=0$,
the estimate of $I^a_2$ in \eqref{energy:time5} is the same as \eqref{energy:time10}.
Combining \eqref{initial:data}, \eqref{BA1}, \eqref{energy:time4} and \eqref{energy:time10} for all $|a|\le2N$
together with the smallness of $\ve_1$ derives
\begin{equation*}
\begin{split}
&\quad\sum_{|a|\le2N}\Big\{\|\p Z^au\|^2_{L^2(\cK)}+\int_0^t\int_{\cK}\frac{|\bar\p Z^au(s,x)|^2}{\w{s-|x|}^2}dxds\Big\}\\
&\ls\ve^2+\sum_{|a|\le2N}\ve_1\Big\{\int_0^t\frac{\|\p Z^au(s)\|^2_{L^2(\cK)}ds}{1+s}
+\int_0^t\int_{\cK}\frac{|\bar\p Z^au(s,x)|^2}{\w{s-|x|}^2}dxds\Big\}.
\end{split}
\end{equation*}
Thus, returning to the notation $Z^a=\p_t^j$ we can achieve
\begin{equation}\label{energy:time11}
\sum_{j\le2N}\|\p\p_t^ju\|^2_{L^2(\cK)}\ls\ve^2+\ve_1\int_0^t\sum_{j\le2N}\frac{\|\p\p_t^ju(s)\|^2_{L^2(\cK)}ds}{1+s}.
\end{equation}
Applying Lemma \ref{lem:Gronwall} to \eqref{energy:time11} completes the proof of \eqref{energy:time}.
\end{proof}

\begin{lemma}
Under the assumptions of Theorem \ref{thm1}, let $u$ be the solution of \eqref{QWE} and suppose that \eqref{BA1}-\eqref{BA3} hold.
Then it holds that
\begin{equation}\label{energy:dtdx}
\sum_{|a|\le2N}\|\p\p^au\|_{L^2(\cK)}\ls\ve_1(1+t)^{\ve_2}.
\end{equation}
\end{lemma}
\begin{proof}
Set $\ds E_j(t):=\sum_{k=0}^{2N-j}\|\p\p_t^ku(t)\|_{H^j(\cK)}$ with $0\le j\le2N$.
Then one can find that for $j\ge1$,
\begin{equation}\label{energy:dtdx1}
\begin{split}
E_j(t)&\ls\sum_{k=0}^{2N-j}[\|\p\p_t^ku(t)\|_{L^2(\cK)}
+\sum_{1\le|a|\le j}(\|\p_t\p_t^k\p_x^au(t)\|_{L^2(\cK)}+\|\p_x\p_t^k\p_x^au(t)\|_{L^2(\cK)})]\\
&\ls E_0(t)+E_{j-1}(t)+\sum_{k=0}^{2N-j}\sum_{2\le|a|\le j+1}\|\p_t^k\p_x^au(t)\|_{L^2(\cK)},
\end{split}
\end{equation}
where we have used the fact that
\begin{equation*}
\sum_{k=0}^{2N-j}\sum_{1\le|a|\le j}\|\p_t\p_t^k\p_x^au(t)\|_{L^2(\cK)}
\ls\sum_{k=0}^{2N-j}\sum_{|a'|\le j-1}\|\p_x\p_t^{k+1}\p_x^{a'}u(t)\|_{L^2(\cK)}
\ls E_{j-1}(t).
\end{equation*}
For the last term in \eqref{energy:dtdx1}, it can be deduced from the elliptic estimate \eqref{ellip} that
\begin{equation}\label{energy:dtdx2}
\begin{split}
\|\p_t^k\p_x^au\|_{L^2(\cK)}&\ls\|\Delta_x\p_t^ku\|_{H^{|a|-2}(\cK)}+\|\p_t^ku\|_{H^{|a|-1}(\cK_{R+1})}\\
&\ls\|\p_t^{k+2}u\|_{H^{|a|-2}(\cK)}+\|\p_t^kQ(\p u,\p^2u)\|_{H^{|a|-2}(\cK)}\\
&\quad+\|\p_t^ku\|_{L^2(\cK_{R+1})}+\sum_{1\le|b|\le|a|-1}\|\p_t^k\p_x^bu\|_{L^2(\cK_{R+1})},
\end{split}
\end{equation}
where $\Delta_x=\p_t^2-\Box$ and the equation in \eqref{QWE} have been used.
Furthermore, by \eqref{BA1} with $|a|+k\le2N+1$ we have
\begin{equation}\label{energy:dtdx3}
\begin{split}
\|\p_t^kQ(\p u,\p^2u)\|_{H^{|a|-2}(\cK)}
&\ls\sum_{l\le k}\ve_1(\|\p_t^l\p u\|_{H^{|a|-1}(\cK)}+\|\p_t^{l+2}u\|_{H^{|a|-2}(\cK)})\\
&\ls\ve_1[E_j(t)+E_{j-1}(t)].
\end{split}
\end{equation}
Combining \eqref{energy:dtdx1}-\eqref{energy:dtdx3} with \eqref{BA3} and \eqref{energy:time'} shows
\begin{equation*}
\begin{split}
E_j(t)&\ls E_0(t)+E_{j-1}(t)+\ve_1E_j(t)+\|u\|_{L^2(\cK_{R+1})}\\
&\ls\ve(1+t)^{\ve_2}+E_{j-1}(t)+\ve_1E_j(t)+\ve_1.
\end{split}
\end{equation*}
This, together with the smallness of $\ve_1>\ve$, yields \eqref{energy:dtdx}.
\end{proof}

\begin{lemma}
Under the assumptions of Theorem \ref{thm1}, let $u$ be the solution of \eqref{QWE} and suppose that \eqref{BA1}-\eqref{BA3} hold.
Then
\begin{equation}\label{energy:A}
\sum_{|a|\le2N-1}\|\p Z^au\|_{L^2(\cK)}\ls\ve_1(1+t)^{\ve_2+1/2}.
\end{equation}
\end{lemma}
\begin{proof}
Analogously to \eqref{energy:time2}, \eqref{energy:time3}, \eqref{energy:time4} and \eqref{energy:time10}, we have
\begin{equation}\label{energy:A1}
\begin{split}
&\sum_{|a|\le2N-1}\|\p Z^au\|^2_{L^2(\cK)}\ls\ve^2+\sum_{|a|\le2N-1}\Big\{\ve_1\int_0^t\frac{\|\p Z^au(s)\|^2_{L^2(\cK)}ds}{1+s}
+|\cB^a_1|+|\cB^a_2|\Big\},\\
&\cB^a_1:=\int_0^t\int_{\p\cK}e^q\nu(x)\cdot\nabla Z^au(s,x)\p_tZ^au(s,x)d\sigma ds,\\
&\cB^a_2:=\int_0^t\int_{\p\cK}\sum_{j=1}^2\sum_{\beta,\gamma=0}^2
Q^{j\beta\gamma}\nu_j(x)e^q\p_tZ^au\p_{\beta}Z^au\p_{\gamma}ud\sigma ds,
\end{split}
\end{equation}
where $\nu(x)=(\nu_1(x),\nu_2(x))$ is the unit outer normal of the boundary $\cK$ and $d\sigma$ is the curve measure on $\p\cK$.
According to $\p\cK\subset\overline{\cK_1}$ and the trace theorem, one has
\begin{equation}\label{energy:A2}
\begin{split}
|\cB^a_1|&\ls\sum_{|b|\le|a|}\int_0^t\|(1-\chi_{[1,2]}(x))\p_t\p^bu\|_{L^2(\p\cK)}\|(1-\chi_{[1,2]}(x))\p_x\p^bu\|_{L^2(\p\cK)}ds\\
&\ls\sum_{|b|\le|a|}\int_0^t\|(1-\chi_{[1,2]}(x))\p \p^bu\|^2_{H^1(\cK)}ds\\
&\ls\sum_{|b|\le|a|+1}\int_0^t\|\p \p^bu\|^2_{L^2(\cK_2)}ds\ls\ve_1^2(1+t)^{2\ve_2+1},
\end{split}
\end{equation}
where \eqref{energy:dtdx} has been used. Analogously, $\cB^a_2$ can be also treated.
Then, \eqref{energy:A1} and \eqref{energy:A2} ensure that there is a constant $C_1>0$ such that
\begin{equation}\label{energy:A3}
\begin{split}
\sum_{|a|\le2N-1}\|\p Z^au\|^2_{L^2(\cK)}\le C_1\ve^2+C_1\ve_1^2(1+t)^{2\ve_2+1}\\
+C_1\ve_1\int_0^t\sum_{|a|\le2N-1}\frac{\|\p Z^au(s)\|^2_{L^2(\cK)}ds}{1+s}.
\end{split}
\end{equation}
Applying Lemma \ref{lem:Gronwall} to \eqref{energy:A3} with $A=C_1\ve^2$, $B=C_1\ve_1$, $C=C_1\ve_1^2$, $D=2\ve_2+1>B=C_1\ve_1$ and
the smallness of $\ve_1$, we arrive at
\begin{equation*}
\begin{split}
\sum_{|a|\le2N-1}\|\p Z^au\|^2_{L^2(\cK)}&\le C_1\ve_1^2(1+t)^{2\ve_2+1}+\frac{C_1^2\ve_1^3}{2\ve_2+1-C_1\ve_1}(1+t)^{2\ve_2+1}\\
&\ls\ve_1^2(1+t)^{2\ve_2+1},
\end{split}
\end{equation*}
which yields \eqref{energy:A}.
\end{proof}

\subsection{Decay estimates of local energy and improved energy estimates}\label{sect:LocEnergy}

To improve the energy estimate \eqref{energy:A}, one needs a better estimate on the boundary term \eqref{energy:A2}.
This subsection aims to derive the precise time decay estimates of the local energy to problem \eqref{QWE}
by means of the spacetime pointwise estimate and the elliptic estimate.
To this end, we will treat the local energy $\|u\|_{L^2(\cK_R)}$, which can be derived by \eqref{loc:divform}.
\begin{lemma}
Under the assumptions of Theorem \ref{thm1}, let $u$ be the solution of \eqref{QWE} and suppose that \eqref{BA1}-\eqref{BA3} hold.
Then we have
\begin{equation}\label{u:loc:pw}
\|u\|_{L^\infty(\cK_R)}\ls(\ve+\ve_1^2)(1+t)^{-1}.
\end{equation}
\end{lemma}
\begin{proof}
It follows from \eqref{eqn:good:div} with $a=0$ that
\begin{equation}\label{u:loc:pw1}
\begin{split}
\Box V_0&=\frac{1}{2}\sum_{\alpha=0}^2
[\Delta,C_{000}^{0,\alpha}(\omega)\chi_{[1/2,1]}(x)](u\p_{\alpha}u)\\
&\quad+\sum_{\alpha=0}^2\{C(\omega)(1-\chi_{[1/2,1]}(x))Q_0(\p_{\alpha}u,u)+C(\omega)\p_{\alpha}u\Box u\\
&\qquad+C(\omega)\chi_{[1/2,1]}(x)(u\Box\p_{\alpha}u+\p_{\alpha}u\Box u)\}\\
&\quad+\p_\theta(\frac{C(\omega)u \p^2u}{|x|})
+\sum_{\alpha=0}^2\p_{\alpha}(\frac{C(\omega)u Z\p u}{|x|})
+\frac{C(\omega)u Z\p u}{|x|^2},
\end{split}
\end{equation}
where the good unknown $V_0$ is defined by \eqref{def:goodunknown}.
Applying \eqref{loc:divform} with $\mu=\nu=\ve_2$ to \eqref{u:loc:pw1} yields
\begin{equation}\label{u:loc:pw2}
\begin{split}
\w{t}\|V_0\|_{L^\infty(\cK_R)}&\ls\ve+\sup_{s\in[0,t]}\w{s}^{3/2+\ve_2}\|G_0^{loc}\|_{L^\infty(\cK_3)}\\
&+\sup_{(s,y)\in[0,t]\times\cK}\w{y}^{1/2}\cW_{3/2+\ve_2,1}(s,y)|G_0(s,y)|\\
&+\sum_{|a|\le3}\sup_{(s,y)\in[0,t]\times\cK}\w{y}^{1/2}\cW_{1+2\ve_2,1}(s,y)|Z^a(\frac{u Z\p u}{|y|})|,
\end{split}
\end{equation}
where the initial data \eqref{initial:data} and \eqref{comm:estimate} have been used, and
\begin{equation}\label{u:loc:pw3}
\begin{split}
G_0&:=\sum_{|a|\le2}\Big\{|Z^a(u\Box\p u)|+|Z^a(\p u\Box u)|\\
&\qquad+\frac{|Z^a(u\p u)|+|Z^a(Zu\p u)|+|Z^a(uZ\p u)|}{|x|^2}\Big\},\\
G_0^{loc}&:=\sum_{|a|\le1}\sum_{|b|\le1,|c|\le2}|\p^a(\p^bu\p^c\p u)|.
\end{split}
\end{equation}
In the region $\cK\cap\{y:|y|\le3+s/2\}\supseteq\cK_3$, it follows from \eqref{eqn:high}, \eqref{BA1}, \eqref{BA3} and $N\ge59$ that
\begin{equation}\label{u:loc:pw4}
\begin{split}
|G_0(s,y)|&\ls\ve_1^3\w{y}^{-1}\w{s}^{-5/2}+\ve_1^2\w{y}^{-5/2}\w{s}^{-1.6},\\
|G_0^{loc}(s,y)|&\ls\ve_1^2\w{s}^{-1.6},
\end{split}
\end{equation}
which yields
\begin{equation}\label{u:loc:pw5}
\sup_{s\in[0,t]}\w{s}^{3/2+\ve_2}\|G_0^{loc}\|_{L^\infty(\cK_3)}\ls\ve_1^2.
\end{equation}
In the region $\cK\cap\{y: |y|\ge3+s/2\}$, Lemmas \ref{lem:eqn:high}-\ref{lem:null:structure} and \eqref{BA1}-\eqref{BA3} imply
\begin{equation}\label{u:loc:pw6}
|G_0(s,y)|\ls\ve_1^2\w{s+|y|}^{2\ve_2-5/2}\w{s-|y|}^{-1}.
\end{equation}
On the other hand, it can be deduced from \eqref{BA1} and \eqref{BA3} that
\begin{equation}\label{u:loc:pw7}
\sum_{|a|\le3}|Z^a(\frac{u Z\p u}{|y|})|\ls\ve_1^2\w{y}^{-3/2}\w{s+|y|}^{\ve_2-1/2}\w{s-|y|}^{-1}.
\end{equation}
Collecting \eqref{u:loc:pw2}-\eqref{u:loc:pw7} yields
\begin{equation*}
\|V_0\|_{L^\infty(\cK_R)}\ls(\ve+\ve_1^2)(1+t)^{-1}.
\end{equation*}
This, together with \eqref{def:goodunknown}, \eqref{BA1} and \eqref{BA3}, completes the proof of \eqref{u:loc:pw}.
\end{proof}

\begin{lemma}\label{lem:loc:energy1}
Under the assumptions of Theorem \ref{thm1}, let $u$ be the solution of \eqref{QWE} and suppose that \eqref{BA1}-\eqref{BA3} hold.
Then we have
\begin{equation}\label{LocEnergydt}
\sum_{j\le2N-8}\|\p_t\p_t^ju\|_{L^2(\cK_R)}\ls(\ve+\ve_1^2)(1+t)^{4\ve_2-1/2}.
\end{equation}
\end{lemma}
\begin{proof}
Note that for $Z^a=\p_t^j$ with $j=|a|$ in \eqref{eqn:good}, the term $\Box(\tilde Z^a-Z^a)u$ vanishes.
Applying \eqref{loc:dt:sharp} to $\Box V_a$ given by \eqref{eqn:good} with $j=|a|\le2N-8$ and $\mu=\nu=\ve_2/2$ yields
\begin{equation}\label{LocEnergydt1}
\begin{split}
\w{t}\sum_{|a|\le2N-8}\|\p_tV_a\|_{L^\infty(\cK_R)}
\ls\ve+\sum_{|d|\le2}\sum_{|a|\le2N-8}\sup_{s\in[0,t]}\w{s}\|\p^d\Box V_a(s)\|_{L^2(\cK_3)}\\
+\sum_{|d|\le3}\sum_{|a|\le2N-8}\sup_{(s,y)\in[0,t]\times\cK}\w{y}^{1/2}\cW_{1+\ve_2,1}(s,y)|Z^d\Box V_a(s,y)|.
\end{split}
\end{equation}
By \eqref{Sobo:ineq} and \eqref{energy:A}, we can see that
\begin{equation}\label{LocEnergydt2}
\sum_{|a|\le2N-3}|\p Z^au|\ls\ve_1\w{x}^{-1/2}(1+t)^{\ve_2+1/2}.
\end{equation}
It is noted that $(\Box\p_{\mu}Z^bu)Z^cu$ in the third line of \eqref{eqn:good} with $|c|>N$ must be $(\Box\p_{\mu}Z^bu)\p_t^{|c|}u$.
Therefore, collecting \eqref{eqn:high}, \eqref{Q12:null:structure}, \eqref{eqn:good}, \eqref{BA1}-\eqref{BA3} and \eqref{LocEnergydt2} leads to
\begin{equation}\label{LocEnergydt3}
\begin{split}
&|Z^d\Box V_a(s,y)|\ls\ve_1^2\w{y}^{-3/2}(1+s)^{\ve_2+1/2}\w{s+|y|}^{\ve_2-1/2}\w{s-|y|}^{\ve_2-1/2}\\
&+\ve_1^3(1+s)^{\ve_2+1/2}(\w{y}^{-3/2}\w{s-|y|}^{-2}+\w{y}^{-1}\w{s+|y|}^{\ve_2-1/2}\w{s-|y|}^{\ve_2-3/2}).
\end{split}
\end{equation}
Substituting \eqref{LocEnergydt3} into \eqref{LocEnergydt1} with \eqref{eqn:good}, \eqref{BA1}, \eqref{energy:dtdx}, \eqref{u:loc:pw} and the fact $\supp_y u(s,y)\subset\{y:|y|\le s+M_0\}$ derives
\begin{equation*}
\sum_{|a|\le2N-8}\|\p_tV_a\|_{L^\infty(\cK_R)}\ls(\ve+\ve_1^2)(1+t)^{4\ve_2-1/2}.
\end{equation*}
Combining this estimate with $Z^a=\p_t^j$, \eqref{def:goodunknown}, \eqref{BA1}, \eqref{energy:dtdx}, \eqref{u:loc:pw} and the standard Sobolev embedding yields \eqref{LocEnergydt}.
\end{proof}

\begin{remark}\label{rmk5-1}
Without the decay factor $\w{t-|x|}^{0.001-1/2}$  in \eqref{thm1:decay:c}, one can not achieve the local energy decay estimate \eqref{LocEnergydt}.
Indeed, if we apply \eqref{Kubo15cpaa} instead of \eqref{impr:pw:Cauchy}
as in Remark \ref{rmk4-1}, then we can only obtain
the weak decay factor $\w{t+|x|}^{-1/2}(\min\{\w{x},\w{t-|x|}\})^{-1/2}$ for the estimate of $u$.
Substituting this into \eqref{LocEnergydt3} yields that for $|y|\le s/2$,
\begin{equation}\label{YHC-35}
|Z^d\Box V_a(s,y)|\ls\ve_1^2\w{y}^{-2}(1+s)^{\ve_2+1/2}\w{s+|y|}^{-1/2}+\cdots
\end{equation}
On the other hand, in order to avoid estimating $Z^au$, if one attends to apply \eqref{loc:dt:sharp} to \eqref{eqn:high} instead of \eqref{eqn:good}, then such an estimate is obtained
\begin{equation}\label{YHC-36}
|Z^d\Box Z^au(s,y)|\ls\ve_1^2\w{y}^{-1}(1+s)^{\ve_2+1/2}\w{s-|y|}^{-1}.
\end{equation}
Substituting \eqref{YHC-35} and \eqref{YHC-36} into \eqref{loc:dt:sharp} and the last term in \eqref{loc:dt:sharp} derives
\begin{equation*}
\begin{split}
\sup_{(s,y)\in[0,t]\times\cK}\w{y}^{1/2}\cW_{1+\mu+\nu,1}(s,y)|Z^d\Box V_a(s,y)|&\ls\ve_1^2\w{t}^{1+\mu+\nu+\ve_2},\\
\sup_{(s,y)\in[0,t]\times\cK}\w{y}^{1/2}\cW_{1+\mu+\nu,1}(s,y)|Z^d\Box Z^au(s,y)|&\ls\ve_1^2\w{t}^{1+\mu+\nu+\ve_2}.
\end{split}
\end{equation*}
Thus, both \eqref{YHC-35} and \eqref{YHC-36} can not be used to derive the local energy decay estimate \eqref{LocEnergydt}.
\end{remark}

\begin{lemma}
Under the assumptions of Theorem \ref{thm1}, let $u$ be the solution of \eqref{QWE} and suppose that \eqref{BA1}-\eqref{BA3} hold.
Then we have
\begin{equation}\label{loc:energy}
\sum_{|a|\le2N-8}\|\p^au\|_{L^2(\cK_R)}\ls(\ve+\ve_1^2)(1+t)^{4\ve_2-1/2}.
\end{equation}
\end{lemma}
\begin{proof}
For $j=0,1,2,\cdots,2N-8$, denote $\ds E_j^{loc}(t):=\sum_{k\le j}\|\p_t^k\p_x^{2N-8-j}u\|_{L^2(\cK_{R+j})}$.
When $j\le2N-10$, we have $\p_x^{2N-8-j}=\p_x^2\p_x^{2N-10-j}$.
Thus, one can apply the elliptic estimate \eqref{ellip} to $(1-\chi_{[R+j,R+j+1]})\p_t^{k}u$
to obtain
\begin{equation}\label{loc:energy1}
\begin{split}
E_j^{loc}(t)&\ls\sum_{k\le j}\|\p_x^{2N-8-j}(1-\chi_{[R+j,R+j+1]})\p_t^ku\|_{L^2(\cK)}\\
&\ls\sum_{k\le j}\{\|\Delta[(1-\chi_{[R+j,R+j+1]})\p_t^ku]\|_{H^{2N-10-j}(\cK)}\\
&\quad+\sum_{k\le j}\|(1-\chi_{[R+j,R+j+1]})\p_t^ku\|_{H^{2N-9-j}(\cK)}\}\\
&\ls\sum_{k\le j}\|(1-\chi_{[R+j,R+j+1]})\Delta\p_t^ku\|_{H^{2N-10-j}(\cK)}+\sum_{j+1\le l\le2N-8}E_l^{loc}(t)\\
&\ls\sum_{k\le j}\|(\Box\p_t^ku,\p_t^{k+2}u)\|_{H^{2N-10-j}(\cK_{R+j+1})}+\sum_{j+1\le l\le2N-8}E_l^{loc}(t),
\end{split}
\end{equation}
where the fact of $\Delta=\p_t^2-\Box$ has been used.
From \eqref{QWE}, \eqref{BA1} and \eqref{energy:dtdx}, we can get that for $j\le2N-10$,
\begin{equation}\label{loc:energy2}
\sum_{k\le j}\|\Box\p_t^ku\|_{H^{2N-10-j}(\cK_{R+j+1})}\ls\ve_1^2(1+t)^{\ve_2-1}.
\end{equation}
Thus, \eqref{loc:energy1} and \eqref{loc:energy2} imply that for $j\le2N-10$,
\begin{equation}\label{loc:energy3}
E_j^{loc}(t)\ls\ve_1^2(1+t)^{\ve_2-1}+\sum_{j+1\le l\le2N-8}E_l^{loc}(t).
\end{equation}
Next we turn to the estimate of $E_{2N-9}^{loc}(t)$. It is easy to find that
\begin{equation}\label{loc:energy4}
(E_{2N-9}^{loc}(t))^2\ls\sum_{k\le2N-9}\sum_{i=1}^2\|\p_i[(1-\chi_{[R+2N-9,R+2N-8]})\p_t^ku]\|^2_{L^2(\cK)}.
\end{equation}
Denote $\tilde\chi:=1-\chi_{[R+2N-9,R+2N-8]}$. Then it follows from the integration by parts
together with the boundary condition that
\begin{equation}\label{loc:energy5}
\begin{split}
&\quad\sum_{i=1}^2\|\p_i[(1-\chi_{[R+2N-9,R+2N-8]})\p_t^ku]\|^2_{L^2(\cK)}\\
&=\int_{\cK}\sum_{i=1}^2\p_i[\tilde\chi\p_t^ku\p_i(\tilde\chi\p_t^ku)]dx
-\int_{\cK}\tilde\chi\p_t^ku\Delta(\tilde\chi\p_t^ku)dx\\
&=-\int_{\cK}\tilde\chi^2\p_t^ku(\Delta\p_t^ku)dx-\int_{\cK}\tilde\chi(\Delta\tilde\chi)|\p_t^ku|^2dx
-2\int_{\cK}\tilde\chi\p_t^ku\nabla\tilde\chi\cdot\nabla\p_t^kudx\\
&=-\int_{\cK}\tilde\chi^2\p_t^ku\p_t^{k+2}udx+\int_{\cK}\tilde\chi^2\p_t^ku\Box\p_t^kudx
-\int_{\cK}\tilde\chi(\Delta\tilde\chi)|\p_t^ku|^2dx\\
&\quad-\int_{\cK}\dive(\tilde\chi\nabla\tilde\chi|\p_t^ku|^2)dx
+\int_{\cK}\dive(\tilde\chi\nabla\tilde\chi)|\p_t^ku|^2dx\\
&\ls(E_{2N-8}^{loc}(t))^2+\|\p_t^{2N-7}u\|^2_{L^2(\cK_{R+2N})}+\ve_1^4(1+t)^{2\ve_2-2},
\end{split}
\end{equation}
where we have used $k\le2N-9$, $\Delta=\p_t^2-\Box$ and the estimate of $\Box\p_t^ku$ as in \eqref{loc:energy2}.
On the other hand, by \eqref{u:loc:pw} and \eqref{LocEnergydt}, one has
\begin{equation}\label{loc:energy6}
\begin{split}
E_{2N-8}^{loc}(t)&\ls\sum_{j\le2N-7}\|\p_t^ju\|_{L^2(\cK_{R+2N})}\\
&\ls\sum_{j\le2N-8}\|\p_t\p_t^ju\|_{L^2(\cK_{R+2N})}+\|u\|_{L^2(\cK_{R+2N})}\\
&\ls(\ve+\ve_1^2)(1+t)^{4\ve_2-1/2}+(\ve+\ve_1^2)(1+t)^{-1}.
\end{split}
\end{equation}
Hence, \eqref{loc:energy} can be obtained by \eqref{LocEnergydt}, \eqref{loc:energy3}, \eqref{loc:energy4}, \eqref{loc:energy5} and \eqref{loc:energy6}.
\end{proof}

\begin{lemma}
Under the assumptions of Theorem \ref{thm1}, let $u$ be the solution of \eqref{QWE} and suppose that \eqref{BA1}-\eqref{BA3} hold.
Then
\begin{equation}\label{energy:B}
\sum_{|a|\le2N-10}\|\p Z^au\|_{L^2(\cK)}\ls\ve_1(1+t)^{4\ve_2}.
\end{equation}
\end{lemma}
\begin{proof}
The proof of \eqref{energy:B} is analogous to that of \eqref{energy:A} with a better estimate than \eqref{energy:A2}
as follows
\begin{equation}\label{energy:B1}
\begin{split}
\sum_{|a|\le2N-10}|\cB^a_1|&\ls\sum_{|b|\le|a|+1\le2N-9}\int_0^t\|\p \p^bu(s)\|^2_{L^2(\cK_2)}ds\\
&\ls\ve_1^2\int_0^t(1+s)^{8\ve_2-1}ds\ls\ve_1^2(1+t)^{8\ve_2},
\end{split}
\end{equation}
where $\cB^a_1$ is defined in \eqref{energy:A1} and \eqref{loc:energy} is used.
As in \eqref{energy:A3}, there is constant $C_2>0$ such that
\begin{equation*}
\sum_{|a|\le2N-10}\|\p Z^au\|^2_{L^2(\cK)}\le C_2\ve_1^2(1+t)^{8\ve_2}+C_2\ve_1\int_0^t\sum_{|a|\le2N-10}\frac{\|\p Z^au(s)\|^2_{L^2(\cK)}ds}{1+s}.
\end{equation*}
For sufficiently small $\ve_1>0$, we have
\begin{equation}\label{energy:B2}
\sum_{|a|\le2N-10}\|\p Z^au\|^2_{L^2(\cK)}\le C_2\ve_1^2(1+t)^{8\ve_2}+\ve_2\int_0^t\sum_{|a|\le2N-10}\frac{\|\p Z^au(s)\|^2_{L^2(\cK)}ds}{1+s}.
\end{equation}
Applying Lemma \ref{lem:Gronwall} with $B=\ve_2<D=8\ve_2$ to \eqref{energy:B2} yields \eqref{energy:B}.
\end{proof}

\begin{lemma}\label{lem:loc:energy2}
Under the assumptions of Theorem \ref{thm1}, let $u$ be the solution of \eqref{QWE} and suppose that \eqref{BA1}-\eqref{BA3} hold.
Then one has
\begin{equation}\label{loc:impr:A}
\sum_{|a|\le2N-17}\|\p^au\|_{L^2(\cK_R)}\ls(\ve+\ve_1^2)(1+t)^{7\ve_2-1}.
\end{equation}
\end{lemma}
\begin{proof}
The proof is analogous to that of \eqref{LocEnergydt} by a better estimate than \eqref{LocEnergydt3}.
To this end, according to \eqref{Sobo:ineq} and \eqref{energy:B}, \eqref{LocEnergydt2} can be improved as
\begin{equation}\label{loc:impr:A1}
\sum_{|a|\le2N-12}|\p Z^au|\ls\ve_1\w{x}^{-1/2}(1+t)^{4\ve_2}.
\end{equation}
With \eqref{loc:impr:A1} instead of \eqref{LocEnergydt2}, \eqref{LocEnergydt3} may be improved to
\begin{equation}\label{loc:impr:A2}
\begin{split}
&\sum_{|a|\le2N-17}\sum_{|d|\le3}|Z^d\Box V_a(s,y)|
\ls\ve_1^2\w{y}^{-3/2}(1+s)^{4\ve_2}\w{s+|y|}^{\ve_2-1/2}\w{s-|y|}^{\ve_2-1/2}\\
&\quad+\ve_1^3(1+s)^{4\ve_2}(\w{y}^{-3/2}\w{s-|y|}^{-2}+\w{y}^{-1}\w{s+|y|}^{\ve_2-1/2}\w{s-|y|}^{\ve_2-3/2}).
\end{split}
\end{equation}
Thus, by applying \eqref{loc:dt:sharp} to $\Box V_a$ with $Z^a=\p_t^j$, $j=|a|\le2N-17$, \eqref{initial:data}, \eqref{eqn:good}, \eqref{BA1}, \eqref{u:loc:pw}, \eqref{loc:energy} and \eqref{loc:impr:A2}, we arrive at
\begin{equation*}
\sum_{|a|\le2N-17}\|\p_tV_a\|_{L^\infty(\cK_R)}\ls(\ve+\ve_1^2)(1+t)^{7\ve_2-1}.
\end{equation*}
This, together with \eqref{def:goodunknown}, \eqref{BA1}, \eqref{u:loc:pw}, \eqref{loc:energy}, yields
\begin{equation}\label{loc:impr:A3}
\sum_{j\le2N-17}\|\p_t\p_t^ju\|_{L^\infty(\cK_R)}\ls(\ve+\ve_1^2)(1+t)^{7\ve_2-1}.
\end{equation}
Finally, \eqref{loc:impr:A} can be obtained by the same method as in the proof of \eqref{loc:energy} together with \eqref{u:loc:pw} and \eqref{loc:impr:A3}.
\end{proof}

\section{Improved pointwise estimates and proof of Theorem \ref{thm1}}\label{sect6}

In this section, we improve the pointwise estimates \eqref{BA1}-\eqref{BA3} via the $L^\infty-L^\infty$ estimates established
in Section \ref{sect4} with \eqref{loc:impr:A1}, which is derived by the energy estimates together with
the $L^2-L^\infty$ estimate \eqref{Sobo:ineq}.

\subsection{Decay estimates of the good derivatives}

It is pointed out that the estimate of $Z^au$ will play an essential role in the decay estimates of the good derivatives,
which will be derived in terms of \eqref{InLW:pw}.

\begin{lemma}\label{lem:BA3:A}
Under the assumptions of Theorem \ref{thm1}, let $u$ be the solution of \eqref{QWE} and suppose that \eqref{BA1}-\eqref{BA3} hold.
Then one has
\begin{equation}\label{BA3:impr:A}
\sum_{|a|\le2N-18}|Z^au|\ls(\ve+\ve_1^2)\w{t+|x|}^{5\ve_2}.
\end{equation}
\end{lemma}
\begin{proof}
Due to the lack of the estimates on the  higher order derivatives $Z^au$ in \eqref{eqn:good} and \eqref{eqn:good:div}, we will adopt the equation \eqref{eqn:high} in priority.
Applying \eqref{InLW:pw} to $\Box\tilde Z^au=\Box(\tilde Z^a-Z^a)u+\Box Z^au$ for $|a|\le2N-18$ and $\mu=\nu=\ve_2/4$ yields
\begin{equation}\label{BA3:impr:A1}
\begin{split}
&\frac{\w{t+|x|}^{1/2}}{\ln^2(2+t+|x|)}|\tilde Z^au|\ls\ve+\sum_{|b|\le4}\sup_{s\in[0,t]}\w{s}^{1/2+\ve_2/2}\|\p^b\Box(\tilde Z^a-Z^a)u(s)\|_{L^2(\cK_4)}\\
&\qquad+\sum_{|b|\le5}\sup_{(s,y)\in[0,t]\times(\overline{\R^2\setminus\cK_2})}\w{y}^{1/2}\cW_{1+\ve_2/2,1}(s,y)|\p^b\Box Z^au(s,y)|,
\end{split}
\end{equation}
where we have used \eqref{initial:data} and the fact $\tilde Z=Z$ in the region $|x|\ge1$.
Then from \eqref{loc:energy}, one has
\begin{equation}\label{BA3:impr:A2}
\sum_{|a|\le2N-18}\sum_{|b|\le4}\sup_{s\in[0,t]}\w{s}^{1/2+\ve_2/2}\|\p^b\Box(\tilde Z^a-Z^a)u(s)\|_{L^2(\cK_4)}
\ls(\ve+\ve_1^2)(1+t)^{9\ve_2/2}.
\end{equation}
For the second line of \eqref{BA3:impr:A1}, it can be concluded from \eqref{eqn:high}, \eqref{BA1} and \eqref{loc:impr:A1} that
\begin{equation}\label{BA3:impr:A3}
\sum_{|a|\le2N-18}\sum_{|b|\le5}|Z^b\Box Z^au(s,y)|
\ls\ve_1^2\w{y}^{-1}\w{s-|y|}^{-1}(1+s)^{4\ve_2}.
\end{equation}
Substituting \eqref{BA3:impr:A2} and \eqref{BA3:impr:A3} into \eqref{BA3:impr:A1} leads to
\begin{equation*}
\sum_{|a|\le2N-18}|\tilde Z^au|\ls(\ve+\ve_1^2)\w{t+|x|}^{5\ve_2}.
\end{equation*}
This, together with \eqref{loc:energy}, finishes the proof of \eqref{BA3:impr:A}.
\end{proof}

\begin{lemma}\label{lem:BA2:A}
Under the assumptions of Theorem \ref{thm1}, let $u$ be the solution of \eqref{QWE} and suppose that \eqref{BA1}-\eqref{BA3} hold.
Then one has that for $|x|\ge1+t/2$,
\begin{equation}\label{BA2:impr:A}
\sum_{|a|\le2N-20}|\bar\p Z^au|\ls(\ve+\ve_1^2)\w{t+|x|}^{5\ve_2-1}.
\end{equation}
\end{lemma}
\begin{proof}
The proof mainly follows from Section 4.6 in \cite{Kubo13} by utilizing the better decay estimate \eqref{BA3:impr:A}.
Set $\p_{\pm}:=\p_t\pm\p_r$ and $r=|x|$. It is easy to check that
\begin{equation}\label{BA2:improv1}
\begin{split}
\bar\p_1&=\frac{x_1}{r}\p_t+\frac{1}{r^2}(x_1^2+x_2^2)\p_1
=\frac{x_1}{r}\p_t+\frac{1}{r^2}[x_1(x_1\p_1+x_2\p_2)-x_1x_2\p_2+x_2^2\p_1]\\
&=\frac{x_1}{r}\p_+-\frac{x_2}{r^2}\Omega,\\
\bar\p_2&=\frac{x_2}{r}\p_++\frac{x_1}{r^2}\Omega.
\end{split}
\end{equation}
We now focus on the estimate of $\p_+Z^au$.
Note that
\begin{equation}\label{BA2:improv2}
r^{1/2}\Box w=\p_-\p_+(r^{1/2}w)-r^{-3/2}(w/4+\Omega^2w).
\end{equation}
For fixed $t,x$ with $r=|x|\ge1+t/2$, integrating \eqref{BA2:improv2} along the integral curve of $\p_-$ yields that
\begin{equation}\label{BA2:improv3}
\begin{split}
&\quad\;\p_+(r^{1/2}w)(t,r\frac{x}{|x|})-\p_+(r^{1/2}w)(0,(r+t)\frac{x}{|x|})\\
&=\int_0^t\p_-\p_+((r+t-s)^{1/2}w)(s,(r+t-s)\frac{x}{|x|})ds\\
&=\int_0^t\{(r+t-s)^{1/2}\Box w+(r+t-s)^{-3/2}(w/4+\Omega^2w)\}(s,(r+t-s)\frac{x}{|x|})ds.
\end{split}
\end{equation}
Set $w=Z^au$ with $|a|\le2N-20$ in \eqref{BA2:improv3}. Then it follows from \eqref{BA3:impr:A3} that for $|y|\ge1+s/2$,
\begin{equation}\label{BA2:impr:A2}
|\Box Z^au(s,y)|\ls\ve_1^2\w{s+|y|}^{4\ve_2-1}\w{s-|y|}^{-1}.
\end{equation}
Substituting \eqref{initial:data}, \eqref{BA3:impr:A} and \eqref{BA2:impr:A2} into \eqref{BA2:improv3} yields that for $|x|\ge1+t/2$,
\begin{equation*}
\begin{split}
|\p_+(r^{1/2}Z^au)(t,x)|&\ls\ve\w{x}^{-1}+\ve_1^2(r+t)^{4\ve_2-1/2}\int_0^t(1+|r+t-2s|)^{-1}ds\\
&\quad+(\ve+\ve_1^2)\int_0^t(r+t-s)^{-3/2}\w{r+t}^{5\ve_2}ds\\
&\ls(\ve+\ve_1^2)\w{t+|x|}^{5\ve_2-1/2},
\end{split}
\end{equation*}
which derives
\begin{equation*}
\begin{split}
|\p_+Z^au|&\ls r^{-1/2}|\p_+(r^{1/2}Z^au)(t,x)|+r^{-1}|Z^au(t,x)|\\
&\ls(\ve+\ve_1^2)\w{t+|x|}^{5\ve_2-1}.
\end{split}
\end{equation*}
This, together with \eqref{BA3:impr:A} and \eqref{BA2:improv1}, completes the proof \eqref{BA2:impr:A}.
\end{proof}

\begin{lemma}
Under the assumptions of Theorem \ref{thm1}, let $u$ be the solution of \eqref{QWE} and suppose that \eqref{BA1}-\eqref{BA3} hold.
Then we have
\begin{equation}\label{BA3:impr:B}
\sum_{|a|\le2N-24}|Z^au|\ls(\ve+\ve_1^2)\w{t+|x|}^{7\ve_2-1/2}.
\end{equation}
\end{lemma}
\begin{proof}
Applying \eqref{InLW:pw} to \eqref{eqn:good} for $|a|\le2N-24$ and $\mu=\nu=\ve_2/4$ yields
\begin{equation}\label{BA3:impr:B1}
\begin{split}
&\frac{\w{t+|x|}^{1/2}}{\ln^2(2+t+|x|)}|V_a|\ls\ve+\sum_{|d|\le4}\sup_{s\in[0,t]}\w{s}^{1/2+\ve_2}\|\p^d\Box V_a(s)\|_{L^2(\cK_4)}\\
&\qquad+\sum_{|d|\le5}\sup_{(s,y)\in[0,t]\times(\overline{\R^2\setminus\cK_2})}\w{y}^{1/2}\cW_{1+\ve_2/2,1}(s,y)|\p^d\Box V_a(s,y)|,
\end{split}
\end{equation}
where we have used \eqref{initial:data}.
Analogously to \eqref{BA3:impr:A2}, one can obtain from \eqref{loc:energy} that
\begin{equation}\label{BA3:impr:B2}
\sum_{|a|\le2N-24}\sum_{|b|\le4}\sup_{s\in[0,t]}\w{s}^{1/2+\ve_2}\|\p^b\Box V_a(s)\|_{L^2(\cK_4)}
\ls(\ve+\ve_1^2)(1+t)^{5\ve_2}.
\end{equation}
Next, we turn to the estimate of the second line in \eqref{BA3:impr:B1} and deal with the terms $\Box\p_{\mu}Z^buZ^cu$ and $|x|^{-1}C(\omega)\p^2_{\mu\nu}Z^bu\Omega Z^cu$ in the third and fourth lines of \eqref{eqn:good}, respectively.
By Lemmas \ref{lem:eqn:high}-\ref{lem:null:structure}, \eqref{BA1}-\eqref{BA3}, \eqref{loc:impr:A1} and \eqref{BA3:impr:A},
we arrive at
\begin{equation*}
\begin{split}
&\quad \sum_{|b|+|c|\le2N-19}(|\Box\p_{\mu}Z^buZ^cu|+|x|^{-1}|C(\omega)\p^2_{\mu\nu}Z^bu\Omega Z^cu|)\\
&\ls\ve_1^3\w{y}^{-1}(\w{s+|y|}^{6\ve_2-1}\w{s-|y|}^{-1}
+\w{s+|y|}^{\ve_2-1/2}\w{s-|y|}^{\ve_2-3/2}(1+s)^{4\ve_2})\\
&+\ve_1^2\w{y}^{-3/2}\w{s-|y|}^{-1}\w{s+|y|}^{6\ve_2}.
\end{split}
\end{equation*}
Analogously, one can achieve
\begin{equation}\label{BA3:impr:B3}
\sum_{|a|\le2N-24}\sum_{|d|\le5}|Z^d\Box V_a(s,y)|\ls\ve_1^2\w{y}^{-3/2}\w{s-|y|}^{-1}(1+s)^{6\ve_2}.
\end{equation}
Substituting \eqref{BA3:impr:B2} and \eqref{BA3:impr:B3} into \eqref{BA3:impr:B1} leads to
\begin{equation}\label{BA3:impr:B4}
\sum_{|a|\le2N-24}|V_a|\ls(\ve+\ve_1^2)\w{t+|x|}^{7\ve_2-1/2}.
\end{equation}
Then it follows from \eqref{def:goodunknown}, \eqref{BA1}, \eqref{BA3}, \eqref{loc:impr:A}, \eqref{loc:impr:A1}, \eqref{BA3:impr:A} and \eqref{BA3:impr:B4} that
\begin{equation*}
\begin{split}
\sum_{|a|\le2N-24}|Z^au|&\ls(\ve+\ve_1^2)\w{t+|x|}^{7\ve_2-1/2}+\ve_1^2\w{t}^{4\ve_2}\w{x}^{-1/2}\w{t+|x|}^{\ve_2-1/2}\\
&\ls(\ve+\ve_1^2)\w{t+|x|}^{7\ve_2-1/2}.
\end{split}
\end{equation*}
This completes the proof of \eqref{BA3:impr:B}.
\end{proof}

\begin{lemma}
Under the assumptions of Theorem \ref{thm1}, let $u$ be the solution of \eqref{QWE} and suppose that \eqref{BA1}-\eqref{BA3} hold.
Then one has that for $|x|\ge1+t/2$
\begin{equation}\label{BA2:impr:B}
\sum_{|a|\le2N-26}|\bar\p Z^au|\ls(\ve+\ve_1^2)\w{t+|x|}^{7\ve_2-3/2}.
\end{equation}
\end{lemma}
\begin{proof}
By choosing $w=V_a$ with $|a|\le2N-26$ in \eqref{BA2:improv3}, it follows from \eqref{BA3:impr:B3} that for $|y|\ge1+s/2$,
\begin{equation}\label{BA2:impr:B1}
|\Box V_a(s,y)|\ls\ve_1^2\w{s+|y|}^{6\ve_2-3/2}\w{s-|y|}^{-1}.
\end{equation}
Substituting \eqref{initial:data}, \eqref{BA3:impr:B} and \eqref{BA2:impr:B1} into \eqref{BA2:improv3} yields that for $|x|\ge1+t/2$,
\begin{equation*}
\begin{split}
|\p_+(r^{1/2}V_a)(t,x)|&\ls\ve\w{x}^{-1}+\ve_1^2(r+t)^{6\ve_2-1}\int_0^t(1+|r+t-2s|)^{-1}ds\\
&\quad+(\ve+\ve_1^2)\int_0^t(r+t-s)^{-3/2}\w{r+t}^{7\ve_2-1/2}ds\\
&\ls(\ve+\ve_1^2)\w{t+|x|}^{7\ve_2-1}.
\end{split}
\end{equation*}
Together with \eqref{def:goodunknown}, \eqref{loc:impr:A1} and \eqref{BA3:impr:B}, this leads to
\begin{equation}\label{BA2:impr:B2}
\begin{split}
|\p_+V_a|&\ls r^{-1/2}|\p_+(r^{1/2}V_a)(t,x)|+r^{-1}|V_a(t,x)|\\
&\ls(\ve+\ve_1^2)\w{t+|x|}^{7\ve_2-3/2}.
\end{split}
\end{equation}
Collecting \eqref{BA2:impr:B2} with \eqref{loc:impr:A1}, \eqref{BA2:improv1} and \eqref{BA3:impr:B} derives
\begin{equation}\label{BA2:impr:B3}
\sum_{|a|\le2N-26}|\bar\p V_a|\ls(\ve+\ve_1^2)\w{t+|x|}^{7\ve_2-3/2}.
\end{equation}
By the definition \eqref{def:goodunknown}, \eqref{BA1}-\eqref{BA3}, \eqref{loc:impr:A1}, \eqref{BA2:impr:A}, \eqref{BA3:impr:B} and \eqref{BA2:impr:B3},
we can obtain
\begin{equation*}
\begin{split}
\sum_{|a|\le2N-26}|\bar\p Z^au|
&\ls(\ve+\ve_1^2)\w{t+|x|}^{7\ve_2-3/2}+\sum_{|b|+|c|\le2N-26}|\bar\p Z^bu||\p Z^cu|\\
&\quad+\sum_{|b|+|c|\le2N-26}|Z^cu|(|\bar\p\p Z^bu|+\w{x}^{-1}|\p Z^bu|)\\
&\ls(\ve+\ve_1^2)\w{t+|x|}^{7\ve_2-3/2}.
\end{split}
\end{equation*}
This completes the proof of \eqref{BA2:impr:B}.
\end{proof}

\subsection{Crucial pointwise estimates}

\begin{lemma}
Under the assumptions of Theorem \ref{thm1}, let $u$ be the solution of \eqref{QWE} and suppose that \eqref{BA1}-\eqref{BA3} hold.
Then we have
\begin{equation}\label{BA1:impr:A}
\sum_{|a|\le2N-37}|\p Z^au|\ls(\ve+\ve_1^2)\w{x}^{-1/2}\w{t-|x|}^{-1}\w{t+|x|}^{10\ve_2}.
\end{equation}
\end{lemma}
\begin{proof}
Utilizing \eqref{InLW:dpw} to \eqref{eqn:good} for $|a|\le2N-24$ and $\mu=\nu=\ve_2/4$ yields
\begin{equation}\label{BA1:impr:A1}
\begin{split}
&\w{x}^{1/2}\w{t-|x|}^{1/2-\ve_2/4}|\p V_a|
\ls\ve+\sum_{|b|\le3}\sup_{s\in[0,t]}\w{s}^{1/2}\|\p^b\Box V_a(s)\|_{L^2(\cK_4)}\\
&\qquad+\sum_{|b|\le4}\sup_{(s,y)\in[0,t]\times(\overline{\R^2\setminus\cK_2})}\w{y}^{1/2}\cW_{1+\ve_2/2,1}(s,y)|Z^b\Box V_a(s,y)|,
\end{split}
\end{equation}
where we have used \eqref{initial:data}.
It follows from \eqref{eqn:good} and \eqref{loc:impr:A} that
\begin{equation}\label{BA1:impr:A2}
\sum_{|a|\le2N-24}\sum_{|b|\le3}\sup_{s\in[0,t]}\w{s}^{1-7\ve_2}\|\p^b\Box V_a(s)\|_{L^2(\cK_4)}
\ls\ve+\ve_1^2.
\end{equation}
Substituting \eqref{BA3:impr:B3} and \eqref{BA1:impr:A2} into \eqref{BA1:impr:A1} gives
\begin{equation*}
\sum_{|a|\le2N-24}|\p V_a|\ls(\ve+\ve_1^2)\w{x}^{-1/2}\w{t-|x|}^{-1/2}(1+t)^{7\ve_2}.
\end{equation*}
This, together with \eqref{def:goodunknown}, \eqref{BA1}, \eqref{BA3}, \eqref{loc:impr:A}, \eqref{loc:impr:A1}, \eqref{BA3:impr:B}
and the Sobolev embedding, derives
\begin{equation}\label{BA1:impr:A3}
\sum_{|a|\le2N-24}|\p Z^au|\ls(\ve+\ve_1^2)\w{x}^{-1/2}\w{t-|x|}^{-1/2}(1+t)^{7\ve_2}.
\end{equation}
Applying \eqref{InLW:dpw:div} to \eqref{eqn:good:div} for $|a|\le2N-37$ and $\mu=\nu=\ve_2$ leads to
\begin{equation}\label{BA1:impr:A4}
\begin{split}
&\w{x}^{1/2}\w{t-|x|}|\p V_a|\ls(\ve+\ve_1^2)\w{t}^{7\ve_2}
+\sum_{i=1,2}\sup_{s\in[0,t]}\w{s}^{3/2+\ve_2}\|G_a^i(s)\|_{L^\infty(\cK_3)}\\
&\quad +\sup_{(s,y)\in[0,t]\times(\overline{\R^2\setminus\cK_2})}\w{y}^{1/2}\cW_{3/2+\ve_2,1}(s,y)|G_a(s,y)|\\
&\quad +\sum_{\alpha=0}^2\sup_{(s,y)\in[0,t]\times(\overline{\R^2\setminus\cK_2})}
\w{y}^{1/2}\cW_{1+2\ve_2,1}(s,y)|(G_a,G_a^\theta,G_a^\alpha)(s,y)|,
\end{split}
\end{equation}
where we have used \eqref{initial:data} and \eqref{BA1:impr:A2}, and $(G_a,G_a^\theta,G_a^\alpha)$ is given by
\begin{equation}\label{BA1:impr:A5}
\begin{split}
G_a&:=\sum_{|d|\le9}\sum_{|b|+|c|\le|a|}\{|Z^d(Z^bu\Box\p Z^cu)|+|Z^d(\p Z^bu\Box Z^cu)|\\
&\qquad+\frac{|Z^d(Z^bu\p Z^cu)|+|Z^d(Z^buZ\p Z^cu)|+|Z^d(ZZ^bu\p Z^cu)|}{|x|^2}\},\\
G_a^{\Xi}&:=\sum_{|d|\le10}\sum_{|b|+|c|\le|a|}\frac{Z^d(Z\p Z^buZ^cu)}{|x|},\qquad \Xi=\theta,0,1,2.
\end{split}
\end{equation}
By Lemmas \ref{lem:eqn:high}-\ref{lem:null:structure}, \eqref{BA1}-\eqref{BA3}, \eqref{BA3:impr:B}, \eqref{BA2:impr:B} and \eqref{BA1:impr:A3},
one has that for $|y|\ge1+s/2$,
\begin{equation}\label{BA1:impr:A6}
\begin{split}
|G_a|&\ls\ve_1^3\w{y}^{9\ve_2-5/2}\w{s-|y|}^{\ve_2-1}+\ve_1^2\w{y}^{8\ve_2-3}\w{s-|y|}^{\ve_2-1}\\
&\ls\ve_1^2\w{y}^{9\ve_2-5/2}\w{s-|y|}^{\ve_2-1}.
\end{split}
\end{equation}
In the region $|y|\le1+s/2$, it follows from \eqref{BA1}, \eqref{BA3}, \eqref{BA3:impr:B} and \eqref{BA1:impr:A3} that
\begin{equation}\label{BA1:impr:A7}
\begin{split}
|G_a|&\ls\ve_1^3\w{y}^{-1}\w{s}^{9\ve_2-5/2}+\ve_1^2\w{y}^{-5/2}\w{s}^{9\ve_2-3/2}\\
&\ls\ve_1^2\w{y}^{-2}\w{s}^{9\ve_2-3/2}.
\end{split}
\end{equation}
By using \eqref{BA1}, \eqref{BA3}, \eqref{BA3:impr:B} and \eqref{BA1:impr:A3} again, we have
\begin{equation}\label{BA1:impr:A8}
|G_a^{\Xi}|\ls\ve_1^2\w{y}^{-3/2}\w{s-|y|}^{\ve_2-1}\w{s+|y|}^{8\ve_2-1/2}.
\end{equation}
Collecting \eqref{BA1:impr:A4}-\eqref{BA1:impr:A8} with \eqref{def:goodunknown}, \eqref{BA1}, \eqref{BA3}, \eqref{loc:impr:A}, 
\eqref{BA3:impr:B} and \eqref{BA1:impr:A3} yields \eqref{BA1:impr:A}.
\end{proof}

\begin{lemma}\label{lem:loc:B}
Under the assumptions of Theorem \ref{thm1}, let $u$ be the solution of \eqref{QWE} and suppose that \eqref{BA1}-\eqref{BA3} hold.
Then we have
\begin{equation}\label{loc:impr:B}
\sum_{|a|\le2N-41}\|\p^au\|_{L^2(\cK_R)}\ls(\ve+\ve_1^2)(1+t)^{-1}.
\end{equation}
\end{lemma}
\begin{proof}
It follows from \eqref{BA2:impr:B} and \eqref{BA1:impr:A} that
\begin{equation}\label{loc:impr:B1}
\sum_{|a|\le2N-37}|\bar\p Z^au|\ls\ve_1\w{x}^{-1/2}\w{t+|x|}^{10\ve_2-1}.
\end{equation}
Let $Z^a=\p_t^j$ with $j=|a|$ in \eqref{eqn:high}.
Applying \eqref{loc:dt:sharp} to \eqref{eqn:high} instead of \eqref{eqn:good} with $j\le2N-41$, $\mu=\nu=\ve_2/2$ and
utilizing \eqref{initial:data}, one has that for $R>1$,
\begin{equation}\label{loc:impr:B2}
\begin{split}
\w{t}\sum_{j\le2N-41}\|\p_t\p_t^ju\|_{L^\infty(\cK_R)}
\ls\ve+\sum_{|a|\le2N-41}\sum_{|b|\le2}\sup_{s\in[0,t]}\w{s}\|\p^b\Box Z^au(s)\|_{L^2(\cK_3)}\\
+\sum_{|a|\le2N-41}\sum_{|b|\le3}\sup_{(s,y)\in[0,t]\times\cK}\w{y}^{1/2}\cW_{1+\ve_2,1}(s,y)|Z^b\Box Z^au(s,y)|.
\end{split}
\end{equation}
Although we have chosen $Z^a=\p_t^j$, from Lemmas \ref{lem:eqn:high}-\ref{lem:null:structure}, \eqref{BA1:impr:A} and \eqref{loc:impr:B1},
it actually holds that for $Z\in\{\p,\Omega\}$,
\begin{equation}\label{loc:impr:B3}
\sum_{|a|\le2N-41}\sum_{|b|\le3}|Z^b\Box Z^au|\ls\ve_1^2\w{y}^{-1}\w{s-|y|}^{-1}\w{s+|y|}^{20\ve_2-1}.
\end{equation}
Substituting \eqref{loc:impr:B3} into \eqref{loc:impr:B2} with an argument as \eqref{loc:energy} yields \eqref{loc:impr:B}.
\end{proof}

\begin{lemma}\label{lem:BA3:C}
Under the assumptions of Theorem \ref{thm1}, let $u$ be the solution of \eqref{QWE} and suppose that \eqref{BA1}-\eqref{BA3} hold.
Then one has
\begin{equation}\label{BA3:impr:C}
\sum_{|a|\le2N-43}|Z^au|\ls(\ve+\ve_1^2)\w{t+|x|}^{\ve_2/2-1/2}\w{t-|x|}^{-0.4}.
\end{equation}
\end{lemma}
\begin{proof}
The proof of the estimate \eqref{BA3:impr:C} is similar to that of \eqref{BA3:impr:A} and \eqref{BA3:impr:B}.
In fact, applying \eqref{InLW:pw} to $\Box\tilde Z^au=\Box(\tilde Z^a-Z^a)u+\Box Z^au$ with $\mu=0.4$, $\nu=\ve_2$, $|a|\le2N-43$, \eqref{initial:data}, \eqref{loc:impr:A} and \eqref{loc:impr:B3} yields \eqref{BA3:impr:C}.
\end{proof}

\begin{lemma}[Improvement of \eqref{BA1}]\label{lem:NL:pw}
Under the assumptions of Theorem \ref{thm1}, let $u$ be the solution of \eqref{QWE} and suppose that \eqref{BA1}-\eqref{BA3} hold.
Then it holds that
\begin{equation}\label{BA1:impr:B}
\sum_{|a|\le2N-53}|\p Z^au|\ls(\ve+\ve_1^2)\w{x}^{-1/2}\w{t-|x|}^{-1}.
\end{equation}
\end{lemma}
\begin{proof}
Utilizing \eqref{InLW:dpw:div} to \eqref{eqn:good:div} for $|a|\le2N-53$ and $\mu=\nu=\ve_2$ yields
\begin{equation}\label{BA1:impr:B1}
\begin{split}
&\w{x}^{1/2}\w{t-|x|}|\p V_a|\ls\ve+\ve_1^2
+\sum_{i=1,2}\sup_{s\in[0,t]}\w{s}^{3/2+\ve_2}\|G_a^i(s)\|_{L^\infty(\cK_3)}\\
&\quad +\sup_{(s,y)\in[0,t]\times(\overline{\R^2\setminus\cK_2})}\w{y}^{1/2}\cW_{3/2+\ve_2,1}(s,y)|G_a(s,y)|\\
&\quad +\sum_{\alpha=0}^2\sup_{(s,y)\in[0,t]\times(\overline{\R^2\setminus\cK_2})}
\w{y}^{1/2}\cW_{1+2\ve_2,1}(s,y)|(G_a,G_a^\theta,G_a^\alpha)(s,y)|,
\end{split}
\end{equation}
where we have used \eqref{initial:data} and \eqref{loc:impr:B}.
Note that based on \eqref{BA1:impr:A} and \eqref{BA3:impr:C}, \eqref{BA1:impr:A7} and \eqref{BA1:impr:A8} can be improved as follows
\begin{equation}\label{BA1:impr:B2}
\begin{split}
|G_a|&\ls\ve_1^3\w{y}^{-1}\w{s}^{9\ve_2-5/2}+\ve_1^2\w{y}^{-5/2}\w{s}^{-1.8},\qquad |y|\le1+s/2,\\
|G_a^{\Xi}|&\ls\ve_1^2\w{y}^{-3/2}\w{s-|y|}^{-1.4}\w{s+|y|}^{11\ve_2-1/2}.
\end{split}
\end{equation}
Therefore, combining \eqref{BA1:impr:B1}-\eqref{BA1:impr:B2} with \eqref{def:goodunknown}, \eqref{BA1:impr:A}, \eqref{loc:impr:B} and \eqref{BA3:impr:C} yields \eqref{BA1:impr:B}.
\end{proof}

\begin{lemma}[Improvements of \eqref{BA2} and \eqref{BA3}]\label{lem:BA3:D}
Under the assumptions of Theorem \ref{thm1}, let $u$ be the solution of \eqref{QWE} and suppose that \eqref{BA1}-\eqref{BA3} hold.
Then we have
\begin{equation}\label{BA2:impr:C}
\sum_{|a|\le2N-53}|\bar\p Z^au|\ls(\ve+\ve_1^2)\w{x}^{-1/2}\w{t+|x|}^{\ve_2/2-1}
\end{equation}
and
\begin{equation}\label{BA3:impr:D}
\sum_{|a|\le2N-59}|Z^au|\ls(\ve+\ve_1^2)\w{t+|x|}^{\ve_2-1/2}\w{t-|x|}^{\ve_2-1/2}.
\end{equation}
\end{lemma}
\begin{proof}
The estimate \eqref{BA2:impr:C} can be obtained as for \eqref{BA2:impr:A} and \eqref{BA2:impr:B} with \eqref{loc:impr:B3}, \eqref{BA3:impr:C} and \eqref{BA1:impr:B}.
In terms of \eqref{BA1:impr:B} and \eqref{BA2:impr:C}, \eqref{loc:impr:B3} can be further improved to
\begin{equation}\label{BA3:impr:D1}
\sum_{|a|\le2N-59}\sum_{|b|\le5}|Z^b\Box Z^au|\ls\ve_1^2\w{y}^{-1}\w{s-|y|}^{-1}\w{s+|y|}^{\ve_2/2-1}.
\end{equation}
Applying \eqref{InLW:pw} to $\Box\tilde Z^au=\Box(\tilde Z^a-Z^a)u+\Box Z^au$ with $\mu=1/2-\ve_2$, $\nu=\ve_2/2$, \eqref{loc:impr:B} and \eqref{BA3:impr:D1} yields \eqref{BA3:impr:D}.
\end{proof}

\subsection{Proof of Theorem \ref{thm1}}
\begin{proof}[Proof of Theorem \ref{thm1}]
Collecting \eqref{BA1:impr:B}, \eqref{BA2:impr:C} and \eqref{BA3:impr:D} yields that there is $C_3\ge1$ such that
\begin{align*}
&\sum_{|a|\le2N-53}|\p Z^au|\le C_3(\ve+\ve_1^2)\w{x}^{-1/2}\w{t-|x|}^{-1},\\
&\sum_{|a|\le2N-53}|\bar\p Z^au|\le C_3(\ve+\ve_1^2)\w{x}^{-1/2}\w{t+|x|}^{\ve_2-1},\\
&\sum_{|a|\le2N-59}|Z^au|\le C_3(\ve+\ve_1^2)\w{t+|x|}^{\ve_2-1/2}\w{t-|x|}^{\ve_2-1/2}.
\end{align*}
Choosing $\ve_1=4C_3\ve$ and $\ve_0=\frac{1}{16C_3^2}$, then for $N\ge59$, \eqref{BA1}-\eqref{BA3} can be improved to
\begin{align*}
&\sum_{|a|\le N+1}|\p Z^au|\le\frac{\ve_1}{2}\w{x}^{-1/2}\w{t-|x|}^{-1},\\
&\sum_{|a|\le N+1}|\bar\p Z^au|\le\frac{\ve_1}{2}\w{x}^{-1/2}\w{t+|x|}^{\ve_2-1},\\
&\sum_{|a|\le N}|Z^au|\le\frac{\ve_1}{2}\w{t+|x|}^{\ve_2-1/2}\w{t-|x|}^{\ve_2-1/2}.
\end{align*}
This, together with the local existence of classical solution to the initial boundary value problem
of the hyperbolic equation, yields
that \eqref{QWE} has a global solution $u\in\bigcap\limits_{j=0}^{2N+1}C^{j}([0,\infty), H^{2N+1-j}(\cK))$.
Furthermore, \eqref{thm1:decay:a}-\eqref{thm1:decay:c} and \eqref{thm1:decay:LE} are obtained by \eqref{BA1:impr:B}, \eqref{BA2:impr:C}, \eqref{BA3:impr:D} and \eqref{loc:impr:B}, respectively.
\end{proof}

\vskip 0.2 true cm

{\bf \color{blue}{Conflict of Interest Statement:}}

\vskip 0.1 true cm

{\bf The authors declare that there is no conflict of interest in relation to this article.}

\vskip 0.2 true cm
{\bf \color{blue}{Data availability statement:}}

\vskip 0.1 true cm

{\bf  Data sharing is not applicable to this article as no data sets are generated
during the current study.}

\end{document}